\documentclass[11pt,b5paper,notitlepage]{article}
\usepackage[b5paper, margin={0.5in,0.65in}]{geometry}

\usepackage{amsmath,amscd,amssymb,amsthm,mathrsfs,amsfonts,layout,indentfirst,graphicx,caption,mathabx,   stmaryrd,appendix,calc,imakeidx,upgreek}
\usepackage{palatino}  
\usepackage{slashed} 
\usepackage{mathrsfs} 
\usepackage{extarrows} 
\usepackage{enumitem} 
\makeindex

\usepackage[nottoc]{tocbibind}   

\usepackage{lipsum}
\let\OLDthebibliography\thebibliography
\renewcommand\thebibliography[1]{
	\OLDthebibliography{#1}
	\setlength{\parskip}{0pt}
	\setlength{\itemsep}{2pt} 
}

\allowdisplaybreaks  
\usepackage{latexsym}
\usepackage{chngcntr}
\usepackage[colorlinks,linkcolor=blue,anchorcolor=blue,linktocpage]{hyperref}  
\hypersetup{citecolor=[rgb]{0,0.5,0}}
\counterwithin{figure}{section}

\theoremstyle{definition}
\newtheorem{df}{Definition}[section]

\newtheorem{rem}[df]{Remark}

\newtheorem{cv}[df]{Convention}

\theoremstyle{plain}
\newtheorem{thm}[df]{Theorem}

\newtheorem{pp}[df]{Proposition}
\newtheorem{co}[df]{Corollary}
\newtheorem{lm}[df]{Lemma}

\newtheorem{cond}{Condition}
\newtheorem{thmn}{Theorem}

\newcommand{\fk}{\mathfrak}
\newcommand{\mc}{\mathcal}
\newcommand{\wtd}{\widetilde}
\newcommand{\wht}{\widehat}

\newcommand{\ovl}{\overline}

\newcommand{\End}{\mathrm{End}} 
\newcommand{\id}{\mathbf{1}}
\newcommand{\Hom}{\mathrm{Hom}}

\newcommand{\ev}{\mathrm{ev}}
\newcommand{\coev}{\mathrm{coev}}

\newcommand{\Rep}{\mathrm{Rep}}
\newcommand{\Repu}{\mathrm{Rep}^{\mathrm u}}
\newcommand{\diag}{\mathrm{diag}}
\newcommand{\Dom}{\scr D}
\newcommand{\loc}{\mathrm{par}}

\newcommand{\Diffp}{\mathrm{Diff}^+}
\newcommand{\Diff}{\mathrm{Diff}}
\newcommand{\PSU}{\mathrm{PSU}(1,1)}
\newcommand{\UPSU}{\widetilde{\mathrm{PSU}}(1,1)}

\newcommand{\bk}[1]{\langle {#1}\rangle}
\newcommand{\GA}{\mathscr G_{\mathcal A}}
\newcommand{\GAV}{\mathscr G_{\mathcal A_V}}

\newcommand{\Vect}{\mathrm{Vec}}
\newcommand{\Vectc}{\mathrm{Vec}^{\mathbb C}}
\newcommand{\scr}{\mathscr}
\newcommand{\Jtd}{\widetilde{\mathcal J}}
\newcommand{\gk}{\mathfrak g}
\newcommand{\hk}{\mathfrak h}

\newcommand{\Repf}{\mathrm{Rep}^{\mathrm f}}
\newcommand{\im}{\mathbf{i}}

\newcommand{\RepA}{\mathrm{Rep}(\mc A)}
\newcommand{\RepfA}{\mathrm{Rep}^{\mathrm f}(\mc A)}
\newcommand{\RepV}{\mathrm{Rep}(V)}
\newcommand{\RepuV}{\mathrm{Rep}^{\mathrm u}(V)}

\newcommand{\mbb}{\mathbb}

\newcommand{\Obj}{\mathrm{Obj}}

\newcommand{\bpr}{{}^\backprime}
\newcommand{\pr}{\mathrm{pr}}
\newcommand{\Gr}{\mathrm{Gr}}
\newcommand{\rank}{\mathrm{rank}}

\numberwithin{equation}{section}

\title{Unbounded field operators in categorical extensions  of conformal nets}
\author{{\sc Bin Gui}
}
\date{}
\begin{document}\sloppy 
	\pagenumbering{arabic}
	
	\maketitle

\tableofcontents
	
\newpage

\begin{abstract}

In  rational conformal field theory, a main conjecture is that the (unitary) modular tensor category associated to a (unitary) vertex operator algebras (VOA) $V$ is equivalent to the one associated to the corresponding  conformal net $\mc A_V$ \cite{Kaw15,Kaw18}. In \cite{Gui21a}, we gave a systematic treatment of this conjecture by introducing the notion of \emph{categorical extensions of conformal nets}. One major difficulty encountered in that article is to prove the strong braiding of smeared intertwining operators of $V$. 

In this article, we develop a theory of unbounded operators in categorical extensions of conformal nets, and show that strong braiding follows from some other conditions which are much easier to verify. As an application, we prove the equivalence of unitary modular tensor categories of $V$ and $\mc A_V$ for the following examples: all WZW models, all even lattice VOAs, all parafermion VOAs with positive integer levels,  type $ADE$ discrete series $W$-algebras,  their tensor products, and their regular coset VOAs.  

In the case of WZW models, our work together with \cite{Fin96} solve the longstanding conjecture that given a simple Lie algebra with positive integer level, the unitary modular tensor categories associated to the corresponding affine Lie algebra, quantum group at certain roots of unity, and loop group conformal net are equivalent.
\end{abstract}

\section{Introduction}

Conformal nets and vertex operator algebras (VOAs) are two major mathematical axiomatizations of two dimensional chiral conformal field theories (CFTs). In the conformal net approach, chiral CFTs are formulated and studied under the framework of algebraic quantum field theory \cite{HK64,Haag96} which relies heavily on the methods of operator algebras and especially  Jones subfactor theory \cite{Jon83,Lon89,Lon90}. VOAs originate from the study of Moonshine conjecture and infinite dimensional Lie algebras \cite{FLM89,Bor92}, and are deeply related to modular forms \cite{Zhu96} and the moduli spaces of algebraic curves \cite{TUY89,FB04,DGT19a,DGT19b}. The reason for such  differences in the two approaches can be explained partially by the fact that the full-boundary CFTs extending conformal nets live on Minkowskian spacetimes, whereas those extending VOAs live on Euclidean spacetimes (e.g. compact Riemann surfaces with boundaries); see \cite{KR08} for a detailed discussion. Despite such differences, one has many similar and parallel results in the two approaches; see  \cite{Kaw15,Kaw18} for an overview.

Tensor categories (cf. for example \cite{BK01,EGNO15}), which relate rational CFTs with 3d topological quantum field theories (first observed in physics by Witten \cite{Wit89}) in a mathematically rigorous way \cite{Tur94}, play an important role in the representation theories of conformal nets \cite{FRS89,FRS92,KLM01} and VOAs \cite{Hua05a,NT05,Hua08b}. Moreover, many similar results in the two approaches can be formulated in terms of tensor categories. Therefore, it is desirable to establish a framework under which one can systematically relate conformal nets and VOAs and their representation tensor categories.

The first important progress on this problem  was made by Carpi-Kawahigashi-Longo-Weiner in \cite{CKLW18}. Using smeared vertex operators, the authors of \cite{CKLW18} gave a natural construction of a conformal net $\mc A_V$ from a sufficiently nice (i.e. strongly local) unitary simple VOA $V$, and proved that many familiar unitary VOAs satisfy these nice conditions. Their ideas were later generalized by the author to the intertwining operators of VOAs to prove the unitarity of the braided tensor category $\RepuV$ of unitary $V$-modules \cite{Gui19a,Gui19b,Gui19c}, and to prove the unitary equivalence of $\RepuV$ with the braided $C^*$-tensor category $\Repf(\mc A_V)$ of dualizable (i.e. finite-dimensional) $\mc A_V$-modules \cite{Gui21a}. As an alternative to smeared vertex operators, in \cite{Ten19a,Ten19b,Ten24}, Tener used Segal CFT \cite{Seg04,Ten17}  to relate $V$ and $\mc A_V$, and proved many similar results as well as new ones. Most remarkably, he gave in \cite{Ten24}  the first systematic and complete proof of the complete rationality of the conformal nets associated to unitary affine VOAs and type $ADE$ discrete series $W$-algebras.

\subsection{Equivalence of braided $C^*$-tensor categories}

In this article, we continue our study of the equivalence of braided  $C^*$-tensor categories $\RepuV\simeq\Repf(\mc A_V)$ initiated in \cite{Gui21a}. The techniques in this article allow us to finish proving the long-standing conjecture that unitary affine VOAs and their conformal nets have equivalent modular tensor categories. Moreover, our approach is general enough to be applied to many other examples  including lattice VOAs, parafermion VOAs, and type $ADE$ discrete series $W$-algebras.

One of the main unsolved problems in \cite{Gui21a} is proving the \emph{strong braiding of smeared intertwining operators}. To be more precise, we have shown in \cite{Gui21a}  that if $V$ has sufficiently many intertwining operators satisfying 

(a) polynomial energy bounds, 

(b) strong intertwining property,

(c) strong braiding,\\
then $\RepuV$ is equivalent to a braided $C^*$-tensor subcategory of $\Rep(\mc A_V)$. Although these three conditions are expected to hold for all rational unitary chiral CFTs, the condition of strong braiding is much more difficult to prove than the other two. In \cite{Gui21a}, the only method for proving strong braiding is to show that sufficiently many intertwining operators satisfy linear energy bounds (see section \ref{lb69} for the precise definition), which is NOT expected to hold even for all unitary affine VOAs. As a consequence, in \cite{Gui21a} we were not able to prove  the equivalence of tensor categories for type $BDEF$ affine VOAs. This difficulty is completely resolved in this article: we show that strong braiding follows from polynomial energy bounds and the strong intertwining property, i.e. that (a) and (b) imply (c). See theorem \ref{lb73} for more details. Consequently, we are able to prove $\Repu(V)\simeq\Repf(\mc A_V)$ for a huge number of examples including all WZW models and their regular cosets. The more precise statement is the following:

\begin{thmn}\label{lb94}
	Let $V$ be a (finite) tensor product of the following unitary VOAs:
	\begin{itemize}
		\item Unitary affine VOAs $L_l(\gk)$.
		\item Discrete series $W$-algebras $\mc W_{l'}(\gk)$ of type $ADE$.
		\item Unitary parafermion VOAs $K_l(\gk)$.
		\item Lattice VOAs $V_{\Lambda}$.
	\end{itemize}
Then  the following are true:
	
	\begin{enumerate}[label=(\alph*)]
		\item $V$ is completely unitary (cf. section \ref{lb87}), and  $\RepuV$ is therefore a unitary modular tensor category. 
		\item Any $W_i\in\Obj(\RepuV)$ is strongly integrable, the $*$-functor $\fk F:\RepuV\rightarrow\Repf(\mc A_V)$ (defined in section \ref{lb87}) implements an isomorphism of braided $C^*$-tensor categories, and the conclusion in corollary \ref{lb65} (about strong braiding) holds verbatim for $V$.
		\item $\mc A_V$ is completely rational.
	\end{enumerate}

	Moreover if $V$ is a unitary sub-VOA of $U$ which   is also a tensor product of the above unitary VOAs, and if  $V^c$ (the commutant of $V$ in $U$) is regular, then $V^c$ satisfies  (a) (b) (c).
\end{thmn}

\subsection{Strong braiding of smeared intertwining operators}

The strong braiding of intertwining operators is  a generalization of the strong locality of vertex operators \cite{CKLW18} and the strong intertwining property of intertwining operators \cite{Gui19b}. We now explain the meaning of these notions.

Let us assume that $V$ is a unitary simple VOA.  Let $Y(v,z)$ be the vertex operator associated to the vector $v\in V$. For any $f\in C^\infty(\mbb S^1)$ one can define the smeared vertex operator
\begin{align}
Y(v,f)=\oint_{\mbb S^1} Y(v,z)f(z)\frac{dz}{2\im\pi}=\int_0^{2\pi}Y(v,e^{\im\theta})f(e^{\im\theta})\cdot\frac{e^{\im\theta}}{2\pi}d\theta
\end{align}
which is indeed a closable/closed operator on $\mc H_0$, the Hilbert space completion of $V$. (Note that since $V$ is unitary, as a vector space $V$ has an inner product.) Moreover, $V$ is inside the domain of any (finite) product of smeared vertex operators. \footnote{This fact, together with the weak locality mentioned below, relies only on the fact that the $n$-point correlation functions formed by products of vertex operators are rational functions. I thank Sebastiano Carpi for pointing out this to me.}

Smeared vertex operators always satisfy weak locality, which means that if $I,J$ are disjoint (non-dense and non-empty) open intervals of the unit circle $\mbb S^1$, and if $u,v\in V,f\in C_c^\infty(I),g\in C_c^\infty(g)$, then $[Y(u,f),Y(v,g)]=0$ when acting on any vector inside $V$. However, to construct a conformal net $\mc A_V$ from $V$, one needs the \textbf{strong locality} (of vertex operators), which means that for the above $u,v,f,g$, the von Neumann algebras generated by $Y(u,f)$ and by $Y(v,g)$ commute. (In this case, we say that the closable/closed operators $Y(u,f)$ and $Y(v,g)$ commute strongly. Weak commutativity does not always imply strong commutativity by the celebrated example of Nelson \cite{Nel59}.) 

One can generalize the notion of strong locality to smeared intertwining operators. Assume $V$ is regular (equivalently, rational and $C_2$-cofinite \cite{ABD04}), which corresponds to rational chiral CFTs in physics. We also assume that $V$ is strongly unitary, i.e. any $V$-module admits a unitary structure. An intertwining operator of $V$ is a ``charged field" of $V$. More precisely, if $W_i,W_j,W_k$ are unitary $V$-modules, then a type $W_k\choose W_iW_j$ (written as $k\choose i~j$ for short) intertwining operator $\mc Y$ associates linearly to each vector $w^{(i)}\in W_i$  a multivalued function $\mc Y(w^{(i)},z)$ on $z\in\mbb C^\times=\mbb C-\{0\}$ whose values are linear maps from $W_j$ to the ``algebraic completion" of $W_k$.(See \cite{FHL93} for the rigorous definition.)  We say that $W_i$ is the charge space of $\mc Y$.  Let $\mc H_i,\mc H_j,\mc H_k,\dots$ be the Hilbert space completions of $W_i,W_j,W_k,\dots$. Then one can define the smeared intertwining operator in a similar way as smeared vertex operator, except that the smeared intertwining operator depends not only on the smooth function $f\in C_c^\infty(I)$ supported in an interval $I$, but also on a choice of (continuous) argument function $\arg_I$ of $I$ (equivalently, a branch of $I$ in the universal cover of $\mbb S^1$). We call $\wtd f=(f,\arg_I)$ an arg-valued function. Then for any $w^{(i)}\in W_i$ and for any $\wtd f$, we have a closable/closed smeared intertwining operator $\mc Y(w^{(i)},\wtd f)$ mapping from (a dense subspace of) $\mc H_j$ to $\mc H_k$.  Note that a vertex operator is a special kind of intertwining operator. Then the  smeared intertwining operator $\mc Y(w^{(i)},\wtd f)$ and the smeared vertex operators $Y_j(v,g),Y_k(v,g)$ associated to the vertex operators $Y_j,Y_k$ of the unitary $V$-modules $W_j,W_k$ satisfy the weak locality: if the supports of $\wtd f$ and $g$ are disjoint, then the weak intertwining property
\begin{align*}
\mc Y(w^{(i)},f)Y_j(v,g)=Y_k(v,g)\mc Y(w^{(i)},f)
\end{align*}
holds when acting on $W_j$. If for any $\wtd f,g,v$ this commutative relation holds ``strongly" in the sense of the commutativity of the associated von Neumann algebras, then we say that the charged field $\mc Y(w^{(i)},z)$ satisfies the \textbf{strong intertwining property}. (See section \ref{lb67} for details.) More generally, one can study the weak and strong locality of two smeared intertwining operators, which are called the weak and \textbf{strong braiding} of intertwining operators in our paper. (See section \ref{lb68}). Moreover,  weak braiding is always true and not hard to show,\footnote{We have only proved this when the intertwining operators are energy-bounded (see theorem \ref{lb60} or \cite{Gui21a} theorem 4.8). Although all intertwining operators of unitary regular VOAs are expected to be energy-bounded, so far we do not have a general proof of this fact. However, we will prove the weak braiding in a future work without using the energy bounds condition.} while strong braiding requires more difficult techniques to be proved.

By saying that $\mc Y(w^{(i)},z)$ is \textbf{energy-bounded}, we mean roughly that there exists $n\in\mbb N$ such that the smeared intertwining operator $\mc Y(w^{(i)},\wtd f)$ is bounded by $L_0^n$ for each $\wtd f$. (See section \ref{lb69} for details.) If $\mc Y(w^{{(i)}},z)$ is energy-bounded for each vector $w^{(i)}$, we say that $\mc Y$ is an energy-bounded intertwining operator. We say that $V$ is strongly energy-bounded if the vertex operators associated to all unitary $V$-modules  are energy-bounded. We say that $V$ is strongly unitary if each $V$-module admits a unitary structure. The following main theorem of this article implies theorem \ref{lb94}.

\begin{thmn}[Main theorem]\label{lb73}	Let $V$ satisfy condition \ref{cd2} of section \ref{lb87}, which is equivalent to the following: Let $V$ be a strongly unitary and strongly energy-bounded simple regular VOA. Assume that $V$ is strongly local. Assume that $\mc F^V$ is a set of irreducible unitary $V$-modules tensor-generating the category $\RepV$ of $V$-modules. Suppose that for each $W_i\in\mc F$ there is a non-zero quasi-primary vector $w^{(i)}_0\in W_i$ such that for any unitary intertwining operator $\mc Y$ with charge space $W_i$, $\mc Y(w^{(i)}_0,z)$ is energy-bounded and satisfies the strong intertwining property. Then the following are true.
	
(1) The category $\RepuV$ of unitary $V$-modules is a unitary modular tensor category (Thm. \ref{lb51}).
	
(2) There is a (naturally defined) fully faithful $*$-functor $\fk F:\RepuV\rightarrow\Rep(\mc A_V)$ (Thm. \ref{lb55}) such that $\RepuV$ is equivalent to a braided $C^*$-tensor subcategory of $\Rep(\mc A_V)$ under $\fk F$ (Cor. \ref{lb71}).
	
(3) Energy-bounded smeared intertwining operators satisfy strong braiding (Cor. \ref{lb65}). In particular, they satisfy the strong intertwining property (Rem. \ref{lb72}).
\end{thmn}

Roughly speaking, the assumption on intertwining operators is that $V$ has sufficiently many intertwining operators that are energy-bounded and satisfy the strong intertwining property. Part (1) of this theorem is not new and follows essentially from \cite{Gui19a,Gui19b}. Part (2) claims the equivalence of braided $C^*$-tensor structures. (Note that the unitary ribbon structure is uniquely determined by the  braided $C^*$-tensor structure.) Part (3) claims the strong braiding, which follows immediately from the method used to prove (2).

\subsection{Main idea of the proof}

Let us explain the main idea of proving parts (2) and (3) of theorem \ref{lb73}. We first compare this theorem with the main result of \cite{Gui21a}, which says roughly that the assumptions of theorem \ref{lb73}, together with (3), imply (2). Thus, the main improvement of theorem \ref{lb73} is that part (3) (strong braiding) is now a consequence but no  longer an assumption. 

The main idea in \cite{Gui21a} is the following. For the conformal net $\mc A_V$ and any representations $\mc H_i,\mc H_j$, for suitable $\xi\in\mc H_i$ and $\eta\in\mc H_j$ ``localized" in disjoint arg-valued intervals $\wtd I,\wtd J$, one has fusion product $\xi\boxtimes\eta$, written as $L(\xi,\wtd I)\eta$ and also as $R(\eta,\wtd J)\xi$. $L(\xi,\wtd I)$ and $R(\eta,\wtd J)$ describe the left and right actions of $\xi,\eta$ on $\mc H_j$ and $\mc H_i$ respectively. Moreover, for the vectors $\xi,\eta$ considered in \cite{Gui21a}, the left and the right (linear) operators $L(\xi,\wtd I),R(\eta,\wtd J)$ are \emph{bounded} operators and satisfy braid relations which capture the braided $C^*$-tensor structure of $\Rep(\mc A_V)$. On the VOA side, the (smeared) intertwining operators of $V$ also record the braided $C^*$-tensor structure of $\Repu(V)$. To show that $\Repu(V)$ and $\Rep(\mc A_V)$ are compatible, we relate the smeared intertwining operators of $V$ with the left and right operators. Unfortunately, the smeared intertwining operators are often not bounded. \cite{Gui21a} addresses this issue by taking polar decompositions of smeared intertwining operators and relating the partial isometry parts (the phases) with the bounded $L$ and $R$ operators. To establish such a relation, one has to first prove that the phases of smeared intertwining operators also satisfy nice braid relations, which relies on  the strong braiding of intertwining operators.

Different from the above approach which first takes polar decompositions and then relates bounded linear operators,  in this article we relate (possibly) unbounded closed operators directly. To be more precise, we loosen the conditions on $\xi,\eta$ which yield unbounded closed left and right operators $\scr L(\xi,\wtd I),\scr R(\eta,\wtd J)$. These $\scr L$ and $\scr R$ operators satisfy the strong intertwining property and (more generally) the strong braiding, and the phases of these operators give (bounded) $L$ and $R$ operators. Using the polynomial energy bounds and the strong intertwining properties of $V$-intertwining operators, one can relate the smeared intertwining operators of $V$ with the $\scr L$ and $\scr R$ operators of $\mc A_V$ since  weak braiding is satisfied on both sides. Therefore, as $\scr L$ and $\scr R$ satisfy strong braiding, so do the smeared intertwining operators  of $V$.

Note that part (3) of theorem \ref{lb73} also claims that if sufficiently many intertwining operators of $V$ are energy-bounded and satisfy the strong intertwining operators, then the strong braiding holds not only for these intertwining operators, but also for any other ones satisfying only polynomial energy bounds. In particular, the strong intertwining property will be a consequence of polynomial energy bounds. This result generalizes \cite{CKLW18} theorem 8.1, which says that all the vertex operators of $V$ satisfy strong locality if sufficiently many quasi-primary fields satisfy polynomial energy bounds and strong locality. Similar to that theorem, the proof of our result relies heavily on the Bisognano-Wichmann theorem. See section \ref{lb48}. This result can be applied to the following situation: Let $V$ and $V'$ satisfy the assumptions of theorem \ref{lb73}. Assume that all the intertwining operators of $V$ and $V'$ are energy-bounded and satisfy the strong intertwining property. Then any intertwining operator of $V\otimes V'$, which is clearly energy-bounded since it is a sum of tensor products of intertwining operators of $V$ and $V'$, also satisfies  strong braiding and (in particular) the strong intertwining property. (See theorem \ref{lb83} for the precise statement.)

\subsection{Outline of the paper}

This article is organized as follows. In chapter 2, we prove theorem \ref{lb73} under the general framework of operator algebras and algebraic quantum field theories. (In particular, theorem \ref{lb34} and corollary \ref{lb41} prove theorem \ref{lb73}-(2) and  (3) respectively.) We are able to do so since the main techniques of proving this theorem are operator algebraic. We hope that people with little background in VOA can still understand the main ideas of the proofs.

In section 2.1 we review the basic facts about the  (bounded) $L$ and $R$ operators. Then, in section 2.2, we define  the unbounded closed $\scr L$ and $\scr R$ operators, and prove that they satisfy many important properties including the strong braiding. The symbol $\mc H_i^\pr(I)$ denotes the subspace  of all vectors in $\mc H_i$ with closable/closed $\scr L$ and $\scr R$ operators "localized" in the interval $I$. To relate $\scr L,\scr R$ with the smeared intertwining operators, one also needs to calculate the $4$-point correlation functions defined by $\scr L,\scr R$ as in proposition \ref{lb26}. For that purpose, we consider in section 2.3 a nice subspace $\mc H_i^w(I)\subset \mc H_i^\pr(I)$ and compute the $4$-point functions for the vectors in this kind of subspaces.

In section 2.4, we introduce the notion of a weak categorical extension, which consists of a list of axioms satisfied by the $\mc L$ and $\mc R$ operators. These $\mc L$ and $\mc R$ operators are the abstraction of \emph{products of} smeared intertwining operators satisfying the polynomial energy bounds and the strong intertwining property. Taking products  of smeared intertwining operators is necessary for  constructing enough operators that could be related to $\scr L$ and $\scr R$, and for extending the left and right actions from $\mc F^V$ (see the statement of theorem \ref{lb73}) to all the objects of $\Repu(V)$. The process of taking such products is described, in section 2.7 and especially in the proof of theorem \ref{lb40}, as the construction of a weak categorical extension from a weak categorical \emph{local} extension, in which the $\mc L$ and $\mc R$ operators describe single smeared intertwining operators. Note that in both sections, polynomial energy bounds are not explicitly mentioned in the axioms; two crucial consequences of energy bounds, the smoothness and the localizability, are required instead. Section 2.7 can be read immediately after section 2.4.

Section 2.5 is the climax of this article: we show that the existence of weak categorical (local) extensions implies the equivalence of braided $C^*$-tensor categories and the strong braiding. This proves theorem \ref{lb73}-(2) and part of (3). To finish proving (3), we need to show that polynomial energy bounds imply  the strong intertwining property and (hence) strong braiding. This is discussed in section 2.6 where the weak left and right operators $A,B$ are the abstraction of energy-bounded smeared intertwining operators.

In chapter 3 we prove theorem \ref{lb73} in the (original) language of VOA by verifying that the smeared intertwining operators satisfy the requirements  in chapter 2. A large part of this work has been done in \cite{Gui19a} and \cite{Gui21a} chapter 4, and is reviewed here for the readers' convenience. We remark that theorem \ref{lb43} and proposition \ref{lb47} are somewhat special cases of the main results of chapter 2. We give independent proofs of these two results, which contain many key ideas in  chapter 2. We hope that these  proofs are helpful for the readers to understand the arguments in chapter 2. Note that theorem \ref{lb43} is exactly \cite{CKLW18} theorem 8.1, which states that $V$ is strongly local if sufficiently many quasi-primary fields are. Our proof  differs from that of \cite{CKLW18} in that we use proposition \ref{lb38} to prove $Y(v,f)\Omega\in\Dom(S_I)$ for any $v\in V$, while \cite{CKLW18} uses the theory of real Hilbert subspaces developed in \cite{Lon08}. In section 3.6, we prove that tensor products or cosets of VOAs satisfying condition \ref{cd2} of section \ref{lb87} also satisfy condition \ref{cd2}. Here, condition \ref{cd2} is equivalent to the assumptions of theorem \ref{lb73}. With the help of these methods, we prove theorem \ref{lb94} by checking that the examples in this theorem satisfy condition \ref{cd2}. We also prove that all the intertwining operators of type $ADE$ unitary affine VOAs, discrete series $W$-algebras, and unitary parafermion VOAs are energy-bounded and satisfy strong braiding.

\subsubsection*{Acknowledgment}

I would like to thank Sebastiano Carpi, Yoh Tanimoto, James Tener, and Mih\'aly Weiner for helpful discussions. I also thank the referees for their valuable comments. I am grateful to Zhipeng Yang for his warm hospitality and the support of Yunnan Key Laboratory of Modern Analytical Mathematics and Applications. Special thanks go to my friends Jing An, Yiwang Chen, Ning Guo, Siqi He, Xuelun Hou, Yuhang Liu, Huadi Qu, Shi Wang, Mingchen Xia, Zetian Yan, Zhipeng Yang, Bingyu Zhang, Hong-Wei Zhang, Yu Zhao, Fan Zheng for their long-term support since the pandemic.

This work was supported by NSFC Grant 12401159.

\section{General theory}

\subsection{Categorical extensions of conformal nets}\label{lb29}

In this article, we let $\mbb N=\{0,1,2,3,\dots\}$ and $\mbb Z_+=\{1,2,3,\dots\}$. Set $\im=\sqrt{-1}$. If $\mc H,\mc K$ are Hilbert spaces, we say that  $A:\mc H\rightarrow\mc K$ is an unbounded operator on $\mc H$ if $A$ is a linear map from a subspace of $\mc H$ to $\mc K$. We let $\Dom(A)$ denote the domain of $A$. If $\Dom(A)$ is dense in $\mc H$ and $A$ is closable, we let $\ovl A$ be its closure. We say that $A:\mc H\rightarrow\mc K$ is bounded if $A$ is a continuous unbounded operator which is  defined everywhere (i.e. $\Dom(A)=\mc H$). A continuous unbounded operator is densely defined but not necessarily everywhere defined.

Let $\mc J$ be the set of non-empty non-dense open intervals in the unit circle $\mbb S^1$. If $I\in\mc J$, then $I'$ denotes the interior of the complement of $I$, which is also an element in $\mc J$. Let $\Diffp(\mbb S^1)$ be the group of orientation-preserving diffeomorphisms of $\mbb S^1$, which contains the subgroup $\PSU$ of M\"obius transforms of $\mbb S^1$. If $I\in\mc J$, we let $\Diff_I(\mbb S^1)$ be the subgroup of all $g\in\Diffp(\mbb S^1)$ such that $gx=x$ whenever $x\in I'$. Throughout this article, we let $\mc A$ be an (irreducible) conformal net. This means that each $I\in\mathcal J$ is associated with a von Neumann algebra $\mathcal A(I)$ acting on a fixed separable Hilbert space $\mathcal H_0$, such that the following conditions hold:\\
(a) (Isotony) If $I_1\subset I_2\in\mathcal J$, then $\mathcal A(I_1)$ is a von Neumann subalgebra of $\mathcal A(I_2)$.\\
(b) (Locality) If $I_1,I_2\in\mathcal J$ are disjoint, then $\mathcal A(I_1)$ and $\mathcal A(I_2)$ commute.\\
(c) (Conformal covariance) We have a strongly continuous projective unitary representation $U$ of $\Diffp(\mbb S^1)$ on $\mathcal H_0$ which restricts to a strongly continuous unitary representation of $\PSU$, such that for any $g\in \Diffp(\mbb S^1),I\in\mathcal J$, and any representing element $V\in\mathcal U(\mathcal H_0)$ of $U(g)$,
\begin{align*}
V\mathcal A(I)V^*=\mathcal A(gI).
\end{align*}
Moreover, if $g\in\Diff_I(S^1)$, then $V\in\mc A(I)$.\\
(d) (Positivity of energy) The action of the rotation subgroup $\varrho$ of $\PSU$ on $\mc H_0$ has positive generator.\\
(e) There exists a  unique (up to scalar) unit vector $\Omega\in\mathcal H_0$ fixed by $\PSU$. ($\Omega$ is called the vacuum vector.) Moreover, $\Omega$ is  cyclic under the action of $\bigvee_{I\in\mathcal J}\mathcal M(I)$ (the von Neumann algebra generated by all $\mathcal M(I)$).

$\mathcal A$ satisfies the following well-known properties (cf. for example \cite{GL96} and the reference therein):\\
(1) (Additivity) $\mathcal A(I)=\bigvee_\alpha\mathcal A(I_\alpha)$ if $\{I_\alpha\}$ is a set of open intervals whose union is $I$.\\
(2) (Haag duality) $\mathcal A(I)'=\mathcal A(I')$.\\ 
(3) (Reeh-Schlieder theorem) $\mathcal A(I)\Omega$ is dense in $\mathcal H_0$ for any $I\in\mathcal J$\\
(4) For each $I\in\mathcal J$, $\mathcal A(I)$ is a type III factor.

A pair $(\pi_i,\mc H_i)$ (or $\mc H_i$ for short) is called a (normal) representation of $\mc A$ (or $\mc A$-module) if $\mc H_i$ is a separable Hilbert space, and for each $I\in\mc J$, $\pi_{i,I}$ is a (normal $*$-) representation of $\mc A(I)$ on $\mc H_i$ such that $\pi_{i,I}|_{\mc A(I_0)}=\pi_{i,I_0}$ when $I_0\subset I$. If $x\in\mc A(I)$, we sometimes write $\pi_{i,I}(x)$ as $x$ for brevity. By \cite{Hen19}, any $\mc A$-module $\mc H_i$ is \textbf{conformal covariant}, which means the following: Let $\scr G$ be the universal cover of $\Diffp(\mbb S^1)$, and let $\GA$ be the central extension of $\scr G$ defined by 
\begin{align}
\GA=\{(g,V)\in\mathscr G\times \mathcal U(\mathcal H_0)| V \textrm{ is a representing element of } U(g) \}.\label{eq35}
\end{align}
The topology of $\GA$ inherits from that of $\scr G\times\mc U(\mc H_0)$. So we have exact sequence
\begin{align}
1\rightarrow \mbb T\rightarrow\GA\rightarrow\scr G\rightarrow 1
\end{align}
where $\mbb T$ is the subgroup of $\scr G\times\mc U(\mc H_0)$ consisting of all $(1,\lambda\id_{\mc H_0})$ where $\lambda\in\mbb C$ and $|\lambda|=1$. Then there exists a (unique) strongly continuous unitary representation $U_i$ of $\GA$ on $\mc H_i$ such that for any $I\in\mc J$ and $g\in\GA(I)$,
\begin{align}
U_i(g)=\pi_i(U(g)).\label{eq34}
\end{align}
As a consequence, for any $x\in\mc A(I)$ and $g\in\GA$, one has
\begin{align}
U_i(g)\pi_{i,I}(x)U_i(g)^*=\pi_{i,gI}(U(g)xU(g)^*).
\end{align}
See \cite{Gui21a} for more details. We  write $U_i(g)$ as $g$ when no confusion arises. Homomorphisms of $\mc A$-modules intertwine the representations of $\GA$.

Let $\RepA$ be the $C^*$-category of $\mc A$-modules whose objects are denoted by $\mc H_i,\mc H_j,\mc H_k,\dots$. The vector space of (bounded) homomorphisms between any $\Rep A$-modules  $\mc H_i,\mc H_j$ is written as $\Hom_{\mc A}(\mc H_i,\mc H_j)$. Then one can equip $\RepA$ with a structure of braided $C^*$-tensor category via Doplicher-Haag-Roberts (DHR) superselection theory \cite{FRS89,FRS92} or, equivalently, via Connes fusion \cite{Gui21a}. The unit of $\RepA$ is $\mc H_0$. We write the tensor (fusion) product of two $\mc A$-modules $\mc H_i,\mc H_j$ as $\mc H_i\boxtimes\mc H_j$. We assume without loss of generality that $\RepA$ is strict, which means that we will not distinguish between $\mc H_0,\mc H_0\boxtimes\mc H_i,\mc H_i\boxtimes\mc H_0$, or $(\mc H_i\boxtimes \mc H_j)\boxtimes\mc H_k$ and $\mc H_i\boxtimes(\mc H_j\boxtimes\mc H_k)$ (abbreviated to $\mc H_i\boxtimes\mc H_j\boxtimes \mc H_k$). In the following, we review the definition and the basic properties of closed vector-labeled categorical extensions of $\mc A$  introduced in \cite{Gui21a}.

To begin with, if $\mc H_i,\mc H_j$ are $\mc A$-modules and $I\in\mc J$, then $\Hom_{\mc A(I')}(\mc H_i,\mc H_j)$ denotes the vector space of bounded linear operators $T:\mc H_i\rightarrow\mc H_j$ such that $T\pi_{i,I'}(x)=\pi_{j,I'}(x)T$ for any $x\in\mc A(I')$. We then define $\mc H_i(I)=\Hom_{\mc A(I')}(\mc H_0,\mc H_i)\Omega$, which is a dense subspace of $\mc H_i$. Note that $I\subset J$ implies $\mc H_i(I)\subset\mc H_i(J)$. Moreover, if $G\in\Hom_{\mc A}(\mc H_i,\mc H_j)$, then $G\mc H_i(I)\subset\mc H_j(I)$.

If $I\in\mc J$, an arg-function of $I$ is, by definition, a continuous function $\arg_I:I\rightarrow\mbb R$ such that for any $e^{\im t}\in I$, $\arg_I(e^{\im t})-t\in 2\pi\mbb Z$. $\wtd I=(I,\arg_I)$ is called an \textbf{arg-valued interval}. Equivalently, $\wtd I$ is a branch of $I$ in the universal cover of $\mbb S^1$. We let $\wtd{\mc J}$ be the set of arg-valued intervals. If $\wtd I=(I,\arg_I)$ and $\wtd J=(J,\arg_J)$ are in $\Jtd$, we say that $\wtd I$ and $\wtd J$ are disjoint if $I$ and $J$ are so. Suppose moreover that for any $z\in I,\zeta\in J$ we have $\arg_J(\zeta)<\arg_I(z)<\arg_J(\zeta)+2\pi$, then we say that $\wtd I$ is \textbf{anticlockwise} to $\wtd J$ (equivalently, $\wtd J$ is \textbf{clockwise} to $\wtd I$). We write $\wtd I\subset\wtd J$ if $I\subset J$ and $\arg_J|_I=\arg_I$. Given $\wtd I\in\Jtd$, we also define $\wtd I'=(I',\arg_{I'})\in\Jtd$ such that $\wtd I$ is anticlockwise to $\wtd I'$. We say that $\wtd I'$ is the \textbf{clockwise complement} of $\wtd I$. We let $\wtd{\mbb S^1_+}\in\Jtd$ where $\mbb S^1_+=\{a+\im b\in\mbb S^1:b>0 \}$ is the upper semi-circle, and $\arg_{\mbb S^1_+}$ takes values in $(0,\pi)$. We set $\wtd{\mbb S^1_-}$ to be the clockwise complement of $\wtd{\mbb S^1_+}$.

\begin{df}
	A closed and  vector-labeled \textbf{categorical extension} $\scr E=(\mc A,\RepA,\boxtimes,\mc H)$ of $\mc A$ associates, to any  $\mc H_i,\mc H_k\in\Obj(\RepA)$ and any $\wtd I\in\Jtd,\fk \xi\in\mc H_i(I)$, bounded linear operators
	\begin{gather*}
	L(\xi,\wtd I)\in\Hom_{\mc A(I')}(\mc H_k,\mc H_i\boxtimes\mc H_k),\\
	R(\xi,\wtd I)\in\Hom_{\mc A(I')}(\mc H_k,\mc H_k\boxtimes\mc H_i),
	\end{gather*}
	such that the following conditions are satisfied:\\
	(a) (Isotony) If $\wtd I_1\subset\wtd I_2\in\Jtd$, and $\xi\in\mc H_i(I_1)$, then $L(\xi,\wtd I_1)=L(\xi,\wtd I_2)$, $R(\xi,\wtd I_1)=R(\xi,\wtd I_2)$ when acting on any  $\mc H_k\in\Obj(\RepA)$.\\
	(b) (Naturality) If $\mc H_{i},\mc H_k,\mc H_{k'}\in\Obj(\RepA)$, $F\in\Hom_{\mc A}(\mc H_k,\mc H_{k'})$,  the following diagrams commute for any $\wtd I\in\Jtd,\xi\in\mc H_i(I)$.
	\begin{gather}
	\begin{CD}
	\mc H_k @>F>> \mc H_{k'}\\
	@V L(\xi,\wtd I)  VV @V L(\xi,\wtd I)  VV\\
	\mc H_i\boxtimes\mc H_k @> \id_i\boxtimes F>> \mc H_i\boxtimes\mc H_{k'}
	\end{CD}\qquad\qquad
	\begin{CD}
	\mc H_k @> R(\xi,\wtd I)  >> \mc H_k\boxtimes\mc H_i\\
	@V F VV @V F\boxtimes\id_i  VV\\
	\mc H_{k'} @>R(\xi,\wtd I) >> \mc H_{k'}\boxtimes\mc H_i
	\end{CD}.
	\end{gather}
	(c) (State-field correspondence\footnote{For general (i.e., non-necessarily closed or vector-labeled) categorical extensions, this axiom is replaced by the neutrality and the Reeh-Schlieder property; see \cite{Gui21a} section 3.1.}) For any $\mc H_i\in\Obj(\RepA)$, under the identifications $\mc H_i=\mc H_i\boxtimes\mc H_0=\mc H_0\boxtimes\mc H_i$, the relations
	\begin{align}
	L(\xi,\wtd I)\Omega=R(\xi,\wtd I)\Omega=\xi
	\end{align}
	hold for any $\wtd I\in\Jtd,\xi\in\mc H_i(I)$. It follows immediately that when acting on $\mc H_0$, $L(\xi,\wtd I)$ equals $R(\xi,\wtd I)$ and is independent of $\arg_I$.\\
	(d) (Density of fusion products) If $\mc H_i,\mc H_k\in\Obj(\RepA),\wtd I\in\Jtd$, then the set $L(\mc H_i(I),\wtd I)\mc H_k$ spans a dense subspace of $\mc H_i\boxtimes\mc H_k$, and $R(\mc H_i(I),\wtd I)\mc H_k$ spans a dense subspace of $\mc H_k\boxtimes\mc H_i$.\\
	(e) (Locality) For any $\mc H_k\in\Obj(\RepA)$, disjoint $\wtd I,\wtd J\in\Jtd$ with $\wtd I$ anticlockwise to $\wtd J$, and any $\xi\in\mc H_i(I),\eta\in\mc H_j(J)$, the following diagram \eqref{eq5}  commutes adjointly.
	\begin{align}\label{eq5}
	\begin{CD}
	\mc H_k @> \quad R(\eta,\wtd J)\quad   >> \mc H_k\boxtimes\mc H_j\\
	@V L(\xi,\wtd I)   V  V @V L(\xi,\wtd I) VV\\
	\mc H_i\boxtimes\mc H_k @> \quad R(\eta,\wtd J) \quad  >> \mc H_i\boxtimes\mc H_k\boxtimes\mc H_j
	\end{CD}
	\end{align}
	Here, the \textbf{adjoint commutativity} of diagram \eqref{eq5} means that $R(\eta,\wtd J)L(\xi,\wtd I)=L(\xi,\wtd I)R(\eta,\wtd J)$ when acting on $\mc H_k$, and $R(\eta,\wtd J)L(\xi,\wtd I)^*=L(\xi,\wtd I)^*R(\eta,\wtd J)$ when acting on $\mc H_i\boxtimes\mc H_k$. (Cf. \cite{Gui21a} section 3.1.)\\
	(f) (Braiding) There is a unitary linear map $\mbb B_{i,j}:\mc H_i\boxtimes\mc H_j\rightarrow\mc H_j\boxtimes \mc H_i$  for any $\mc H_i,\mc H_j\in\Obj(\RepA)$, such that  
	\begin{align}
	\mbb B_{i,j} L(\xi,\wtd I)\eta=R(\xi,\wtd I)\eta
	\end{align}
	whenever $\wtd I\in\Jtd,\xi\in\mc H_i(I)$, $\eta\in\mc H_j$.
\end{df}

Note that $\mbb B_{i,j}$ is unique by the density of fusion products. Moreover, $\mbb B_{i,j}$ commutes with the actions of $\mc A$, and  is the same as the braid operator of $\RepA$; see \cite{Gui21a} sections 3.2, 3.3. The existence of $\scr E$ was also proved in \cite{Gui21a} sections 3.2. A sketch of the construction can be found in \cite{Gui21b} section A. 

Note that the categorical extensions of $\mc A$ are unique up to unitary isomorphisms, cf. Thm. 3.10 and 3.12 of \cite{Gui21a}. In Sec. \ref{lb33} of this paper, we will extend the uniqueness to weak categorical extensions. As mentioned in the Introduction of \cite{Gui21a} and described more rigorously in \cite{Gui21c}, $\scr E$ can be viewed as a universal non-local extension of $\mc A$ such that any extension of $\mc A$ is a subquotient of $\scr E$ described by $Q$-systems. Therefore, the uniqueness of $\scr E$ says that $\mc A$ has a unique universal non-local extension of $\scr E$. (The relation between $\scr E$ and a non-local extension is similar to that between a free group and a group determined by some relations. Therefore, the $Q$-systems play the same role as the group relations.) Since $\scr E$ does not live on a fixed Hilbert space but on a category of Hilbert spaces, we call it a categorical extension.

We collect some useful formulas. By the locality and the state-field correspondence, it is  easy to see that
\begin{align}
L(\xi,\wtd I)\eta=R(\eta,\wtd J)\xi
\end{align}
whenever $\xi\in\mc H_i(I)$, $\eta\in\mc H_j(J)$, and $\wtd I$ is anticlockwise to $\wtd J$. Moreover, if $F\in\Hom_{\mc A}(\mc H_i,\mc H_{i'})$, $G\in\Hom_{\mc A}(\mc H_j,\mc H_{j'})$,  $\xi\in\mc H_i(I)$, and $\eta\in\mc H_j$, then
\begin{align}
(F\boxtimes G)L(\xi,\wtd I)\eta=L(F\xi,\wtd I)G\eta,\qquad (G\boxtimes F)R(\xi,\wtd I)\eta=R(F\xi,\wtd I)G\eta;\label{eq14}
\end{align}
see \cite{Gui21b} section 2. We also recall the following fusion relations proved in \cite{Gui21b} proposition 2.3.
\begin{pp}\label{lb10}
	Let $\mc H_i,\mc H_j,\mc H_k\in\Obj(\RepA)$, $\wtd I\in\Jtd$, and $\xi\in\mc H_i(I)$.\\
	(a) If $\eta\in\mc H_j(I)$, then $L(\xi,\wtd I)\eta\in(\mc H_i\boxtimes\mc H_j)(I)$, $R(\xi,\wtd I)\eta\in(\mc H_j\boxtimes\mc H_i)(I)$, and
	\begin{gather}
	L(\xi,\wtd I)L(\eta,\wtd I)|_{\mc H_k}=L(L(\xi,\wtd I)\eta,\wtd I)|_{\mc H_k},\label{eq9}\\
	R(\xi,\wtd I)R(\eta,\wtd I)|_{\mc H_k}=R(R(\xi,\wtd I)\eta,\wtd I)|_{\mc H_k}.\label{eq10}
	\end{gather}
	(b) If $\psi\in(\mc H_i\boxtimes H_j)(I)$ and $\phi\in (\mc H_j\boxtimes H_i)(I)$, then $L(\xi,\wtd I)^*\psi\in\mc H_j(I)$, $R(\xi,\wtd I)^*\phi\in\mc H_j(I)$, and
	\begin{gather}
	L(\xi,\wtd I)^*L(\psi,\wtd I)|_{\mc H_k}=L(L(\xi,\wtd I)^*\psi,\wtd I)|_{\mc H_k},\label{eq11}\\
	R(\xi,\wtd I)^*R(\phi,\wtd I)|_{\mc H_k}=R(R(\xi,\wtd I)^*\phi,\wtd I)|_{\mc H_k}.\label{eq12}
	\end{gather}
\end{pp}
As a special case, we see that if $\xi\in\mc H_i(I)$ and $x\in\mc A(I)$, then $x\xi\in\mc H_i(I)$, and
\begin{gather}
L(x\xi,\wtd I)=xL(\xi,\wtd I),\qquad R(x\xi,\wtd I)=xR(\xi,\wtd I).\label{eq13}
\end{gather}
Set $\xi=\Omega$ and notice that $L(\Omega,\wtd I)=\id$. Then we have
\begin{align}
L(x\Omega,\wtd I)|_{\mc H_k}= R(x\Omega,\wtd I)|_{\mc H_k}=\pi_{k,I}(x)
\end{align}
for any $\mc H_k\in\Obj(\RepA)$.

Next, we discuss the conformal covariance of $\scr E$. For any $\wtd I=(I,\arg_I)\in\Jtd$ and $g\in\GA$, we have $gI$ defined by the action of $\Diffp(\mbb S^1)$ on $\mbb S^1$. We now set $g\wtd I=(gI,\arg_{gI})$, where $\arg_{gI}$ is defined as follows. Choose any map $\gamma:[0,1]\rightarrow\GA$ satisfying $\gamma(0)=1,\gamma(1)=g$ such that $\gamma$ descends to a (continuous) path in $\scr G$. Then for any $z\in I$ there is a path $\gamma_z:[0,1]\rightarrow \mbb S^1$ defined by $\gamma_z(t)=\gamma(t)z$. The argument $\arg_I(z)$ of $z$ changes continuously along the path $\gamma_z$ to an argument of $gz$, whose value is denoted by $\arg_{gI}(gz)$.  

\begin{thm}[\cite{Gui21a} theorem 3.13]\label{lb100}
	$\scr E=(\mc A,\RepA,\boxtimes,\mc H)$ is \textbf{conformal covariant}, which means that for any $g\in\GA,\wtd I\in\wtd{\mc J},\mc H_i\in\Obj(\RepA),\xi\in\mc H_i(I)$, there exists an element $g\xi g^{-1}\in\mc H_i(g I)$ such that
	\begin{align}
	L(g\xi g^{-1},g\wtd I)=gL(\xi,\wtd I)g^{-1},\qquad R(g\xi g^{-1},g\wtd I)=gR(\xi,\wtd I)g^{-1}\label{eq15}
	\end{align}
	when acting on any $\mc H_j\in\Obj(\RepA)$.
\end{thm}
The element $g\xi g^{-1}$ satisfying the requirement of Thm. \ref{lb100} is uniquely determined by $\wtd I,\xi,g$, since we clearly have
\begin{align}
g\xi g^{-1}=gL(\xi,\wtd I)g^{-1}\Omega=gR(\xi,\wtd I)g^{-1}\Omega.
\end{align}

The following density result is easy but useful.

\begin{pp}\label{lb2}
Let $\wtd I\in\Jtd$. Assume that  $\mc H_i^0(I)$ and $\mc H_j^0$ are dense subspaces of $\mc H_i(I)$ and $\mc H_j$ respectively. Then vectors of the form $L(\xi,\wtd I)\eta$ (resp. $R(\xi,\wtd I)\eta$)  span a dense subspace of $\mc H_i\boxtimes\mc H_j$ (resp. $\mc H_j\boxtimes\mc H_i$), where $\xi\in\mc H_i^0(I),\eta\in\mc H_j^0$.
\end{pp}
\begin{proof}
By the boundedness of $L(\xi,\wtd I)$ when $\xi\in\mc H_i^0(I)$, $L(\mc H_i^0(I),\wtd I)\mc H_j^0$ is dense in 	$L(\mc H_i^0(I),\wtd I)\mc H_j$. Thus, for any $\wtd J$ clockwise to $\wtd I$, vectors in $L(\mc H_i^0(I),\wtd I)\mc H_j(J)=R(\mc H_j(J),\wtd J)\mc H_i^0$ can be approximated by those in $L(\mc H_i^0(I),\wtd I)\mc H_j^0$. On the other hand, $R(\mc H_j(J),\wtd J)\mc H_i^0$ is dense in $R(\mc H_j(J),\wtd J)\mc H_i$, and the latter spans a dense subspace of $\mc H_i\boxtimes\mc H_j$. Therefore $L(\mc H_i^0(I),\wtd I)\mc H_j^0$ also spans a dense subspace of $\mc H_i\boxtimes\mc H_j$.
\end{proof}

We are going to prove a weak version of additivity property for $\scr E$. First, we need a lemma.

\begin{lm}\label{lb6}
	Suppose that $\xi\in\mc H_i(I)$ and $L(\xi,\wtd I)|_{\mc H_0}$ is unitary, then for any $\mc H_j\in\Obj(\RepA)$, $L(\xi,\wtd I)|_{\mc H_j}$ and $R(\xi,\wtd I)|_{\mc H_j}$ are also unitary.
\end{lm}

\begin{proof}
	Assume that $L(\xi,\wtd I)^*L(\xi,\wtd I)|_{\mc H_0}=\id_0$ and  $L(\xi,\wtd I)L(\xi,\wtd I)^*|_{\mc H_i}=\id_i$. Choose any $\mc H_j$ and any $\wtd J$ clockwise to $\wtd I$. Then for any $\eta\in\mc H_j(\wtd J)$ and $\mu\in\mc H_i(\wtd J)$,
	\begin{align*}
	L(\xi,\wtd I)^*L(\xi,\wtd I)\eta=L(\xi,\wtd I)^*L(\xi,\wtd I)R(\eta,\wtd J)\Omega=R(\eta,\wtd J)L(\xi,\wtd I)^*L(\xi,\wtd I)\Omega=R(\eta,\wtd J)\Omega=\eta,
	\end{align*}
	and 
	\begin{align*}
	L(\xi,\wtd I)L(\xi,\wtd I)^*R(\mu,\wtd J)\eta=R(\mu,\wtd J)L(\xi,\wtd I)L(\xi,\wtd I)^*\eta=R(\mu,\wtd J)\eta.
	\end{align*}
\end{proof}

For any $\xi\in\mc H_i(I)$, we define $\lVert L(\xi,\wtd I) \lVert\equiv \lVert R(\xi,\wtd I) \lVert:=\lVert L(\xi,\wtd I)|_{\mc H_0} \lVert=\lVert R(\xi,\wtd I)|_{\mc H_0} \lVert$ to be the norms of $L(\xi,\wtd I)$ and $R(\xi,\wtd I)$.  \index{Lxi@$\lVert L(\xi,\wtd I) \lVert=\lVert R(\xi,\wtd I) \lVert$} The following proposition is also needed.

\begin{pp}\label{lb20}
We have $\lVert L(\xi,\wtd I)|_{\mc H_j} \lVert=\lVert R(\xi,\wtd I)|_{\mc H_j} \lVert=\lVert L(\xi,\wtd I)\lVert=\lVert R(\xi,\wtd I)\lVert$ for any non-trivial $\mc A_j$-module $\mc H_j$.
\end{pp}
\begin{proof}
Since $\mc A(I')$ is a type III factor, the representations of $\mc A(I')$ on $\mc H_0$ and on $\mc H_i$ are equivalent. Thus we can choose a unitary $U\in\Hom_{\mc A(I')}(\mc H_0,\mc H_i)$. Let $\mu=U\Omega$. Then $L(\mu,\wtd I)|_{\mc H_0}=U$ is unitary. Notice that   $x:=L(\mu,\wtd I)^*L(\xi,\wtd I)|_{\mc H_0}\in\mc A(I)$, and $\xi=L(\mu,\wtd I)x\Omega$. By \eqref{eq9}, we have $L(\xi,\wtd I)|_{\mc H_j}=L(\mu,\wtd I)\pi_{i,I}(x)$. By lemma \ref{lb6}, $L(\mu,\wtd I)$ is unitary on  $\mc H_j$. Therefore, $\lVert L(\xi,\wtd I)|_{\mc H_j} \lVert=\lVert x\lVert$, which is independent of $\mc H_j$.
\end{proof}

We now prove the weak additivity property for $\scr E$. For $I_0,I\in\mc J$, we write $I_0\Subset I$ if the closure of $I_0$ is a subset of $I$.

\begin{pp}\label{lb7}
Choose $\wtd I\in\Jtd$. Then for any $\xi\in\mc H_i(I)$, there exists a sequence of vectors $\xi_n$ in $\mc H_i$ satisfying the following properties:\\
	(a) For each $n$, $\xi_n\in\mc H_i(I_n)$ for some $I_n\subset\joinrel\subset I$.\\
	(b) For each $\mc H_j\in\Obj(\RepA)$, we have $\sup_{n\in\mbb Z_+}\big\lVert L(\xi_n,\wtd I) \big\lVert\leq\big\lVert L(\xi,\wtd I) \big\lVert$, and $L(\xi_n,\wtd I)|_{\mc H_j}$ converges $*$-strongly to $L(\xi,\wtd I)|_{\mc H_j}$.\footnote{A sequence of bounded operators $x_n$ is said to converge $*$-strongly to $x$ if $x_n$ and $x_n^*$ converge strongly to $x^*$ and $x_n^*$ respectively.}
\end{pp}
In particular, $\xi_n$ is converging to $\xi$ since $\xi_n=L(\xi_n,\wtd I)\Omega$ and $\xi=L(\xi,\wtd I)\Omega$.

\begin{proof}
Fix $I_0\Subset I$. As in the proof of proposition \ref{lb20}, we can choose $\mu\in\mc H_i(I_0)$ such that $L(\mu,\wtd I)$ is unitary on any $\mc A$-module. Choose any $\xi\in\mc H_i(I)$ and set $x=L(\mu,\wtd I)^*L(\xi,\wtd I)|_{\mc H_0}\in\mc A(I)$. Then $\xi=L(\mu,\wtd I)x\Omega$.  By the additivity of $\mc A$, there exist a sequence  $x_n$ of operators converging $*$-strongly to $x$, such that each $x_n$ belongs to $\mc A(I_n)$ for some $I_n\in\mc J$ satisfying $I_0\subset I_n\Subset I$. Moreover, by the Kaplansky density theorem, we may assume that $\lVert x_n\lVert\leq\lVert x\lVert$. Set $\xi_n=L(\mu,\wtd I)x_n\Omega$. Then $\xi_n\in\mc H_i(I_n)$, and by \eqref{eq9}, we have $L(\xi_n,\wtd I)=L(\mu,\wtd I)x_n$, which converges $*$-strongly to $L(\mu,\wtd I)x=L(\xi,\wtd I)$ on any $\mc H_j$, and $\big\lVert L(\xi_n,\wtd I)\big\lVert= \lVert x_n\lVert\leq \lVert x\lVert=\big\lVert L(\xi,\wtd I)\big\lVert$.
\end{proof}

\subsection{Closable field operators}\label{lb11}

For any $\wtd I\in\Jtd$, recall that $\wtd I'$ is the clockwise complement of $\wtd I$. We define ${\bpr\wtd I}\in\Jtd$ such that $(\bpr\wtd I)'=\wtd I$, and call ${\bpr\wtd I}$ the \textbf{anticlockwise complement of $\wtd I$}.  For any $\xi\in\mc H_i$, we let $\scr L(\xi,\wtd I)$ (resp. $\scr R(\xi,\wtd I)$) act on any $\mc H_j\in\Obj(\RepA)$ as an unbounded operator $\mc H_j\rightarrow\mc H_i\boxtimes\mc H_j$ (resp. $\mc H_j\rightarrow\mc H_j\boxtimes\mc H_i$) with domain $\mc H_j(I')$ such that for any $\eta\in\mc H_j(I')$,
\begin{align}
\scr L(\xi,\wtd I)\eta=R(\eta,\wtd I')\xi,\quad \text{resp.}  \quad \scr R(\xi,\wtd I)\eta=L(\eta,\bpr\wtd I)\xi.\label{eq7}
\end{align}
It is clear that $\Omega$ is inside the domains of $\scr L(\xi,\wtd I)|_{\mc H_0}$ and $\scr R(\xi,\wtd I)|_{\mc H_0}$, and the state-field correspondence
\begin{align}
\scr L(\xi,\wtd I)\Omega=\scr R(\xi,\wtd I)\Omega=\xi
\end{align}
is satisfied. We also have that
\begin{align}
\scr L(\xi,\wtd I)|_{\mc H_0}=\scr R(\xi,\wtd I)|_{\mc H_0},
\end{align}
and that they depend only on $I$ but not on the choice of $\arg_I$. Indeed, both operators send any $y\Omega\in\mc A(I')\Omega$ to $y\xi$.

\begin{df}
For any $\mc H_i\in\Obj(\RepA)$ we let $\mc H_i^\pr(I)$ \index{Hi@$\mc H_i^\pr(I)$} be the set of all $\xi\in\mc H_i$ such that $\scr L(\xi,\wtd I)|_{\mc H_0}=\scr R(\xi,\wtd I)|_{\mc H_0}$ is closable. Clearly $\mc H_i(I)\subset\mc H_i^\pr(I)$.
\end{df}

Observe the following easy fact:

\begin{pp}
	If $I\subset J\in\mc J$ then $\mc H_i^\pr(I)\subset\mc H_i^\pr(J)$.
\end{pp}

The following theorem was proved in \cite{Gui21b} section 7. Here we give a different but (hopefully) more conceptual proof.

\begin{thm}\label{lb12}
Choose any $\mc H_i\in\Obj(\RepA)$, $\wtd I\in\Jtd$, and $\xi\in\mc H_i^\pr(I)$. Then $\scr L(\xi,\wtd I)|_{\mc H_j}$ and $\scr R(\xi,\wtd I)|_{\mc H_j}$ are closable for any $\mc H_j\in\Obj(\RepA)$.
\end{thm}

\begin{proof}
Let $\wtd J$ be clockwise to $\wtd I$. Then by lemma \ref{lb6}, there exists $\mu\in\mc H_j(\wtd J)$ such that $R(\mu,\wtd J)$ is unitary when acting on any $\mc A$-module. We claim that
\begin{align}
\scr L(\xi,\wtd I)R(\mu,\wtd J)|_{\mc H_0}= R(\mu,\wtd J)\scr L(\xi,\wtd I)|_{\mc H_0}.\label{eq26}
\end{align}
If this is proved, then $\scr L(\xi,\wtd I)|_{\mc H_j}$ is unitarily equivalent to $\scr L(\xi,\wtd I)|_{\mc H_0}$ through the unitary operators $R(\mu,\wtd J)|_{\mc H_0}$ and $R(\mu,\wtd J)|_{\mc H_i}$. Therefore $\scr L(\xi,\wtd I)|_{\mc H_j}$ is closable since $\scr L(\xi,\wtd I)|_{\mc H_0}$ is so.

The domain of the right hand side of \eqref{eq26} is the same as that of $\scr L(\xi,\wtd I)|_{\mc H_0}$, which is $\mc H_0(I')=\mc A(I')\Omega$. Clearly $R(\mu,\wtd J)\mc A(I')\Omega$ is a subspace of $\mc H_j(I')$ (which is the domain of $\scr L(\xi,\wtd I)|_{\mc H_j}$). On the other hand, we have $R(\mu,\wtd J)^*\mc H_j(I')=R(\mu,\wtd J)^*\Hom_{\mc A(I)}(\mc H_0,\mc H_j)\Omega$, which is a subspace of $\Hom_{\mc A(I)}(\mc H_0,\mc H_0)\Omega=\mc A(I')\Omega$ by Haag duality. Therefore $R(\mu,\wtd J)\mc A(I')\Omega$ equals $\mc H_j(I')$, which shows that both sides of \eqref{eq26} have the same domain $\mc A(I')\Omega$. Now, we choose any $\chi\in\mc A(I')\Omega=\mc H_0(I')$ and use proposition \ref{lb10} to compute that
\begin{align*}
\scr L(\xi,\wtd I)R(\mu,\wtd J)\chi=R(R(\mu,\wtd J)\chi,\wtd J)\xi=R(\mu,\wtd J)R(\chi,\wtd J)\xi=R(\mu,\wtd J)\scr L(\xi,\wtd I)\chi,
\end{align*}
which proves equation \eqref{eq26} and hence the closability of $\scr L(\xi,\wtd I)|_{\mc H_j}$. That $\scr R(\xi,\wtd I)|_{\mc H_j}$ is closable follows from a similar argument.
\end{proof}

\begin{cv}
Whenever $\xi\in\mc H_i^\pr(I)$, we will always understand $\scr L(\xi,\wtd I)$ and $\scr R(\xi,\wtd I)$ as closed operators, which are the closures of those defined by \eqref{eq7}.
\end{cv}

Note that when $\xi\in\mc H_i(I)$, it is clear  that our definition of the two closed operators agree with the original bounded ones. The following proposition shows that $\scr L(\xi,\wtd I)$ and $\scr R(\xi,\wtd I)$ intertwine the actions of $\mc A(I')$.

\begin{pp}\label{lb18}
	For any $\mc H_i,\mc H_j\in\Obj(\RepA)$, $\xi\in\mc H_i^\pr(I)$, and $x\in\mc A(I')$, the following diagrams commute strongly:
	\begin{gather}
	\begin{CD}
	\mc H_j @>~\pi_{j,I'}(x)~>> \mc H_j\\
	@V \scr L(\xi,\wtd I)  VV @VV \scr L(\xi,\wtd I) V\\
	\mc H_i\boxtimes\mc H_j @> \pi_{i\boxtimes j,I'}(x)>> \mc H_i\boxtimes\mc H_j
	\end{CD}\qquad\qquad\qquad
	\begin{CD}
	\mc H_j @> \scr R(\xi,\wtd I)  >> \mc H_j\boxtimes\mc H_i\\
	@V \pi_{j,I'}(x) VV @VV \pi_{j\boxtimes i,I'}(x) V\\
	\mc H_j @>\scr R(\xi,\wtd I) >> \mc H_j\boxtimes\mc H_i
	\end{CD}
	\end{gather}
\end{pp}

\begin{proof}
	Choose any $\eta\in\mc H_j(I')$. Then
	\begin{align*}
	\scr L(\xi,\wtd I)x\eta=R(x\eta,\wtd I')\xi \xlongequal{\eqref{eq13}} xR(\eta,\wtd I')\xi=x\scr L(\xi,\wtd I)\eta.
	\end{align*}
	Similarly we have $\scr L(\xi,\wtd I)x^*\eta=x^*\scr L(\xi,\wtd I)\eta$. This proves the strong commutativity of the first diagram. (Notice proposition \ref{lb15}.) The second diagram can be proved similarly.
\end{proof}

In the above proposition, we have actually used the following definition.

\begin{df}\label{lb77}
	Let $\mc P_0,\mc Q_0, \mc R_0,\mc S_0$ be pre-Hilbert spaces with completions $\mc P,\mc Q,\mc R,\mc S$ respectively. Let  $A:\mc P\rightarrow\mc R,B:\mc Q\rightarrow\mc S,C:\mc P\rightarrow\mc Q,D:\mc R\rightarrow\mc S$ be closable operators whose domains are subspaces of $\mc P_0,\mc Q_0, \mc P_0,\mc R_0$ respectively, and whose ranges are inside $\mc R_0,\mc S_0,\mc Q_0,\mc S_0$ respectively. By saying that the diagram of closable operators
	\begin{align}
	\begin{CD}
	\mc P_0 @>C>> \mc Q_0\\
	@V A VV @V B VV\\
	\mc R_0 @>D>> \mc S_0
	\end{CD}
	\end{align}
	\textbf{commutes strongly}, we mean the following: Let $\mc H=\mc P\oplus\mc Q\oplus\mc R\oplus\mc S$. Define unbounded closable operators $R,S$ on $\mc H$ with domains $\Dom(R)=\Dom(A)\oplus\Dom(B)\oplus\mc R\oplus \mc S$, $\Dom(S)=\Dom(C)\oplus\mc Q\oplus\Dom(D)\oplus \mc S$, such that
	\begin{gather*}
	R(\xi\oplus\eta\oplus\chi\oplus\varsigma)=0\oplus 0\oplus A\xi\oplus B\eta\qquad(\forall \xi\in\Dom(A),\eta\in\Dom(B),\chi\in\mc R,\varsigma\in \mc S),\\
	S(\xi\oplus\eta\oplus\chi\oplus\varsigma)=0\oplus C\xi\oplus 0\oplus D\chi   \qquad(\forall \xi\in\Dom(C),\eta\in\mc Q,\chi\in \Dom(D),\varsigma\in\mc S).
	\end{gather*}
	(Such construction is called the \textbf{extension} from $A,B$ to $R$, and from $C,D$ to $S$.) Then (the closures of) $R$ and $S$ commute strongly (cf. Ch. \ref{lb9}).
\end{df}

We now state the main result of this section.

\begin{thm}\label{lb13}
	For any $\mc H_i,\mc H_j,\mc H_k,\mc H_{k'}\in\Obj(\RepA)$, the following are satisfied.
	
	(a) (Isotony) If $\wtd I_1\subset\wtd I_2\in\Jtd$, and $\xi\in\mc H_i^\pr(I_1)$, then $\scr L(\xi,\wtd I_1)\supset \scr L(\xi,\wtd I_2)$, $\scr R(\xi,\wtd I_1)\supset \scr R(\xi,\wtd I_2)$ when acting on   $\mc H_k$.
	
	(b) (Naturality) If $G\in\Hom_{\mc A}(\mc H_k,\mc H_{k'})$, then for any $\wtd I\in\Jtd,\xi\in\mc H_i^\pr(I)$, the following diagrams of closed operators commute strongly.
	\begin{gather}
	\begin{CD}
	\mc H_k @>G>> \mc H_{k'}\\
	@V \scr L(\xi,\wtd I)  VV @V \scr L(\xi,\wtd I)  VV\\
	\mc H_i\boxtimes\mc H_k @> \id_i\boxtimes G>> \mc H_i\boxtimes\mc H_{k'}
	\end{CD}\qquad\qquad
	\begin{CD}
	\mc H_k @> \scr R(\xi,\wtd I)  >> \mc H_k\boxtimes\mc H_i\\
	@V G VV @V G\boxtimes\id_i  VV\\
	\mc H_{k'} @>\scr R(\xi,\wtd I) >> \mc H_{k'}\boxtimes\mc H_i
	\end{CD}.\label{eq45}
	\end{gather}

	(c) (Locality) For any  disjoint $\wtd I,\wtd J\in\Jtd$ with $\wtd I$ anticlockwise to $\wtd J$, and any $\xi\in\mc H_i^\pr(I),\eta\in\mc H_j^\pr(J)$, the following diagram \eqref{eq48}  commutes strongly.
	\begin{align}
	\begin{CD}
	\mc H_k @> \quad \scr R(\eta,\wtd J)\quad   >> \mc H_k\boxtimes\mc H_j\\
	@V \scr L(\xi,\wtd I)   V  V @V \scr L(\xi,\wtd I) VV\\
	\mc H_i\boxtimes\mc H_k @> \quad \scr R(\eta,\wtd J) \quad  >> \mc H_i\boxtimes\mc H_k\boxtimes\mc H_j
	\end{CD}\label{eq48}
	\end{align}

	(d) (Braiding) For any $\wtd I\in\Jtd,\xi\in\mc H_i^\pr(I)$, we have
	\begin{align}
	\mbb B_{i,j} \scr L(\xi,\wtd I)|_{\mc H_j}=\scr R(\xi,\wtd I)|_{\mc H_j}.\label{eq49}
	\end{align}
	
	(e) (M\"obius covariance) For any $g\in\UPSU,\wtd I\in\wtd{\mc J},\xi\in\mc H_i^\pr(I)$, we have $g\xi\in\mc H_i^\pr(g I)$, and
	\begin{align}
	\scr L(g\xi,g\wtd I)=g\scr L(\xi,\wtd I)g^{-1},\qquad \scr R(g\xi,g\wtd I)=g\scr R(\xi,\wtd I)g^{-1}\label{eq46}
	\end{align}
	when acting on $\mc H_j$.
\end{thm}

Note that in the axiom of M\"obius covariance, $\UPSU$ is the universal covering of $\PSU$, which is regarded as a subgroup of $\GA$. This is possible since, by \cite{Bar54}, the restriction of any strongly continuous projective representation $\scr G$ to $\UPSU$ can be lifted to a unique strongly continuous unitary representation of $\UPSU$. 

Isotony is obvious; note that the inclusion relations are reversed since $I_1'\supset I_2'$. To check the strong commutativity of the first of \eqref{eq45}, one just need to show that for any $\eta\in\mc H_k(I'),\eta'\in\mc H_{k'}(I')$, we have $\scr L(\xi,\wtd I)G\eta=(\id_i\boxtimes G)\scr L(\xi,\wtd I)\eta$ and  $\scr L(\xi,\wtd I)G^*\eta'=(\id_i\boxtimes G^*)\scr L(\xi,\wtd I)\eta'$. Indeed, by \eqref{eq14}, we have
\begin{align*}
\scr L(\xi,\wtd I)G\eta=R(G\eta,\wtd I')\xi=(\id_i\boxtimes G)R(\eta,\wtd I')\xi=(\id_i\boxtimes G)\scr L(\xi,\wtd I)\eta.
\end{align*}
The other equation is proved similarly. For the second diagram of \eqref{eq45}, one uses a similar argument. 

\begin{proof}[Proof of M\"obius covariance]
Choose any $\mc H_j\in\Obj(\RepA)$ and $\eta\in\mc H_j(gI')$. Then $g^{-1}\eta\in\mc H_j(I')$, and 
	\begin{align*}
	g\scr L(\xi,\wtd I)g^{-1}\eta=gR(g^{-1}\eta,\wtd I')\xi \xlongequal{\eqref{eq15}} R(\eta,g\wtd I')g\xi=\scr L(g\xi,g\wtd I)\eta.
	\end{align*}
This proves the first equation of \eqref{eq46}. In particular, $\scr L(g\xi,g\wtd I)$ is closable since $g\scr L(\xi,\wtd I)g^{-1}$ is so. Therefore $\xi=g\scr L(\xi,\wtd I)g^{-1}\Omega\in\mc H_i^\pr(I)$. The second half follows similarly.	
\end{proof}

\begin{proof}[Proof of braiding]
By M\"obius covariance, it suffices to prove \eqref{eq49} assuming $\wtd I=\wtd{\mbb S^1_+}$. Then we have $\wtd I'=\varrho(-2\pi)\cdot\bpr\wtd I$. Note that before taking closures, the domains of $\scr L(\xi,\wtd I)|_{\mc H_j}$ and $\scr R(\xi,\wtd I)|_{\mc H_j}$ are both $\mc H_j(I')$. Therefore, it suffices to show that $\mbb B \scr L(\xi,\wtd I)\eta=\scr R(\xi,\wtd I)\eta$ for any $\eta\in \mc H_j(I')$. We set $u=\mbb B^2\varrho(-2\pi)(\varrho(2\pi)\boxtimes\varrho(2\pi))$ on $\mc H_i\boxtimes\mc H_j$, and compute
\begin{align}
&~~~~~~ \mbb B \scr L(\xi,\wtd I)\eta=\mbb B R(\eta,\wtd I')\xi=\mbb B^2L(\eta,\wtd I')\xi=\mbb B^2 L(\eta,\varrho(-2\pi)\cdot\bpr\wtd I)\xi\nonumber\\ &\xlongequal{\eqref{eq15}} \mbb B^2\varrho(-2\pi)\cdot L(\varrho(2\pi)\eta,\bpr\wtd I)\cdot \varrho(2\pi)\xi\xlongequal{\eqref{eq14}}uL(\eta,\bpr\wtd I)\xi =u\scr R(\xi,\wtd I)\eta.\label{eq16}
\end{align}
When $\xi\in\mc H_i(I)$, we actually have $\mbb B\scr L(\xi,\wtd I)\eta=\scr R(\xi,\wtd I)\eta$. Thus $u$ must be $1$. Therefore, when $\xi\in\mc H_i^\pr(I)$, \eqref{eq16} is also true with $u=1$.
\end{proof}

To prove the locality, we first prove the following weaker version:
\begin{lm}\label{lb17}
	Diagram \eqref{eq48} commutes strongly when $\xi\in\mc H_i(I)$ or $\eta\in\mc H_j(J)$.
\end{lm}

\begin{proof}
	Let us assume $\eta\in\mc H_j(J)$ and prove the strong commutativity of \eqref{eq48}. Notice that $R(\eta,\wtd J)=R(\eta,\wtd I')$ as bounded operators. Choose any $\chi\in\mc H_k(I')$ and $\phi\in(\mc H_k\boxtimes\mc H_j)(I')$. By proposition \ref{lb10}, we have that $R(\eta,\wtd I)\chi\in(\mc H_k\boxtimes\mc H_j)(I')$, $R(\eta,\wtd I')^*\phi\in\mc H_k(I')$, that
	\begin{align*}
	\scr L(\xi,\wtd I)R(\eta,\wtd I')\chi=R(R(\eta,\wtd I')\chi,\wtd I')\xi=R(\eta,\wtd I')R(\chi,\wtd I')\xi=R(\eta,\wtd I')\scr L(\xi,\wtd I)\chi,
	\end{align*}
	and that
	\begin{align*}
	\scr L(\xi,\wtd I)R(\eta,\wtd I')^*\phi=R(R(\eta,\wtd I')^*\phi,\wtd I')\xi=R(\eta,\wtd I')^*R(\phi,\wtd I')\xi=R(\eta,\wtd I')^*\scr L(\xi,\wtd I)\phi.
	\end{align*}
\end{proof}

We now want to approximate $\scr L(\xi,\wtd I)$ and $\scr R(\xi,\wtd I)$ using bounded $L$ and $R$ operators. Again we assume $\xi\in\mc H_i^\pr(I)$. Consider the polar decomposition $\scr L(\xi,\wtd I)=UH$ where $U$ is a partial isometry from each $\mc H_j$ to $\mc H_i\boxtimes\mc H_j$, and $H$ is a positive closed operator on each $\mc H_j$. Consider the spectral decomposition $H=\int_0^{+\infty}tdq_t$, where for each $\mc H_j\in\Obj(\RepA)$,  $q_t$ is a projection on $\mc H_j$ defined by $q_t=\chi_{[0,t]}(H)$, and $p_t=Uq_tU^*$  is a projection on $\mc H_i\boxtimes\mc H_j$. Then, by the spectral theory, we have the following equation of \emph{bounded} operators:
\begin{align*}
\ovl {p_t\scr L(\xi,\wtd I)}=\scr L(\xi,\wtd I)q_t,
\end{align*}
and for any $\eta\in\mc H_j$,  $\eta$ is inside the domain of $\scr L(\xi,\wtd I)$ if and only if the limit
\begin{align*}
\lim_{t\rightarrow+\infty}\lVert {p_t\scr L(\xi,\wtd I)}\eta \lVert^2=\lim_{t\rightarrow+\infty}\lVert {\scr L(\xi,\wtd I)}q_t\eta \lVert^2
\end{align*}
is finite, in which case we have
\begin{align*}
\scr L(\xi,\wtd I)\eta=\lim_{t\rightarrow+\infty} {p_t\scr L(\xi,\wtd I)}\eta =\lim_{t\rightarrow+\infty}{\scr L(\xi,\wtd I)}q_t\eta. 
\end{align*}
We call $\{p_t:t\geq0\}$ (resp. $\{q_t:t\geq0\}$) the \textbf{left (resp. right) bounding projections} of $\scr L(\xi,\wtd I)|_{\mc H_j}$.
Similarly, for each $\mc H_j$ we have left and right bounding projections $\{e_t\}$ (on $\mc H_j\boxtimes \mc H_i$)  and $\{f_t\}$ (on $\mc H_j$)  of $\scr R(\xi,\wtd I)$. We write $p_t,q_t,e_t,f_t$ respectively as $p_t(\xi,\wtd I),q_t(\xi,\wtd I),e_t(\xi,\wtd I),f_t(\xi,\wtd I)$ \index{pt@$p_t(\xi,\wtd I),q_t(\xi,\wtd I),e_t(\xi,\wtd I),f_t(\xi,\wtd I)$} to emphasize the dependence of the projections on the vectors and the arg-valued intervals.

\begin{lm}\label{lb19}
	For each $t$ and $\eta\in\mc H_j(I')$, $R(\eta,\wtd I')$ commutes with $p_t(\xi,\wtd I)$ and $q_t(\xi,\wtd I)$, and $L(\eta,\bpr\wtd I)$ commutes with $e_t(\xi,\wtd I)$ and $f_t(\xi,\wtd I)$.
\end{lm}
\begin{proof}
	Use lemma \ref{lb17}.
\end{proof}

\begin{pp}
	For each $t$, we have $p_t(\xi,\wtd I)\xi\in\mc H_i(I),e_t(\xi,\wtd I)\xi\in\mc H_i(I)$, and 
	\begin{gather}
	L(p_t(\xi,\wtd I)\xi,\wtd I)=\ovl{p_t(\xi,\wtd I)\scr L(\xi,\wtd I)}=\scr L(\xi,\wtd I)q_t(\xi,\wtd I),\\
	R(e_t(\xi,\wtd I)\xi,\wtd I)=\ovl{e_t(\xi,\wtd I)\scr R(\xi,\wtd I)}=\scr R(\xi,\wtd I)f_t(\xi,\wtd I).
	\end{gather}
\end{pp}

\begin{proof}
	By proposition \ref{lb18}, the bounded operator $\ovl{p_t(\xi,\wtd I)\scr L(\xi,\wtd I)}$ intertwines the actions of $\mc A(I')$. Therefore $p_t(\xi,\wtd I)\xi=p_t(\xi,\wtd I)\scr L(\xi,\wtd I)\Omega\in\mc H_i(I)$.  Choose any $\mc H_j\in\Obj(\RepA)$ and $\eta\in\mc H_j(I')$. Then
	\begin{align*}
	&L(p_t(\xi,\wtd I)\xi,\wtd I)\eta=L(p_t(\xi,\wtd I)\xi,\wtd I)R(\eta,\wtd I')\Omega=R(\eta,\wtd I')L(p_t(\xi,\wtd I)\xi,\wtd I)\Omega\\
	=&R(\eta,\wtd I')p_t(\xi,\wtd I)\xi \xlongequal{\text{Lem. } \ref{lb19}} p_t(\xi,\wtd I)R(\eta,\wtd I')\xi=p_t(\xi,\wtd I)\scr L(\xi,\wtd I)\eta.
	\end{align*}
	A similar argument proves the second equation.
\end{proof}

\begin{proof}[Proof of locality]
	By the above proposition, we know that for each $s,t\geq 0$, diagram \eqref{eq48} commutes adjointly when the vertical lines are multiplied by $p_s(\xi,\wtd I)$, and the horizontal lines are multiplied by $e_t(\eta,\wtd J)$. Since $\scr L(\xi,\wtd I)$ and $\scr R(\eta,\wtd J)$ are affiliated with the von Neumann algebras generated by $\{\ovl{p_s(\xi,\wtd I)\scr L(\xi,\wtd I)}:s\geq 0 \}$ and by $\{\ovl {e_t(\eta,\wtd J)\scr R(\eta,\wtd J)}:t\geq 0\}$ respectively, they must commute strongly.
\end{proof} 

\begin{co}\label{lb25}
	Let $\mc H_i,\mc H_j,\mc H_k\in\Obj(\RepA)$, and let $\wtd J$ be clockwise to $\wtd I$. Suppose that $\xi\in\mc H_i^\pr(I),\eta\in\mc H_j^\pr(J),\chi\in\mc H_k$, and that $\chi$ is in the domains of $\scr L(\xi,\wtd I)\scr R(\eta,\wtd J)|_{\mc H_k}$ and $\scr R(\eta,\wtd J)\scr L(\xi,\wtd I)|_{\mc H_k}$. Then
	\begin{align}
	\scr L(\xi,\wtd I)\scr R(\eta,\wtd J)\chi=\scr R(\eta,\wtd J)\scr L(\xi,\wtd I)\chi.
	\end{align}
\end{co}

\begin{proof}
	Use proposition \ref{lb3} and the locality in theorem \ref{lb13}.
\end{proof}

\subsection{The dense subspaces $\mc H_i^\infty$, $\mc H_i^\infty(I)$, and $\mc H_i^w(I)$}

We now introduce for each $\mc H_i\in\RepA$ some important dense subspaces. Let $L_0$ be the generator of the one parameter rotation subgroup $\varrho$ of $\UPSU$ which is positive by \cite[Thm. 3.8]{Wei06}. ($\varrho$ is the lift of the one parameter rotation subgroup of $\PSU$ to $\UPSU$.) Set \index{Hi@$\mc H_i^\infty$}
\begin{align}
\mc H_i^\infty=\bigcap_{n\in\mbb Z_+}\Dom(L_0^n).
\end{align}
Vectors in $\mc H_i^\infty$ are called smooth. For each $\xi\in\mc H_i$,  we have $\xi_h\in\mc H_i^\infty$ where $h\in C_c^\infty(\mbb R)$ and
\begin{align}
\xi_h=\int_{\mbb R}h(t)\varrho(t)\xi dt.
\end{align}
Since $\xi_h\rightarrow\xi$ as $h$ converges to the $\delta$-function at $0$, we conclude that $\mc H_i^\infty$ is dense in $\mc H_i$. 

A closable operator $T$ from $\mc H_i$ to $\mc H_j$ is called \textbf{smooth} if
\begin{gather}
\mc H_i^\infty\subset\Dom(T),\qquad \mc H_j^\infty\subset\Dom(T^*),\qquad  T\mc H_i^\infty\subset \mc H_j^\infty,\qquad T^*\mc H_j^\infty\subset \mc H_i^\infty.
\end{gather}

\begin{rem}\label{lb102}
Any homomorphism of $\mc A$-modules is smooth since it commutes  with $\varrho(t)$. 
\end{rem}
If $T:\mc H_i\rightarrow\mc H_j$ is  bounded, then it is a routine check that for each $h\in C_c^\infty(\mbb R)$ satisfying $\int_{\mbb R}h(t)dt=1$,
\begin{align}
T_h:=\int_{\mbb R}h(t)\varrho(t)T\varrho(-t)dt
\end{align}
is bounded and smooth, and $T_h$  converges $*$-strongly to $T$  as $h$ converges to the $\delta$-function at $0$. See for example \cite{Gui19b} Prop. 4.2 for details.

\begin{df}
We let $\mc H_i^\infty(I)$  be the set of all $\xi\in\mc H_i(I)$ such that for any $\mc H_j\in\Obj(\RepA)$ and any $\arg_I$, the bounded operator $L (\xi,\wtd I)|_{\mc H_j}$ from $\mc H_j$ to $\mc H_i\boxtimes\mc H_j$ is smooth. In that case, $R(\xi,\wtd I)$ is also smooth since the braid operator $\mbb B$ is smooth.
\end{df}

\begin{pp}\label{lb8}
For each $I\in\mc J$, $\mc H_i^\infty(I)$ is a dense subspace of $\mc H_i(I)$. Moreover, for any $\xi\in\mc H_i(I)$, there is a sequence $\{\xi_n\}$ in $\mc H_i^\infty(I)$ such that $\sup_{n\in\mbb Z_+}\big\lVert  L(\xi_n,\wtd I) \big\lVert\leq\big\lVert  L(\xi,\wtd I) \big\lVert$,  and $L(\xi_n,\wtd I)|_{\mc H_j}$ converges $*$-strongly to $L(\xi,\wtd I)|_{\mc H_j}$ for any $\mc H_j\in\Obj(\RepA)$.
\end{pp}

\begin{proof}
Choose  $I_0\in\mc J$ such that $I_0\subset\joinrel\subset I$. Then there exists $\varepsilon>0$ such that $\varrho(t)I_0\subset I$ whenever $-\varepsilon\leq t\leq \varepsilon$. Then for any $\xi\in\mc H_i(I_0)$ and $h\in C_c^\infty(-\varepsilon,\varepsilon)$,  we have $\xi_h\in\mc H_i^\infty(I)$, and
\begin{align}
L(\xi,\wtd I)_h=L(\xi_h,\wtd I)\label{eq44}
\end{align}
when acting on any $\mc H_j\in\Obj(\RepA)$. Indeed, by the conformal covariance of $\scr E$, we have for any $\wtd J$ clockwise to $\wtd I$ and $\eta\in\mc H_j(J)$ that
\begin{align*}
&L(\xi,\wtd I)_h\cdot \eta=\int_{\mbb R}h(t)\varrho(t)L(\xi,\wtd I)\varrho(-t)\eta\cdot dt=\int_{\mbb R}h(t)L(\varrho(t)\xi,\wtd I)\eta\cdot dt\\
=&\int_{\mbb R}h(t)R(\eta,\wtd J)\varrho(t)\xi\cdot dt=R(\eta,\wtd J)\xi_h=\scr L(\xi_h,\wtd I)\eta.
\end{align*}
So $\scr L(\xi_h,\wtd I)$ is continuous when acting on (a dense subspace of) $\mc H_0$. This shows $\xi_h\in\mc H_i(I)$. By the above computation, we conclude \eqref{eq44} and hence $\xi_h\in\mc H_i^\infty(I)$. Assuming $\int h(t)dt=1$, then the norms of $L(\xi,\wtd I)_h$ are bounded by that of $L(\xi,\wtd I)$. Since we know that $L(\xi,\wtd I)_h$ converges $*$-strongly to $L(\xi,\wtd I)$ as $h$ converges to the $\delta$-function at $0$, we have actually proved the statement when $\xi\in\mc H_i(I_0)$.

Now assume $\xi\in\mc H_i(I)$. Then by proposition \ref{lb7}, there exists a sequence $\{\xi_n'\}$ in $\mc H_i$ such that each $\xi_n'$ belongs to $\mc H_i(I_n)$ for some $I_n\subset\joinrel\subset I$, that $\lVert L(\xi_n',\wtd I) \lVert\leq \lVert L(\xi,\wtd I) \lVert$, and that $L(\xi_n',\wtd I)|_{\mc H_0}$ converges $*$-strongly to $L(\xi,\wtd I)|_{\mc H_0}$. By the first paragraph,  and by the fact that $\mc H_0$ and $\mc H_i$ are separable, we can replace each $\xi_n'$ with some $\xi_n\in\mc H_i^\infty(I)$ such that $\lVert L(\xi_n,\wtd I) \lVert\leq \lVert L(\xi,\wtd I) \lVert$ and $L(\xi_n,\wtd I)|_{\mc H_0}$ converges $*$-strongly to $L(\xi,\wtd I)|_{\mc H_0}$. Now, choose any non-trivial $\mc H_j\in\Obj(\RepA)$, and let $\wtd J=\wtd I'$. By lemma \ref{lb6}, we can choose $\mu\in\mc H_j(J)$ such that $R(\mu,\wtd J)$ is unitary. Then $L(\xi_n,\wtd I)|_{\mc H_j}$ equals $R(\mu,\wtd J)L(\xi_n,\wtd I)R(\mu,\wtd J)^*|_{\mc H_j}$, which converges $*$-strongly to $R(\mu,\wtd J)L(\xi,\wtd I)R(\mu,\wtd J)^*|_{\mc H_j}=L(\xi,\wtd I)|_{\mc H_j}$.
\end{proof}	

Let $\mc A^\infty(I)$ be the set of all $x\in\mc A(I)$ such that $\pi_{i,I}(x)$ is smooth for any $\mc H_i\in\Obj(\RepA)$. Then clearly $\mc H^\infty_0(I)=\mc A^\infty(I)\Omega$. As a special case of proposition \ref{lb8}, we have:

\begin{co}\label{lb24}
$\mc A^\infty(I)$ is a strongly dense $*$-subalgebra of $\mc A(I)$.
\end{co}

\begin{pp}\label{lb21}
Suppose that $\xi\in\mc H_i^\pr(I)$. Then for each $\mc H_j\in\Obj(\RepA)$, $\mc H_j^\infty(I')$ is a core for $\scr L(\xi,\wtd I)|_{\mc H_j}$ and $\scr R(\xi,\wtd I)|_{\mc H_j}$.
\end{pp}

\begin{proof}
Choose any $\eta\in\mc H_j(I')$. By proposition \ref{lb8}, there exist $\eta_n\in\mc H_j^\infty(I')$ such that $R(\eta_n,\wtd I')|_{\mc H_i}$ converges $*$-strongly to $R(\eta,\wtd I')|_{\mc H_i}$. Therefore, as $n\rightarrow\infty$, $\scr L(\xi,\wtd I)\eta_n=R(\eta_n,\wtd I')\xi$ converges to $R(\eta,\wtd I')\xi$. This proves that $\mc H_j^\infty(I')$ is a core for $\scr L(\xi,\wtd I)|_{\mc H_j}$ and for $\scr R(\xi,\wtd I)|_{\mc H_j}=\mbb B_{i,j}\scr L(\xi,\wtd I)|_{\mc H_j}$.
\end{proof}

\begin{df}
Define $\mc H_i^w(I)$ \index{Hi@$\mc H_i^w(I)$} to be the set of all $\xi\in\mc H_i$ such that $\scr L(\xi,\wtd I)|_{\mc H_0}$ (which equals $\scr R(\xi,\wtd I)|_{\mc H_0}$) is closable and has smooth closure. 
\end{df}

The following relations are obvious:
\begin{align*}
\mc H_i^\infty(I)\subset&\mc H_i^w(I)\subset\mc H_i^\infty\\
&~~~\cap\\
&\mc H_i^\pr(I)
\end{align*}
Therefore, $\mc H_i^w(I)$ is a dense subspace of $\mc H_i$.

Using the  results in section \ref{lb11} and chapter \ref{lb9}, we are able to prove many important results, for example:

\begin{pp}\label{lb28}
Let $\wtd J$ be clockwise to $\wtd I$. Suppose that $\xi\in\mc H_i^w(I),\eta\in\mc H_j^w(J),\chi\in\mc H_0^\infty$. Then $\scr L(\eta,\wtd J)\chi=\scr R(\eta,\wtd J)\chi$ are inside the domains of $\scr L(\xi,\wtd I)|_{\mc H_j}$ and $\scr R(\xi,\wtd I)|_{\mc H_j}$, and 
\begin{align}
\scr L(\xi,\wtd I)\scr R(\eta,\wtd J)\chi=\scr R(\eta,\wtd J)\scr L(\xi,\wtd I)\chi.
\end{align}
\end{pp}

\begin{proof}
By theorem \ref{lb13}, $\scr L(\xi,\wtd I)$ and $\scr R(\eta,\wtd J)$ commute strongly. Moreover, since $\xi\in\mc H_i^w(I)$ and $\eta\in\mc H_i^w(J)$, the smooth vector $\chi$ is inside the domains of $\scr L(\xi,\wtd I)^*\scr L(\xi,\wtd I)|_{\mc H_0}$ and $\scr R(\eta,\wtd J)^*\scr R(\eta,\wtd J)|_{\mc H_0}$. This finishes the proof by proposition \ref{lb3}-(b).
\end{proof}

\begin{pp}\label{lb26}
	Suppose that $\wtd J$ is clockwise to $\wtd I$. If $\xi_1,\xi_2\in\mc H_i^w(I)$, $\eta_1,\eta_2\in\mc H_j^w(J)$, and $\chi_1,\chi_2\in\mc H_0^\infty$, then 
	\begin{align}
	&\bk{\scr L(\xi_1,\wtd I)\scr R(\eta_1,\wtd J)\chi_1|\scr L(\xi_2,\wtd I)\scr R(\eta_2,\wtd J)\chi_2}\nonumber\\
	=&\bk{\scr L(\xi_2,\wtd I)^*\scr L(\xi_1,\wtd I)\chi_1|\scr R(\eta_1,\wtd J)^*\scr R(\eta_2,\wtd I)\chi_2}\nonumber\\
	=&\bk{\scr R(\eta_2,\wtd J)^*\scr R(\eta_1,\wtd J)\chi_1|\scr L(\xi_1,\wtd I)^*\scr L(\xi_2,\wtd I)\chi_2}.
	\end{align}
\end{pp}

\begin{proof}
	Use lemma \ref{lb14} and the locality in theorem \ref{lb13}.
\end{proof}

\begin{co}\label{lb75}
Let $\wtd J$ be clockwise to $\wtd I$, and choose $\xi\in\mc H_i^w(I),\eta\in\mc H_j^w(J)$. Then $\xi,\eta$ are inside the domains of $\scr R(\eta,\wtd J)|_{\mc H_i},\scr L(\xi,\wtd I)|_{\mc H_j}$ respectively, and
\begin{align}\label{eq18}
\scr L(\xi,\wtd I)\eta=\scr R(\eta,\wtd J)\xi.
\end{align}
\end{co}
\begin{proof}
Apply proposition \ref{lb28} to the case that  $\chi=\Omega$.
\end{proof}

\begin{co}\label{lb5}
	Suppose that $\wtd J$ is clockwise to $\wtd I$. If $\xi_1,\xi_2\in\mc H_i^w(I)$ and $\eta_1,\eta_2\in\mc H_j^w(J)$, then 
	\begin{align}
	&\bk{\scr L(\xi_1,\wtd I)\eta_1|\scr L(\xi_2,\wtd I)\eta_2 }=\bk{\scr R(\eta_1,\wtd J)\xi_1|\scr R(\eta_2,\wtd J)\xi_2}\nonumber\\
	=&\bk{\scr L(\xi_2,\wtd I)^*\xi_1|\scr R(\eta_1,\wtd J)^*\eta_2}=\bk{\scr R(\eta_2,\wtd J)^*\eta_1|\scr L(\xi_1,\wtd I)^*\xi_2}.\label{eq6}
	\end{align}
\end{co}

\begin{proof}
In proposition \ref{lb26}, set $\chi_1=\chi_2=\Omega$.
\end{proof}

\subsection{Weak categorical extensions}\label{lb33}

Let $\scr C$ be a full ($C^*$-)subcategory of $\Rep(\mc A)$ containing $\mc H_0$ and closed under taking submodules and finite (orthogonal) direct sums. 
Equip $\scr C$  with a tensor $*$-bifunctor $\boxdot$, natural unitary associativity isomorphisms $(\mc H_i\boxdot\mc H_j)\boxdot H_k\xrightarrow\simeq\mc H_i\boxdot(\mc H_j\boxdot\mc H_k)$, and unitary isomorphisms $\mc H_i\boxdot\mc H_0\xrightarrow\simeq\mc H_i,\mc H_0\boxdot\mc H_i\xrightarrow\simeq \mc H_i$, so that $\scr C$ becomes a $C^*$-tensor category with unit $\mc H_0$. We treat $\scr C$ as if it is strict by identifying $(\mc H_i\boxdot\mc H_j)\boxdot H_k$, $\mc H_i\boxdot(\mc H_j\boxdot\mc H_k)$ as $\mc H_i\boxdot\mc H_j\boxdot \mc H_k$, and $\mc H_i\boxdot \mc H_0$, $\mc H_0\boxdot\mc H_i$ as $\mc H_i$. Our main example of $\scr C$ arises from VOAs, cf. Def. \ref{lb101}.

Assume that $T:\mc H_i\rightarrow\mc H_j$ is a closable and smooth operator with domain $\mc H_i^\infty$. Unlike previous sections, we will differentiate between $T$  and its closure $\ovl T$. We define $T^\dagger$ to be $T^*|_{\mc H_j^\infty}$, \index{T@$T^\dagger$} called the \textbf{formal adjoint} of $T$. Then $T^\dagger:\mc H_j\rightarrow\mc H_i$ has  dense domain $\mc H_j^\infty$. A dense subspace $\Dom_0$ of $\mc H_i^\infty$ is called \textbf{quasi-rotation invariant} (\textbf{QRI} for short), if there exists $\delta>0$ and a dense subspace $\Dom_\delta\subset\Dom_0$ such that $\varrho(t)\Dom_\delta\subset\Dom_0$ for any $t\in(-\delta,\delta)$. For example, $\mc H_i^\infty(I)$ is QRI for any $I\in\mc J$. We say that $T$ is \textbf{localizable} if any dense and QRI subspace $\Dom_0$ of $\mc H_i^\infty$ is a core for $T$. By saying that $T:\mc H_i\rightarrow\mc H_j$ is smooth and localizable, we always assume that $T$ is closable and has domain $\mc H_i^\infty$.

Suppose that in the following diagrams, $A,B,C,D$ are smooth. (Namely, the formal adjoints of these operators should also map smooth spaces to smooth ones.) 
\begin{align}
\begin{CD}
\mc H_i^\infty @>C>> \mc H_j^\infty\\
@V A VV @V B VV\\
\mc H_k^\infty @>D>> \mc \mc H_l^\infty.
\end{CD}
\end{align}
We say that this diagram \textbf{commutes adjointly} if $DA=BC$ on $\mc H_i^\infty$, and if $CA^\dagger=B^\dagger D$ on $\mc H_k^\infty$ (equivalently, $AC^\dagger=D^\dagger B$ on $\mc H_j^\infty$).

\begin{df}\label{lb30}
Let $\fk H$ assign, to each $\wtd I\in\Jtd$ and $\mc H_i\in\Obj(\scr C)$, a set $\fk H_i(\wtd I)$  such that $\fk H_i(\wtd I_1)\subset\fk H_i(\wtd I_2)$  whenever $\wtd I_1\subset\wtd I_2$. A \textbf{weak categorical extension} $\scr E^w=(\mc A,\scr C,\boxdot,\fk H)$ of $\mc A$ associates, to any $\mc H_i,\mc H_k\in\Obj(\scr C),\wtd I\in\Jtd,\fk a\in\fk H_i(\wtd I)$, smooth and localizable operators
\begin{gather*}	
\mc L(\fk a,\wtd I):\mc H_k\rightarrow\mc H_i\boxdot\mc H_k,\\
\mc R(\fk a,\wtd I):\mc H_k\rightarrow\mc H_k\boxdot\mc H_i,
\end{gather*}
such that for any $\mc H_i,\mc H_j,\mc H_k,\mc H_{k'}\in\Obj(\scr C)$ and any $\wtd I,\wtd J,\wtd I_1,\wtd I_2\in\Jtd$ with $\wtd I_1\subset\wtd I_2$, the following conditions are satisfied:

	(a) (Isotony) If  $\fk a\in\fk H_i(\wtd I_1)$, then $\mc L(\fk a,\wtd I_1)=\mc L(\fk a,\wtd I_2)$, $\mc R(\fk a,\wtd I_1)=\mc R(\fk a,\wtd I_2)$ when acting on $\mc H_k^\infty$.
	
	(b) (Naturality) If  $F\in\Hom_{\mc A}(\mc H_k,\mc H_{k'})$,  the following diagrams commute\footnote{Note that they also commute adjointly and hence strongly since we have $(F\boxdot G)^*=F^*\boxdot G^*$ for any morphisms $F,G$ in a $C^*$-tensor category.}   for any $\fk a\in\fk H_i(\wtd I)$.
	\begin{gather}
	\begin{CD}
	\mc H_k^\infty @>F>> \mc H_{k'}^\infty\\
	@V \mc L(\fk a,\wtd I)  VV @V \mc L(\fk a,\wtd I)  VV\\
	(\mc H_i\boxdot\mc H_k)^\infty @> \id_i\boxdot F>> (\mc H_i\boxdot\mc H_{k'})^\infty
	\end{CD}\qquad\qquad
	\begin{CD}
	\mc H_k^\infty @> \mc R(\fk a,\wtd I)  >> (\mc H_k\boxdot\mc H_i)^\infty\\
	@V F VV @V F\boxdot\id_i  VV\\
	\mc H_{k'}^\infty @>\mc R(\fk a,\wtd I) >> (\mc H_{k'}\boxdot\mc H_i)^\infty
	\end{CD}.
	\end{gather}
	
	(c) (Neutrality) Under the identifications $\mc H_i=\mc H_i\boxdot\mc H_0=\mc H_0\boxdot\mc H_i$, for any $\fk a\in\fk H_i(\wtd I)$ we have
	\begin{align}
	\mc L(\fk a,\wtd I)|_{\mc H_0^\infty}=\mc R(\fk a,\wtd I)|_{\mc H_0^\infty}.	
	\end{align}
	
	(d) (Reeh-Schlieder property) Under the identification $\mc H_i=\mc H_i\boxdot\mc H_0$, the set $\mc L(\fk H_i(\wtd I),\wtd I)\Omega$ spans a dense subspace of $\mc H_i$.
	
	(e) (Density of fusion products) The set $\mc L(\fk H_i(\wtd I),\wtd I)\mc H_k^\infty$ spans a dense subspace of $\mc H_i\boxdot\mc H_k$, and $\mc R(\fk H_i(\wtd I),\wtd I)\mc H_k^\infty$ spans a dense subspace of $\mc H_k\boxdot\mc H_i$.
	
(f) (Intertwining property) 	For any $\fk a\in\fk H_i(\wtd I)$ and $x\in\mc A^\infty(I')$, the following diagrams commute adjointly:
\begin{gather}
\begin{CD}
\mc H_k^\infty @>~\pi_{k,I'}(x)~>> \mc H_k^\infty\\
@V \mc L(\fk a,\wtd I)  VV @VV \mc L(\fk a,\wtd I) V\\
(\mc H_i\boxdot\mc H_k)^\infty @> \pi_{i\boxdot k,I'}(x)>> (\mc H_i\boxdot\mc H_k)^\infty
\end{CD}\qquad\qquad\qquad
\begin{CD}
\mc H_k^\infty @> \mc R(\fk a,\wtd I)  >> (\mc H_k\boxdot\mc H_i)^\infty\\
@V \pi_{k,I'}(x) VV @VV \pi_{k\boxdot i,I'}(x) V\\
\mc H_k^\infty @>\mc R(\fk a,\wtd I) >> (\mc H_k\boxdot\mc H_i)^\infty
\end{CD}
\end{gather}

	(g) (Weak locality) Assume that $\wtd I$ is anticlockwise to $\wtd J$. Then for  any $\fk a\in\fk H_i(\wtd I),\fk b\in\fk H_j(\wtd J)$, the following diagram   commutes adjointly.
	\begin{align}
	\begin{CD}
	\mc H_k^\infty @> \quad \mc R(\fk b,\wtd J)\quad   >> (\mc H_k\boxdot\mc H_j)^\infty\\
	@V \mc L(\fk a,\wtd I)   V  V @V \mc L(\fk a,\wtd I) VV\\
	(\mc H_i\boxdot\mc H_k)^\infty @> \quad \mc R(\fk b,\wtd J) \quad  >> (\mc H_i\boxdot\mc H_k\boxdot\mc H_j)^\infty
	\end{CD}
	\end{align}
	
	(h) (Braiding) There is a unitary linear map $\ss_{i,k}:\mc H_i\boxdot\mc H_k\rightarrow\mc H_k\boxdot \mc H_i$ (the \textbf{braid operator}) such that for any  $\fk a\in\fk H_i(\wtd I)$ and $\eta\in\mc H_k^\infty$,
	\begin{align}
	\ss_{i,k} \mc L(\fk a,\wtd I)\eta=\mc R(\fk a,\wtd I)\eta.
	\end{align}
	
	(i) (Rotation covariance) 	For any $\fk a\in\fk H_i(\wtd I)$ and $g=\varrho(t)$ where $t\in\mbb R$, there exists an element $g\fk a$ inside $\fk H_i(g\wtd I)$, such that for any $\mc H_l\in\Obj(\scr C)$ and $\eta\in\mc H_l^\infty$, the following two equivalent equations are true.
	\begin{gather}
	\mc L(g\fk a,g\wtd I)\eta=g\mc L(\fk a,\wtd I)g^{-1}\eta;\label{eq27}\\
	\mc R(g\fk a,g\wtd I)\eta=g\mc R(\fk a,\wtd I)g^{-1}\eta.\label{eq28}
	\end{gather}
If $g\mc H_k^\infty\subset\mc H_k^\infty$ for each $\mc H_k\in\Obj(\scr C)$ and $g\in\UPSU$, and if the statements in (i) are true for any $g\in\UPSU$, we say that $\scr E^w$ is \textbf{M\"obius covariant}.
\end{df}

The axiom of weak locality explains the adjective ``weak" in the notion of weak categorical extensions. However, as we will see Cor. \ref{lb35}, a weak categorical extension $\scr E^w$ actually satisfies strong locality. Moreover, if we let $\scr E=(\mc A,\Rep(\mc A),\boxtimes,\mc H)$ be a closed and vector-labeled categorical extension, then Thm. \ref{lb34} will imply that $\scr E^w$ can be embedded into the ``unbounded version of $\scr E$" (i.e., $\scr E$ together with the closed operators $\scr L,\scr R$ as described in Sec. \ref{lb11}) so that the latter can be viewed as the maximal extension of $\scr E^w$.

\begin{rem}
In Def. \ref{lb30}, note that for any $\fk a\in\fk H_i(\wtd I)$, $\mc L(\fk a,\wtd I)$ and $\mc R(\fk a,\wtd I)$ depend only on the vector $\xi=\mc L(\fk a,\wtd I)\Omega=\mc R(\fk a,\wtd I)\Omega$. Indeed, choose any $\wtd J$ clockwise to $\wtd I$. Then by weak locality, we know that the action of $\mc L(\fk a,\wtd I)$ on $\mc R(\fk H_j(J),\wtd J)\Omega$  is determined by $\xi$. By Reeh-Schlieder property, $\mc R(\fk H_j(J),\wtd J)\Omega$ spans a dense linear subspace of $\mc H_j$. By rotation covariance, this dense subspace is QRI. Thus it is a core for $\mc L(\fk a,\wtd I)$. Therefore the action of $\mc L(\fk a,\wtd I)$ (and hence of $\mc R(\fk a,\wtd I)$) on $\mc H_j^\infty$ is completely determined by $\fk a$. Moreover, if $\fk a_1,\fk a_2\in\fk H_i(\wtd I)$ and $c_1,c_2\in\mbb C$, we may add a formal linear combination $c_1\fk a_1+c_2\fk a_2$ to the set $\fk H_i(\wtd I)$, define
\begin{align*}
\mc L(c_1\fk a_1+c_2\fk a_2,\wtd I)=c_1\mc L(\fk a_1,\wtd I)+c_2\mc L(\fk a_2,\wtd I),
\end{align*}
and define $\mc R$ similarly. Then the axioms of a weak categorical extension are still satisfied. We may therefore assume that $\scr E^w$ is \textbf{vector-labeled}, which means that each $\fk H_i(\wtd I)$ is a (dense) subspace of $\mc H_i$, 
and that for any $\fk a\in\fk H_i(\wtd I)$, we have $\mc L(\fk a,\wtd I)\Omega=\fk a$.\footnote{In \cite{Gui21a}, we only assume $\fk H_i(\wtd I)$ to be a subset of $\mc H_i$.} We will write elements in $\fk H_i(\wtd I)$ using Greek letters $\xi,\eta,\dots$ in the future.
\end{rem}

In the above discussion, we have actually showed:

\begin{pp}\label{lb23}
Let $\scr E^w$ be vector-labeled. Assume that $\wtd J$ is clockwise to $\wtd I$. Then for any $\mc H_i,\mc H_j\in\Obj(\RepA)$ and any $\xi\in\fk H_i(\wtd I)$, the vector space $\fk H_j(\wtd J)$ is a core for $\mc L(\xi,\wtd I)|_{\mc H_j^\infty}$ and $\mc R(\xi,\wtd I)|_{\mc H_j^\infty}$.
\end{pp}

\begin{co}\label{lb22}
Let $\scr E^w$ be vector-labeled. Assume that $\wtd J$ is clockwise to $\wtd I$, and choose $\mc H_i,\mc H_j\in\Obj(\RepA)$. Then  $\mc L(\fk H_i(\wtd I),\wtd I)\fk H_j(\wtd J)$ (resp. $\mc R(\fk H_i(\wtd I),\wtd I)\fk H_j(\wtd J)$) spans a dense subspace of $\mc H_i\boxdot\mc H_j$ (resp. $\mc H_j\boxdot\mc H_i$).
\end{co}
\begin{proof}
Use the above proposition and the density of fusion products.
\end{proof}

Let us write the closure of $\mc L(\xi,\wtd I)|_{\mc H_j^\infty}$ as $\ovl{\mc L(\xi,\wtd I)}|_{\mc H_j}=\ovl{\mc L(\xi,\wtd I)|_{\mc H_j^\infty}}$.  We close this section with the following important observation.

\begin{pp}\label{lb27}
Let $\scr E^w$ be vector-labeled. Then $\fk H_i(\wtd I)\subset\mc H_i^w(I)$. Moreover, for any $\xi\in\fk H_i(\wtd I)$, we have
\begin{align}
\ovl{\mc L(\xi,\wtd I)}|_{\mc H_0}=\ovl{\mc R(\xi,\wtd I)}|_{\mc H_0}=\scr L(\xi,\wtd I)|_{\mc H_0}=\scr R(\xi,\wtd I)|_{\mc H_0}\label{eq25}
\end{align}
\end{pp}

\begin{proof}
Choose $\xi\in\fk H_i(\wtd I)$. We temporarily let the domain of $\scr L(\xi,\wtd I)|_{\mc H_0}$ be $\mc A(I')\Omega$ since we haven't showed that $\scr L(\xi,\wtd I)|_{\mc H_0}$ is closable. By the intertwining property of $\scr E^w$, $\ovl{\mc L(\xi,\wtd I)}|_{\mc H_0}$ commutes strongly with the actions of $\mc A(I')$ (since $\mc A(I')$ is generated by $\mc A^\infty(I')$). Therefore, for any $y\in\mc A(I')$, we have $\ovl{\mc L(\xi,\wtd I)}y\Omega=y\ovl{\mc L(\xi,\wtd I)}\Omega=y\xi=R(y\Omega,\wtd I')\xi=\scr L(\xi,\wtd I)y\Omega$. This proves that $\scr L(\xi,\wtd I)|_{\mc H_0}\subset\ovl{\mc L(\xi,\wtd I)}|_{\mc H_0}$ and hence that $\scr L(\xi,\wtd I)|_{\mc H_0}$ is closable. We now assume $\scr L(\xi,\wtd I)$ to be closed by taking the closure. Since $\mc A^\infty(I')\Omega=\mc H_0^\infty(I)$ is dense and QRI, $\mc A^\infty(I')\Omega$ is a core for $\ovl{\mc L(\xi,\wtd I)}|_{\mc H_0}$. Therefore $\scr L(\xi,\wtd I)|_{\mc H_0}=\ovl{\mc L(\xi,\wtd I)}|_{\mc H_0}$. In particular, $\scr L(\xi,\wtd I)|_{\mc H_0}$ is smooth. Thus $\xi\in\mc H_i^w(I)$. A similar argument shows $\ovl{\mc R(\xi,\wtd I)}|_{\mc H_0}=\scr R(\xi,\wtd I)|_{\mc H_0}$.
\end{proof}

\begin{rem}
In section \ref{lb29} we defined vector-labeled and closed categorical extensions of $\mc A$. In general, one can define a (strong) categorical extension $\scr E=(\mc A,\scr C,\boxdot,\fk H)$  without assuming $\scr E$ to be vector-labeled or closed (see \cite{Gui21a} section 3.1). The definition is the very similar to definition \ref{lb30}, but with the following differences: (1) All the $\mc L$ and $\mc R$ operators are bounded operators and are written as $L$ and $R$ instead. (2) The smooth vector spaces $\mc H_i^\infty,\mc H_j^\infty$, etc. are replaced by the full Hilbert spaces $\mc H_i,\mc H_j$, etc. (3) Rotation covariance is not required. 

It was shown in \cite{Gui21a} section 3.5 that $\scr E$ can always be extended to a unique vector-labeled and closed categorical extension. By saying that $\scr E=(\mc A,\scr C,\boxdot,\fk H)$ is vector-labeled and closed, we mean that each $\fk H_i(\wtd I)$ equals the vector space $\mc H_i(I)$, and  that $L(\fk a,\wtd I)\Omega=\fk a$ for any $\fk a\in\mc H_i(I)$. 
\end{rem}

\subsection{From weak to strong categorical extensions}\label{lb70}

Let $\scr C$ be as in the last section.  Let $\wht {\scr C}$ be the $C^*$-category of all $\mc A$-modules $\mc H_i$ such that $\mc H_i$ is  unitarily equivalent to some object in $\scr C$. 

\begin{thm}\label{lb34}
Let $\scr E^w=(\mc A,\scr C,\boxdot,\fk H)$ be a vector-labeled weak categorical extension of $\mc A$ with braid operator $\ss$. Then $\wht{\scr C}$ is  closed under Connes fusion product $\boxtimes$, and there is a unique unitary natural \footnote{Naturality means that $\Psi_{i',j'}(F\boxdot G)=(F\boxtimes G)\Psi_{i,j}$ when $F\in\Hom_{\mc A}(\mc H_i,\mc H_{i'})$ and $G\in\Hom_{\mc A}(\mc H_j,\mc H_{j'})$.} isomorphism
\begin{align}
\Psi=\Psi_{i,j}:\mc H_i\boxdot\mc H_j\rightarrow\mc H_i\boxtimes\mc H_j\qquad(\forall \mc H_i,\mc H_j\in\Obj(\scr C)),
\end{align}
such that for any $\wtd I\in\Jtd,\mc H_i,\mc H_j\in\Obj(\scr C)$,
\begin{gather}
\Psi_{i,j}\ovl{\mc L(\xi,\wtd I)}|_{\mc H_j}=\scr L(\xi,\wtd I)|_{\mc H_j},\label{eq20}\\
\Psi_{j,i}\ovl{\mc R(\xi,\wtd I)}|_{\mc H_j}=\scr R(\xi,\wtd I)|_{\mc H_j}.\label{eq21}
\end{gather}

Moreover, $\ss$ is natural and satisfies the Hexagon axioms; $\Psi$ induces an equivalence of braided $C^*$-tensor categories $(\wht{\scr C},\boxtimes,\mathbb B)\simeq(\scr C,\boxdot,\ss)$, which means that for any $\mc H_i,\mc H_j,\mc H_k\in\Obj(\scr C)$, the following conditions are satisfied:\\
	(a) The following diagram commutes.
	\begin{gather}
	\begin{CD}
	\mc H_i\boxdot\mc H_k\boxdot\mc H_j @>~~\quad \Psi_{i\boxdot k,j} \quad~~>> (\mc H_i\boxdot \mc H_k)\boxtimes\mc H_j\\
	@V \Psi_{i,k\boxdot j} VV @V \Psi_{i,k}\boxtimes\id_j VV\\
	\mc H_i\boxtimes(\mc H_k\boxdot\mc H_j) @> \quad \id_i\boxtimes\Psi_{k,j} \quad>> \mc H_i\boxtimes\mc H_k\boxtimes\mc H_j
	\end{CD}\label{eq24}
	\end{gather}
	(b) The following two maps equal $\id_i:\mc H_i\rightarrow\mc H_i$.
	\begin{gather}
	\mc H_i\simeq \mc H_0\boxdot\mc H_i\xrightarrow{\Psi_{0,i}} \mc H_0\boxtimes\mc H_i\simeq\mc H_i,\\
	\mc H_i\simeq \mc H_i\boxdot\mc H_0\xrightarrow{\Psi_{i,0}} \mc H_i\boxtimes\mc H_0\simeq\mc H_i. 
	\end{gather}
	(c) The following diagram commutes.
	\begin{gather}
	\begin{CD}
	\mc H_i\boxdot\mc H_j @> \quad\ss_{i,j} \quad>> \mc H_j\boxdot\mc H_i\\
	@V \Psi_{i,j} VV @V \Psi_{j,i} VV\\
	\mc H_i\boxtimes\mc H_j @>\quad\mbb B_{i,j}\quad>> \mc H_j\boxtimes\mc H_i
	\end{CD}
	\end{gather}
\end{thm}

\begin{proof}

Choose any $\wtd I$ and any clockwise $\wtd J$, and choose $\mc H_i,\mc H_j\in\Obj(\scr C)$. Then for any $\xi_1,\xi_2\in\fk H_i(\wtd I)$ and $\eta_1,\eta_2\in\fk H_j(\wtd J)$, we use the weak locality of $\scr E^w$ to compute
\begin{align}
&\bk{\mc L(\xi_1,\wtd I)\eta_1|\mc L(\xi_2,\wtd I)\eta_2}=\bk{\mc L(\xi_1,\wtd I)\mc R(\eta_1,\wtd J)\Omega|\mc L(\xi_2,\wtd I)\mc R(\eta_2,\wtd J)\Omega}\nonumber\\
=&\bk{\mc L(\xi_2,\wtd I)^\dagger\mc L(\xi_1,\wtd I)\Omega|\mc R(\eta_1,\wtd J)^\dagger\mc R(\eta_2,\wtd J)\Omega}=\bk{\mc L(\xi_2,\wtd I)^*\xi_1|\mc R(\eta_1,\wtd J)^*\eta_2}.\label{eq17}
\end{align}
In the last term of the above equations, $\mc L(\xi_2,\wtd I)$ and $\mc R(\eta_1,\wtd J)$ are mapping from $\mc H_0^\infty$. Thus, by proposition \ref{lb27},  \eqref{eq17} equals $\bk{\scr L(\xi_2,\wtd I)^*\xi_1|\scr R(\eta_1,\wtd J)^*\eta_2}$. By corollary \ref{lb5}, it also equals $\bk{\scr L(\xi_1,\wtd I)\eta_1|\scr L(\xi_2,\wtd I)\eta_2}$. (Note that by corollary \ref{lb75}, $\eta_1,\eta_2$ are in the domains of $\scr L(\xi_1,\wtd I)|_{\mc H_j}$ and $\scr L(\xi_2,\wtd I)|_{\mc H_j}$.) We conclude
\begin{align*}
\bk{\mc L(\xi_1,\wtd I)\eta_1|\mc L(\xi_2,\wtd I)\eta_2}=\bk{\scr L(\xi_1,\wtd I)\eta_1|\scr L(\xi_2,\wtd I)\eta_2}.
\end{align*}
This equation, together with corollary \ref{lb22}, shows that there is a well defined isometry
\begin{gather*}
\Psi_{i,j}^{\wtd I,\wtd J}:\mc H_i\boxdot\mc H_j\rightarrow\mc H_i\boxtimes\mc H_j,\qquad \mc L(\xi,\wtd I)\eta\mapsto\scr L(\xi,\wtd I)\eta.
\end{gather*}
By \eqref{eq18} and the weak locality of $\scr E^w$, $\Psi_{i,j}^{\wtd I,\wtd J}$ sends $\mc R(\eta,\wtd J)\xi=\mc L(\xi,\wtd I)\eta$ to $\scr R(\eta,\wtd J)\xi=\scr L(\xi,\wtd I)\eta$. Therefore, we have a collection of isometries $\{\Psi^{\wtd I,\wtd J}\}_{\wtd I,\wtd J}$ satisfying \begin{gather*}
\Psi_{i,j}^{\wtd I,\wtd J}\mc L(\xi,\wtd I)\eta=\scr L(\xi,\wtd I)\eta=\Psi_{i,j}^{\wtd I,\wtd J}\mc R(\eta,\wtd J)\xi=\scr R(\eta,\wtd J)\xi
\end{gather*}
for any $\mc H_i,\mc H_j\in\Obj(\scr C),\xi\in\fk H_i(\wtd I),\eta\in\fk H_j(\wtd J)$.

Choose another pair $(\wtd I_0,\wtd J_0)$ with $\wtd J_0$ clockwise to $\wtd I_0$. We shall show that $\Psi_{i,j}^{\wtd I,\wtd J}=\Psi_{i,j}^{\wtd I_0,\wtd J_0}$. This is clearly true when $\wtd I_0\subset\wtd I,\wtd J_0\subset\wtd J$. Thus it is also true when $(\wtd I_0,\wtd J_0)$ is close to $(\wtd I,\wtd J)$, in the sense that there is a pair $(\wtd I_1,\wtd J_1)$ such that $\wtd I_1\subset \wtd I,\wtd I_0$ and $\wtd J_1\subset\wtd J,\wtd J_0$. In general, we can find a chain of pairs $(\wtd I_1,\wtd J_1),\dots,(\wtd I_{n-1},\wtd J_{n-1}),(\wtd I_n,\wtd J_n)=(\wtd I,\wtd J)$ such that $(\wtd I_{m-1},\wtd J_{m-1})$ is close to $(\wtd I_m,\wtd J_m)$ for each $1\leq m\leq n$. This shows $\Psi_{i,j}^{\wtd I,\wtd J}=\Psi_{i,j}^{\wtd I_0,\wtd J_0}$. Thus $\Psi_{i,j}^{\wtd I,\wtd J}$ is independent of $\wtd I,\wtd J$, and we shall write it as $\Psi_{i,j}$ in the following.

Choose any $\wtd I\in\Jtd$. We now show that for any $\xi\in\fk H_i(\wtd I)$, \eqref{eq20} is true. Set $\wtd J=\wtd I'$. Then, from our definition of $\Psi_{i,j}$, the  two operators in \eqref{eq20} are clearly equal when acting on $\fk H_j(\wtd J)$. By proposition \ref{lb23}, $\fk H_j(\wtd J)$ is a core for $\Psi_{i,j}\ovl{\mc L(\xi,\wtd I)}$. Therefore $\Psi_{i,j}\ovl{\mc L(\xi,\wtd I)}|_{\mc H_j}\subset \scr L(\xi,\wtd I)|_{\mc H_j}$. Note that $\mc H_j^\infty(J)$ is included in the domains of both operators. By proposition \ref{lb21}, $\mc H_j^\infty(J)$ is a core for $\scr L(\xi,\wtd I)|_{\mc H_j}$. So we must have $\Psi_{i,j}\ovl{\mc L(\xi,\wtd I)}|_{\mc H_j}= \scr L(\xi,\wtd I)|_{\mc H_j}$.   Thus  \eqref{eq20} is proved. A similar argument proves \eqref{eq21}. Since $\scr L(\fk H_i(\wtd I),\wtd I)\mc H_j(J)$ equals $R(\mc H_j(J),\wtd J)\fk H_i(\wtd I)$, it is dense in $R(\mc H_j(J),\wtd J)\mc H_i$ which spans a dense subspace of $\mc H_i\boxtimes\mc H_j$. Therefore $\Psi_{i,j}$ must be surjective. This proves that $\Psi_{i,j}$ is unitary.

Next, we check that $\Psi_{i,j}$ is an isomorphism by checking that $\Psi_{i,j}$ commutes with the actions of $\mc A$. By corollary \ref{lb24}, it suffices to check that $\Psi_{i,j}\pi_{i\boxdot j,K}(x)=\pi_{i\boxtimes j,K}(x)\Psi_{i,j}$ for any $K\in\mc J$ and $x\in\mc A^\infty(K)$.  Choose $\wtd I,\wtd J\in\Jtd$ such that $\wtd J$ is clockwise to $\wtd I$ and that both $I$ and $J$ are disjoint from $K$. We need to show that for any $\xi\in\fk H_i(\wtd I)$ and $\eta\in\fk H_j(\wtd J)$,
\begin{align}
\Psi x\mc L(\xi,\wtd I)\eta=x\scr L(\xi,\wtd I)\eta.\label{eq23}
\end{align}
Set $\chi=x\Omega$. By the  intertwining property of $\scr E^w$, we have $x\mc L(\xi,\wtd I)\eta=\mc L(\xi,\wtd I)\mc R(\eta,\wtd J)x\Omega=\mc L(\xi,\wtd I)\mc R(\eta,\wtd J)\chi$. By proposition \ref{lb18}, we also have $x\scr L(\xi,\wtd I)\eta=\scr L(\xi,\wtd I)\scr R(\eta,\wtd J)\chi$. We now choose any $\xi_0\in\fk H_i(\wtd I),\eta_0\in\fk H_j(\wtd J)$, and compute
\begin{align}
&\bk{\Psi x\mc L(\xi,\wtd I)\eta|\scr L(\xi_0,\wtd I)\eta_0}=\bk{\Psi x\mc L(\xi,\wtd I)\eta|\Psi\mc L(\xi_0,\wtd I)\eta_0}=\bk{x\mc L(\xi,\wtd I)\eta|\mc L(\xi_0,\wtd I)\eta_0}\nonumber\\
=&\bk{\mc L(\xi,\wtd I)\mc R(\eta,\wtd J)\chi|\mc L(\xi_0,\wtd I)\eta_0}\xlongequal{\text{Weak locality of $\scr E^w$}}\bk{\mc L(\xi_0,\wtd I)^*\mc L(\xi,\wtd I)\chi|\mc R(\eta,\wtd J)^*\eta_0}.\label{eq22}
\end{align}
In the last term of \eqref{eq22}, the operators $\mc L(\xi_0,\wtd I)$, $\mc L(\xi,\wtd I)$, and $\mc R(\eta,\wtd J)$ are all acting on $\mc H_0^\infty$. Thus, by proposition \ref{lb27}, \eqref{eq22} equals
\begin{align*}
\bk{\scr L(\xi_0,\wtd I)^*\scr L(\xi,\wtd I)\chi|\scr R(\eta,\wtd J)^*\eta_0}=\bk{\scr L(\xi_0,\wtd I)^*\scr L(\xi,\wtd I)\chi|\scr R(\eta,\wtd J)^*\scr R(\eta_0,\wtd J)\Omega},
\end{align*}
which, by proposition \ref{lb26}, equals
\begin{align*}
\bk{\scr L(\xi,\wtd I)\scr R(\eta,\wtd J)\chi|\scr L(\xi_0,\wtd I)\scr R(\eta_0,\wtd J)\Omega}=\bk{x\scr L(\xi,\wtd I)\eta|\scr L(\xi_0,\wtd I)\eta_0}.
\end{align*}
This proves \eqref{eq23} when evaluated by $\scr L(\xi_0,\wtd I)\eta_0=\Psi\mc L(\xi_0,\wtd I)\eta_0$. By corollary \ref{lb22} and by the unitarity of $\Psi$, vectors of the form $\scr L(\xi_0,\wtd I)\eta_0$ span a dense subspace of $\mc H_i\boxtimes\mc H_j$. Therefore \eqref{eq23} is true.

Using \eqref{eq45} and the naturality of $\scr E^w$, it is easy to check that $\Psi(\id\boxdot G)=(\id\boxtimes G)\Psi$ and $\Psi(F\boxdot\id)=(F\boxtimes\id)\Psi$ if $F$ and $G$ are homomorphisms of $\mc A$-modules. This proves the naturality of $\Psi$.

It remains to check that $\Psi$ satisfies (a), (b), (c) of the theorem. This will imply that  $\ss$ is natural and satisfies Hexagon axioms since these are true for $\mbb B$. To begin with, note that (b) is true by \eqref{eq25}. (c) follows from the braiding property in theorem \ref{lb13} and the braiding axiom of $\scr E^w$. To prove (a), we  choose  $\wtd I$ and a clockwise $\wtd J$, and choose any $\mc H_i,\mc H_j,\mc H_k\in\Obj(\scr C)$, $\xi\in\fk H_i(\wtd I),\eta\in\fk H_j(\wtd J),\chi\in\mc H_k^\infty$. Since $\mc R(\eta,\wtd J)\chi$ (which is in $(\mc H_k\boxdot\mc H_j)^\infty$) is in the domain of $\mc L(\xi,\wtd I)$, by \eqref{eq20}, it is also in the domain of $\scr L(\xi,\wtd I)$. By \eqref{eq21} and that $(\id_i\boxtimes\Psi_{k,j})\scr L(\xi,\wtd I)|_{\mc H_k\boxdot\mc H_j}\subset\scr L(\xi,\wtd I)|_{\mc H_k\boxtimes\mc H_j}\cdot\Psi_{k,j}$, we conclude that  $\scr R(\eta,\wtd J)\chi=\Psi_{k,j}\mc R(\eta,\wtd J)\chi$ is in the domain of  $\scr L(\xi,\wtd I)$. A similar argument shows that $\scr L(\xi,\wtd I)\chi$ is in the domain of $\scr R(\eta,\wtd J)$. Then, using corollary \ref{lb25}, the weak locality of $\scr E^w$, and the naturality of $\mc L,\mc R,\scr L,\scr R$, we compute
\begin{align*}
&\mc L(\xi,\wtd I)\mc R(\eta,\wtd J)\chi\xrightarrow{\Psi_{i,k\boxdot j}} \scr L(\xi,\wtd I)\mc R(\eta,\wtd J)\chi\\
\xrightarrow{\id_i\boxtimes\Psi_{k,j}}&\scr L(\xi,\wtd I)\Psi_{k,j}\mc R(\eta,\wtd J)\chi=\scr L(\xi,\wtd I)\scr R(\eta,\wtd J)\chi ,
\end{align*}
and also
\begin{align*}
&\mc L(\xi,\wtd I)\mc R(\eta,\wtd J)\chi=\mc R(\eta,\wtd J)\mc L(\xi,\wtd I)\chi \xrightarrow{\Psi_{i\boxdot k,j}} \scr R(\eta,\wtd J)\mc L(\xi,\wtd I)\chi\\
\xrightarrow{\Psi_{i,k}\boxtimes\id_j}&\scr R(\eta,\wtd J)\Psi_{i,k}\mc L(\xi,\wtd I)\chi =  \scr R(\eta,\wtd J)\scr L(\xi,\wtd I)\chi=\scr L(\xi,\wtd I)\scr R(\eta,\wtd J)\chi.
\end{align*}
This proves \eqref{eq24} when acting on $\mc L(\fk H_i(\wtd I),\wtd I)\mc R(\fk H_j(\wtd J),\wtd J)\mc H_k^\infty$. By the density axiom of $\scr E^w$, $\mc R(\fk H_j(\wtd J),\wtd J)\mc H_k^\infty$ is a dense subspace of $(\mc H_k\boxdot\mc H_j)^\infty$, and $\mc L(\fk H_i(\wtd I),\wtd I)(\mc H_k\boxdot\mc \mc H_j)^\infty$ is a dense subspace of $\mc H_i\boxdot\mc H_k\boxdot\mc H_j$. By the rotation covariance of $\scr E^w$, $\mc R(\fk H_j(\wtd J),\wtd J)\mc H_k^\infty$ is QRI. Thus it is a core for $\mc L(\xi,\wtd I)$ for each $\xi\in\fk H_i(\wtd I)$. It follows that $\mc L(\fk H_i(\wtd I),\wtd I)\mc R(\fk H_j(\wtd J),\wtd J)\mc H_k^\infty$ is dense in $\mc L(\fk H_i(\wtd I),\wtd I)(\mc H_k\boxdot\mc \mc H_j)^\infty$. Hence it is also dense in $\mc H_i\boxdot\mc H_k\boxdot\mc H_j$. Therefore \eqref{eq24} is true. We've proved (a).
\end{proof}

\begin{co}\label{lb35}
Let $\scr E^w=(\mc A,\scr C,\boxdot,\fk H)$ be a  weak categorical extension of $\mc A$. Then for any $\mc H_i,\mc H_j,\mc H_k\in\Obj(\scr C),\fk a\in\fk H_i(\wtd I),\fk b\in\fk H_j(\wtd J)$, and any $\wtd I,\wtd J\in\Jtd$ with $\wtd J$ clockwise to $\wtd I$, the following diagram of closable operators commutes strongly.
\begin{align}
\begin{CD}
\mc H_k @> \quad \mc R(\fk b,\wtd J)\quad   >> \mc H_k\boxdot\mc H_j\\
@V \mc L(\fk a,\wtd I)   V  V @V \mc L(\fk a,\wtd I) VV\\
\mc H_i\boxdot\mc H_k @> \quad \mc R(\fk b,\wtd J) \quad  >> \mc H_i\boxdot\mc H_k\boxdot\mc H_j
\end{CD}\label{eq19}
\end{align}
\end{co}

\begin{proof}
The proof is similar to \cite{Gui21a} theorem 3.12. Assume without loss of generality that $\scr E^w$ is vector-labeled. Write $\xi=\fk a,\eta=\fk b$. Then by theorem \ref{lb34} and by the locality and the naturality in theorem \ref{lb13}, each of the following four diagrams commutes strongly.
\begin{align}
\begin{CD}
\mc H_k @> \quad \scr R(\eta,\wtd J)\quad   >> \mc H_k\boxtimes\mc H_j  @> ~~\quad\Psi^{-1}_{k,j}\quad~~ >> \mc H_k\boxdot\mc H_j \\
@V \scr L(\xi,\wtd I)   V  V @V \scr L(\xi,\wtd I) VV   @V  \scr L(\xi,\wtd I) VV\\
\mc H_i\boxtimes\mc H_k @> \quad \scr R(\eta,\wtd J) \quad  >> \mc H_i\boxtimes\mc H_k\boxtimes\mc H_j  @> \quad\id_i\boxtimes\Psi^{-1}_{k,j}\quad >>\mc H_i\boxtimes(\mc H_k\boxdot\mc H_j)\\
@ V \Psi^{-1}_{i,k} VV  @V \Psi^{-1}_{i,k}\boxtimes\id_j VV  @V \Psi^{-1}_{i,k\boxdot j}VV\\
\mc H_i\boxdot\mc H_k  @> \quad \scr R(\eta,\wtd J)\quad >> (\mc H_i\boxdot \mc H_k)\boxtimes\mc H_j  @> ~~\quad\Psi^{-1}_{i\boxdot k,j}\quad~~ >> \mc H_i\boxdot\mc H_k\boxdot\mc H_j.
\end{CD}\label{eq43}
\end{align}
Define a unitary isomorphism $\Psi_{i,k,j}:\mc H_i\boxdot\mc H_k\boxdot\mc H_j\rightarrow \mc H_i\boxtimes\mc H_k\boxtimes\mc H_j$ to be $\Psi_{i,k,j}=(\id_i\boxtimes\Psi_{k,j})\Psi_{i,k\boxdot j}=(\Psi_{i,k}\boxtimes\id_j)\Psi_{i\boxdot k,j}$. Identify $\mc H_i\boxdot\mc H_k,\mc H_k\boxdot\mc H_j,\mc H_i\boxdot\mc H_k\boxdot\mc H_j$ with $\mc H_i\boxtimes\mc H_k,\mc H_k\boxtimes\mc H_j,\mc H_i\boxtimes\mc H_k\boxtimes\mc H_j$ via $\Psi_{i,k},\Psi_{k,j},\Psi_{i,k,j}$ respectively. Then, after taking closures, \eqref{eq19} becomes the upper left diagram of \eqref{eq43}, which commutes strongly.
\end{proof}

\subsection{Weak left and right operators}\label{lb48}

Let $\scr C$ be as in section \ref{lb33}, and let $\scr E^w=(\mc A,\scr C,\boxdot,\fk H)$ be a weak categorical extension of $\mc A$.

\begin{df}\label{lb39}
	A \textbf{weak left operator} of $\scr E^w$ is a quadruple $(A,\fk x,\wtd I,\mc H_i)$, where $\fk x$ is a formal element, $\mc H_i\in\Obj(\scr C),\wtd I\in\Jtd$, and for any $\mc H_k\in\Obj(\scr C)$, there is a smooth and localizable operator $A(\fk x,\wtd I):\mc H_k^\infty\rightarrow(\mc H_i\boxdot\mc H_k)^\infty$ such that the following conditions are satisfied:\\
	(a) If $\mc H_k,\mc H_{k'}\in\Obj(\scr C)$, $G\in\Hom_{\mc A}(\mc H_k,\mc H_{k'})$, then  the following diagram commutes.
	\begin{gather}
	\begin{CD}
	\mc H_k^\infty @>G>> \mc H_{k'}^\infty\\
	@V A(\fk x,\wtd I)  VV @V A(\fk x,\wtd I)  VV\\
	(\mc H_i\boxdot\mc H_k)^\infty @> \id_i\boxdot G>> (\mc H_i\boxdot\mc H_{k'})^\infty
	\end{CD}
	\end{gather}
	(b) For any $\mc H_l,\mc H_k\in\Obj(\scr C)$,  $\wtd J\in\Jtd$  clockwise to $\wtd I$, and any $\fk b\in\fk H_l(\wtd J)$, the following diagram  commutes (non-necessarily adjointly).
	\begin{align}
	\begin{CD}
	\mc H_k^\infty @> \quad \mc R(\fk b,\wtd J)\quad   >> (\mc H_k\boxdot\mc H_l)^\infty\\
	@V A(\fk x,\wtd I)   V  V @V A(\fk x,\wtd I) VV\\
	(\mc H_i\boxdot\mc H_k)^\infty @> \quad \mc R(\fk b,\wtd J) \quad  >> (\mc H_i\boxdot\mc H_k\boxdot\mc H_l)^\infty
	\end{CD}
	\end{align}
	
	Similarly, a \textbf{weak right operator} of $\scr E^w$ is a quadruple $(B,\fk y,\wtd J,\mc H_j)$, where $\fk y$ is a formal element, $\mc H_j\in\Obj(\scr C),\wtd J\in\Jtd$, and for any $\mc H_k\in\Obj(\scr C)$, there is a smooth and localizable operator $B(\fk y,\wtd J):\mc H_k^\infty\rightarrow(\mc H_k\boxdot\mc H_j)^\infty$, such that the following conditions are satisfied:\\
	(a) If $\mc H_k,\mc H_{k'}\in\Obj(\scr C)$, $G\in\Hom_{\mc A}(\mc H_k,\mc H_{k'})$, then  the following diagram commutes.
	\begin{gather}
	\begin{CD}
	\mc H_k^\infty  @> B(\fk y,\wtd J)  >> (\mc H_k\boxdot\mc H_j)^\infty\\
	@V G VV @V G\boxdot\id_j  VV\\
	\mc H_{k'}^\infty @>B(\fk y,\wtd J) >> (\mc H_{k'}\boxdot\mc H_j)^\infty
	\end{CD}
	\end{gather}
	(b) For any $\mc H_l,\mc H_k\in\Obj(\scr C)$,  $\wtd I\in\Jtd$  anticlockwise to $\wtd J$, and any $\fk a\in\fk H_l(\wtd I)$, the following diagram  commutes.
	\begin{align}
	\begin{CD}
	\mc H_k^\infty @> \quad B(\fk y,\wtd J)\quad   >> (\mc H_k\boxdot\mc H_j)^\infty\\
	@V \mc L(\fk a,\wtd I)   V  V @V \mc L(\fk a,\wtd I) VV\\
(\mc H_l\boxdot\mc H_k)^\infty @> \quad B(\fk y,\wtd J) \quad  >> (\mc H_l\boxdot\mc H_k\boxdot\mc H_j)^\infty
	\end{CD}
	\end{align}
\end{df}

Recall the natural unitary isomorphism $\Psi=\Psi_{i,j}:\mc H_i\boxdot\mc H_j\rightarrow\mc H_i\boxtimes\mc H_j$ defined in theorem \ref{lb34}. We first state the main result of this section.

\begin{thm}\label{lb36}
Assume that $\scr C$ is rigid, i.e. any object in $\scr C$ has a dual object in $\scr C$. Assume also that $\scr E^w$ is M\"obius covariant. 

(a) If $(A,\fk x,\wtd I,\mc H_i)$ is a weak left operator of $\scr E^w$ and $\xi=A(\fk x,\wtd I)\Omega$, then $\xi\in\mc H_i^\pr(I)$, and for any $\mc H_k\in\Obj(\scr C)$ we have
\begin{align}
\Psi_{i,k}\ovl{A(\fk x,\wtd I)}|_{\mc H_k}=\scr L(\xi,\wtd I)|_{\mc H_k}.\label{eq29}
\end{align}

(b) If $(B,\fk y,\wtd J,\mc H_j)$ is a weak right operator of $\scr E^w$ and $\eta=B(\fk y,\wtd J)\Omega$, then $\eta\in\mc H_j^\pr(J)$, and for any $\mc H_k\in\Obj(\scr C)$ we have
\begin{align}
\Psi_{k,j}\ovl{B(\fk y,\wtd J)}|_{\mc H_k}=\scr R(\eta,\wtd J)|_{\mc H_k}.\label{eq30}
\end{align}
\end{thm}
Since the proofs of (a) and (b) are similar, we shall only prove (a).
\begin{proof}
We first prove that if $\xi\in\mc H_i^\pr(I)$, then \eqref{eq29} is true. So let us assume $\xi=A(\fk x,\wtd I)\Omega$ is in $\mc H_i^\pr(I)$. Assume without loss of generality that $\scr E^w$ is vector-labeled. Since $\fk H_k(\wtd I')$ is a dense and QRI subspace of  $\mc H_k^\infty$, it is a core for $A(\fk x,\wtd I)|_{\mc H_k^\infty}$. For each $\chi\in\fk H_k(\wtd I')$, we compute:
\begin{align}
\Psi_{i,k} A(\fk x,\wtd I)\chi=\Psi_{i,k} A(\fk x,\wtd I)\mc R(\chi,\wtd I')\Omega=\Psi_{i,k} \mc R(\chi,\wtd I')A(\fk x,\wtd I)\Omega=\Psi_{i,k} \mc R(\chi,\wtd I')\xi.\label{eq31}
\end{align}
Since $\xi\in\mc H_i^\infty$, by theorem \ref{lb34} (especially equation \eqref{eq21}), $\xi$ is in the domain of $\scr R(\chi,\wtd I')|_{\mc H_i}$, and \eqref{eq31} equals $\scr R(\chi,\wtd I')\xi$. Since $\xi=\scr L(\xi,\wtd I)\Omega$, we know that $\Omega$ is in the domain of $\scr R(\chi,\wtd I')\scr L(\xi,\wtd I)|_{\mc H_0}$. Clearly $\Omega$ is in the domain of $\scr R(\chi,\wtd I')|_{\mc H_0}$. Therefore, by proposition \ref{lb3}-(a) and the locality in theorem \ref{lb13}, $\Omega$ is in the domain of $\scr L(\xi,\wtd I)\scr R(\chi,\wtd I')|_{\mc H_0}$ (equivalently, $\chi$ is in the domain of $\scr L(\xi,\wtd I)|_{\mc H_k}$), and \eqref{eq31} (which equals $\scr R(\chi,\wtd I')\xi$) equals $\scr L(\xi,\wtd I)\scr R(\chi,\wtd I')\Omega=\scr L(\xi,\wtd I)\chi$. We conclude that $\fk H_k(\wtd I')$ is in the domain of $\scr L(\xi,\wtd I)|_{\mc H_k}$, and \eqref{eq29} holds when acting on the core $\fk H_k(\wtd I')$ for $\ovl{A(\fk x,\wtd I)}|_{\mc H_k}$. In particular, we have $\Psi_{i,k}\ovl{A(\fk x,\wtd I)}|_{\mc H_k}\subset\scr L(\xi,\wtd I)|_{\mc H_k}$. On the other hand, by proposition \ref{lb21}, $\mc H_k^\infty(I')$ is a core for $\scr L(\xi,\wtd I)|_{\mc H_k}$. Therefore \eqref{eq29} is proved.

That $\xi\in\mc H_i^\pr(I)$ is proved in lemma \ref{lb37}.
\end{proof}

\begin{co}\label{lb41}
Assume that $\scr C$ is rigid and $\scr E^w$ is M\"obius covariant. Let $(A,\fk x,\wtd I,\mc H_i)$ and $(B,\fk y,\wtd J,\mc H_j)$ be respectively weak left and right operators of $\scr E^w$, where $\wtd I$ is anticlockwise to $\wtd J$, and $\mc H_i,\mc H_j\in\Obj(\scr C)$. Then these two operators commute strongly, in the sense that for any $\mc H_k\in\Obj(\scr C)$, the following diagram of closable operators commutes strongly.
\begin{align}
\begin{CD}
\mc H_k^\infty @> \quad B(\fk y,\wtd J)\quad   >> (\mc H_k\boxdot\mc H_j)^\infty\\
@V A(\fk x,\wtd I)   V  V @V A(\fk x,\wtd I) VV\\
(\mc H_i\boxdot\mc H_k)^\infty @> \quad B(\fk y,\wtd J) \quad  >> (\mc H_i\boxdot\mc H_k\boxdot\mc H_j)^\infty
\end{CD}
\end{align}
\end{co}

\begin{proof}
As in the proof of corollary \ref{lb35}, the strong commutativity of $A(\fk x,\wtd I)$ and $B(\fk y,\wtd J)$ follows from the naturality of these two operators, theorem \ref{lb36}, and the strong commutativity of $\scr L(\xi,\wtd I)$ and $\scr R(\eta,\wtd J)$ (where $\xi=A(\fk x,\wtd I)\Omega$ and $\eta=B(\fk y,\wtd J)\Omega$).
\end{proof}

To prepare for the proof of $\xi\in\mc H_i^\pr(I)$, we recall the geometric modular theory of rigid categorical extensions proved in \cite{Gui21b}. Consider the one parameter rotation subgroup  $\rho(t)=\left( \begin{array}{cc}
e^{\frac{\im t}2} & 0 \\
0 &e^{\frac{-\im t}2}
\end{array} \right)$ and dilation subgroup $\delta(t)=\left( \begin{array}{cc}
\cosh\frac t2 & -\sinh\frac t2 \\
-\sinh\frac t2 &\cosh\frac t2
\end{array} \right)$ of $\PSU$. For each $I\in\mc J$, define $\delta_I(t)=g\delta(t)g^{-1}$ where $g\in\PSU$ and $g\mbb S^1_+=I$. Then $\delta_I$ is well defined, $\delta_I(t)(I)=I$, and $\delta(t)=\delta_{\mbb S^1_+}(t)$. We lift $\varrho$ and $\delta$ to one-parameter subgroups of $\UPSU$ and denote them by the same symbols $\varrho$ and $\delta_I$.

Let $\RepfA$ be the braided $C^*$-tensor subcategory of $\RepA$ of all dualizable $\mc A$-modules. Then for any $\mc H_i\in\Obj(\RepfA)$, one can find a dual object $\mc H_{\ovl i}\in\Obj(\RepfA)$ and $\ev_{i,\ovl i}\in\Hom_{\mc A}(\mc H_i\boxtimes\mc H_{\ovl i},\mc H_0),\ev_{\ovl i,i}\in\Hom_{\mc A}(\mc H_{\ovl i}\boxtimes\mc H_i,\mc H_0)$ satisfying the conjugate equations
\begin{gather*}
(\ev_{i,\ovl i}\boxtimes\id_i)(\id_i\boxtimes\coev_{\ovl i,i})=\id_i=(\id_i\boxtimes\ev_{\ovl i,i})(\coev_{i,\ovl i}\boxtimes\id_i),\\
(\ev_{\ovl i,i}\boxtimes\id_{\ovl i})(\id_{\ovl i}\boxtimes\coev_{i,\ovl i})=\id_{\ovl i}=(\id_{\ovl i}\boxtimes\ev_{i,\ovl i})(\coev_{\ovl i,i}\boxtimes\id_{\ovl i}),
\end{gather*}
where we set $\coev_{i,\ovl i}=\ev_{i,\ovl i}^*,\coev_{\ovl i,i}=\ev_{\ovl i,i}^*$. Moreover, by \cite{LR97} (see also \cite{Yam04,BDH14}),  we may assume that the $\ev$ and $\coev$ to be standard, which means that the following is also satisfied for any $F\in\End_{\mc A}(\mc H_i)$:
\begin{align}
\ev_{i,\ovl i}(F\boxtimes\id_{\ovl i})\coev_{i,\ovl i}=\ev_{\ovl i,i}(\id_{\ovl i}\boxtimes F)\coev_{\ovl i,i}.
\end{align}

Now, for each $\wtd I\in\Jtd$ and $\mc H_i\in\Obj(\RepfA)$, we define unbounded operators $S_{\wtd I},F_{\wtd I}:\mc H_i\rightarrow\mc H_{\ovl i}$ with common domain $\mc H_i(I)$  by setting
\begin{equation*}
S_{\wtd I}\xi=L(\xi,\wtd I)^*\coev_{i,\ovl i}\Omega,\qquad F_{\wtd I}\xi=R(\xi,\wtd I)^*\coev_{\ovl i,i}\Omega
\end{equation*}
for any $\xi\in\mc H_i(I)$. By \cite{Gui21b} section 5, $S_{\wtd I}$ and $F_{\wtd I}$ are closable, whose closures are denoted by the same symbols in the following. Then we have
\begin{gather}
F_{\wtd I}=\varrho(2\pi)S_{\wtd I},\\
F_{\wtd I'}=S_{\wtd I}^*.
\end{gather}
Thus $\Delta_{\wtd I}:=S_{\wtd I}^*S_{\wtd I}$ equals $F_{\wtd I}^*F_{\wtd I}$. Moreover, by \cite{Gui21b} section 6, $\Delta_{\wtd I}$ is independent of $\arg_I$ and hence can be written as $\Delta_I$, and
\begin{align}
\Delta_I^{\im t}=\delta_I(-2\pi t).
\end{align}
By \cite[Thm. 7.8]{Gui21b}, if $\mc H_i\in\Obj(\RepfA)$ and $\xi\in\mc H_i$, then the following three conditions are equivalent.

(a) $\coev_{i,\ovl i}\Omega$ is in the domain of $(\scr L(\xi,\wtd I)|_{\mc H_{\ovl i}})^*$.

(a') $\coev_{\ovl i,i}\Omega$ is in the domain of $(\scr R(\xi,\wtd I)|_{\mc H_{\ovl i}})^*$.

(b) $\xi$ is in the domain of $\Delta_I^{\frac 12}$ (which are also the domains of $S_{\wtd I}$ and  $F_{\wtd I}$).\\
Moreover, if any of them are true, then   $\xi\in\mc H_i^\pr(I)$, and
\begin{align}
S_{\wtd I}\xi=\scr L(\xi,\wtd I)^*\coev_{i,\ovl i}\Omega,\qquad F_{\wtd I}\xi=\scr R(\xi,\wtd I)^*\coev_{\ovl i,i}\Omega.\label{eq32}
\end{align} 

\begin{lm}\label{lb37}
In theorem \ref{lb36}, we have $\xi\in\mc H_i^\pr(I)$ and $\eta\in\mc H_j^\pr(J)$.
\end{lm}

\begin{proof}
Recall $\xi=A(\fk x,\wtd I)\Omega$. We prove $\xi\in\mc H_i^\pr(I)$ by proving that $\xi$ is in the domain of $S_{\wtd I}|_{\mc H_i}$. Assume without loss of generality that $\scr E^w$ is vector labeled. We also assume that $\mc H_{\ovl i}\in\Obj(\scr C)$. Choose any $\nu\in\fk H_{\ovl i}(\wtd I')$. Then, by proposition \ref{lb27}, $\nu\in\mc H_{\ovl i}^w(I')\subset\mc H_{\ovl i}^\pr(I')$. Moreover, since $\mc R(\nu,\wtd I')|_{\mc H_i^\infty}$ is smooth, by theorem \ref{lb34}, $\scr R(\nu,\wtd I')|_{\mc H_i}$ is smooth. So $\coev_{i,\ovl i}\Omega$ (which is a smooth vector due to Rem. \ref{lb102}) is in the domain of $(\scr R(\nu,\wtd I')|_{\mc H_i})^*$. Therefore $\nu$ is in the domain of $F_{\wtd I'}=S_{\wtd I}^*$.  We compute
\begin{align}
&\bk{S_{\wtd I}^*\nu|\xi}=\bk{F_{\wtd I'}\nu|\xi}\xlongequal{\eqref{eq32}}  \bk{\scr R(\nu,\wtd I')^*\coev_{i,\ovl i}\Omega|\xi}=\bk{\mc R(\nu,\wtd I')^*\Psi^*\coev_{i,\ovl i}\Omega|\xi}\nonumber\\
=&\bk{\Psi^*\coev_{i,\ovl i}\Omega|\mc R(\nu,\wtd I')\xi}=\bk{\Psi^*\coev_{i,\ovl i}\Omega|\mc R(\nu,\wtd I')A(\fk x,\wtd I)\Omega}\nonumber\\
=&\bk{\Psi^*\coev_{i,\ovl i}\Omega|A(\fk x,\wtd I)\mc R(\nu,\wtd I')\Omega}=\bk{A(\fk x,\wtd I)^*\Psi^*\coev_{i,\ovl i}\Omega|\nu},\label{eq33}
\end{align}
where we notice that $\Psi^*\coev_{i,\ovl i}\Omega$ is in the domain of $A(\fk x,\wtd I)^*$ because $A(\fk x,\wtd I)^*$ is smooth and (by Rem. \ref{lb102}) $\Psi^*\coev_{i,\ovl i}\Omega\in(\mc H_i\boxdot\mc H_{\ovl i})^\infty$. 

By the Reeh-Schlieder property of $\scr E^w$, $\fk H_{\ovl i}(\wtd I')$ is a dense subspace of $\mc H_{\ovl i}$. By the M\"obius covariance of $\scr E^w$, we have $\delta_{I'}(-2\pi t)\fk H_{\ovl i}(\wtd I')=\fk H_{\ovl i}(\wtd I')$ and hence $\Delta_{I'}^{\im t}\fk H_{\ovl i}(\wtd I')=\fk H_{\ovl i}(\wtd I')$ for any $t\in\mbb R$. Therefore, by proposition \ref{lb38}, $\fk H_{\ovl i}(\wtd I')$ is a core for $\Delta_{I'}^{\frac 12}|_{\mc H_{\ovl i}}$ and hence for $S_{\wtd I}^*|_{\mc H_{\ovl i}}=F_{\wtd I'}|_{\mc H_{\ovl i}}$. Thus, we conclude from \eqref{eq33} that $\xi\in\Dom(S_{\wtd I}|_{\mc H_i})$ and $S_{\wtd I}\xi=A(\fk x,\wtd I)^*\Psi^*\coev_{i,\ovl i}\Omega$. This proves $\xi\in\mc H_i^\pr(I)$.
\end{proof}

\subsection{Weak categorical partial extensions}\label{lb49}

In this section, we assume that $\scr C$ is equipped with a braid operator $\ss$ so that $(\scr C,\boxdot,\ss)$ is a braided $C^*$-tensor category. In other words, we have a natural $\ss$, where for each $\mc H_i,\mc H_j\in\Obj(\scr R)$, $\ss_{i,j}:\mc H_i\boxdot\mc H_j\rightarrow\mc H_j\boxdot\mc H_i$ is a unitary isomorphism, and $\ss$ satisfies the Hexagon axioms.  We also assume that $\scr C$  is semisimple, i.e., any module $\mc H_i\in\Obj(\scr C)$ is unitarily equivalent to a finite direct sum of irreducible $\mc A$-modules in $\scr C$. If $\mc F$ is a set of  $\mc A$-modules in $\scr C$, we say that $\mc F$ \textbf{generates} $\scr C$ if for any irreducible $\mc H_i\in\Obj(\scr C)$, there exist $\mc H_{i_1},\dots,\mc H_{i_n}\in\mc F$ such that $\mc H_i$ is equivalent to an (irreducible) $\mc A$-submodule of $\mc H_{i_1}\boxdot\mc H_{i_2}\boxdot\cdots\boxdot\mc H_{i_n}$.

\begin{df}\label{lb31}
	Assume that the braided $C^*$-tensor category $(\scr C,\boxdot,\ss)$ is semisimple and $\mc F$ is a generating set of  $\mc A$-modules in $\scr C$. Let $\fk H$ assign, to each $\wtd I\in\Jtd,\mc H_i\in\mc F$, a set $\fk H_i(\wtd I)$ such that $\fk H_i(\wtd I_1)\subset\fk H_i(\wtd I_2)$ whenever $\wtd I_1\subset\wtd I_2$. A \textbf{weak categorical partial extension} $\scr E^w_\loc=(\mc A,\mc F,\boxdot,\fk H)$ of $\mc A$ associates, to any $\mc H_i\in\mc F,\mc H_k\in\Obj(\scr C),\wtd I\in\Jtd,\fk a\in\fk H_i(\wtd I)$, smooth and localizable operators
	\begin{gather*}	
	\mc L(\fk a,\wtd I):\mc H_k\rightarrow\mc H_i\boxdot\mc H_k,\\
	\mc R(\fk a,\wtd I):\mc H_k\rightarrow\mc H_k\boxdot\mc H_i,
	\end{gather*}
	such for any $\mc H_i,\mc H_j\in\mc F$, $\mc H_k,\mc H_{k'}\in\Obj(\scr C)$, and  $\wtd I,\wtd J,\wtd I_1,\wtd I_2\in\Jtd$, the conditions in definition \ref{lb30} are satisfied. (The braid operator $\ss$ in the braiding axiom of definition \ref{lb30} is assumed to be the same as that of $\scr C$.) Moreover, we assume that for any $\mc H_{i_1},\dots,\mc H_{i_n}\in\mc F$, $\mc H_k\in\Obj(\scr C)$, $\fk a_1\in\mc H_{i_1},\dots,\fk a_n\in\mc H_{i_n}$, and $\wtd I\in\Jtd$, the smooth operator $\mc L(\fk a_n,\wtd I)\cdots\mc L(\fk a_1,\wtd I)|_{\mc H_k^\infty}$ is localizable. We say that $\scr E^w_{\loc}$ is \textbf{M\"obius covariant} if \eqref{eq27} (or equivalently, \eqref{eq28}) holds for any $\mc H_i\in\mc F,\mc H_l\in\Obj(\scr C),\wtd I\in\Jtd,\fk a\in\fk H_i(\wtd I),\eta\in\mc H_l^\infty,g\in\UPSU$.
\end{df}

\begin{df}
Let $\scr E^w_\loc$ be a categorical partial extension. If $\mc H_i,\mc H_j\in\mc F$, $\wtd I\in\Jtd$, and $\fk a_i\in\fk H_i(\wtd I),\fk a_j\in\fk H_j(\wtd I)$, we say that $\fk a_i$ is \textbf{anticlockwise} to $\fk a_j$ if there exists $\wtd I_1,\wtd I_2\subset\wtd I$ such that $\wtd I_2$ is anticlockwise to $\wtd I_1$, $\fk a_i\in\fk H_i(\wtd I_1)$, and $\fk a_j\in\fk H_j(\wtd I_2)$.

Given $\mc H_{i_1},\dots,\mc H_{i_n}\in\mc F$, $\wtd I\in\Jtd$, and $\fk a_1\in\fk H_{i_1},\dots \fk a_n\in\fk H_{i_n}$, we define $\fk H_{i_n,\dots,i_1}(\wtd I)$ to be the set of all $(\fk a_n,\dots,\fk a_i)\in\fk H_{i_n}(\wtd I)\times\cdots\times\fk H_{i_1}(\wtd I)$ such that $\fk a_m$ is anticlockwise to $\fk a_{m-1}$ for any $m=2,3,\dots,n$.
\end{df}

\begin{pp}\label{lb32}
For any $(\fk a_n,\dots,\fk a_i)\in \fk H_{i_n,\dots,i_1}(\wtd I)$,
\begin{align}
\mc R(\fk a_1,\wtd I)\cdots\mc R(\fk a_n,\wtd I)|_{\mc H_k^\infty}=\ss_{i_n\boxdot\cdots\boxdot i_1,k}\mc L(\fk a_n,\wtd I)\cdots\mc L(\fk a_1,\wtd I)|_{\mc H_k^\infty}.
\end{align}
\end{pp}

\begin{proof}
We prove this by induction on $n$. The case $n=1$ is clear from the braiding axiom. Suppose that the above equation is true for $n-1$. Let $\mc H_l=\mc H_{i_{n-1}}\boxdot\cdots\boxdot\mc H_{i_1}$. Then, from the Hexagon axiom we have
\begin{align*}
\ss_{i_n\boxdot\cdots\boxdot i_1,k}=\ss_{i_n\boxdot l,k}=(\ss_{i_n,k}\boxdot\id_l)(\id_{i_n}\boxdot\ss_{l,k}).
\end{align*}
Then, for any $\chi\in\mc H_k^\infty$,
\begin{align*}
&\ss_{i_n\boxdot l,k}\mc L(\fk a_n,\wtd I)\cdots\mc L(\fk a_1,\wtd I)\chi\\
=&(\ss_{i_n,k}\boxdot\id_l)(\id_{i_n}\boxdot\ss_{l,k})\mc L(\fk a_n,\wtd I)\mc L(\fk a_{n-1},\wtd I)\cdots\mc L(\fk a_1,\wtd I)\chi\\
=& (\ss_{i_n,k}\boxdot\id_l)\mc L(\fk a_n,\wtd I)\ss_{l,k}\mc L(\fk a_{n-1},\wtd I)\cdots\mc L(\fk a_1,\wtd I)\chi\\
=&(\ss_{i_n,k}\boxdot\id_l)\mc L(\fk a_n,\wtd I)\mc R(\fk a_1,\wtd I)\cdots\mc R(\fk a_{n-1},\wtd I)\chi\\
=&(\ss_{i_n,k}\boxdot\id_l)\mc R(\fk a_1,\wtd I)\cdots\mc R(\fk a_{n-1},\wtd I)\mc L(\fk a_n,\wtd I)\chi\\
=&\mc R(\fk a_1,\wtd I)\cdots\mc R(\fk a_{n-1},\wtd I)\ss_{i_n,k}\mc L(\fk a_n,\wtd I)\chi\\
=&\mc R(\fk a_1,\wtd I)\cdots\mc R(\fk a_{n-1},\wtd I)\mc R(\fk a_n,\wtd I)\chi,
\end{align*}
where we have used the naturality of $\scr E^w_\loc$ in the third and sixth lines, induction in the fourth line,  weak locality in the fifth line, and the braiding axiom in the last line.
\end{proof}

\begin{thm}\label{lb40}
Any weak categorical partial extension $\scr E^w_\loc=(\mc A,\mc F,\boxdot,\fk H)$ can be extended to a weak categorical extension $\scr E^w=(\mc A,\scr C,\boxdot,\fk K)$. Moreover, $\scr E^w$ is M\"obius covariant if $\scr E^w_\loc$ is so.
\end{thm}

\begin{proof}
Our construction of $\scr E^w$ is similar to that in \cite{Gui21a} theorem 3.15. For each $\mc H_i\in\Obj(\scr C)$ and $\wtd I\in\Jtd$, we define
\begin{gather}
\fk K_i(\wtd I)_{n}=\coprod_{\mc H_{i_1},\dots,\mc H_{i_n}\in\mc F}\Hom_{\mc A}(\mc H_{i_n}\boxdot\cdots\boxdot\mc H_{i_1},\mc H_i)\times\fk H_{i_n,\dots,i_1}(\wtd I)
\end{gather}
for each $n=1,2,\dots$, and define
\begin{gather}
\fk K_i(\wtd I)=\coprod_{n\in\mbb Z_+}\fk K_i(\wtd I)_n.
\end{gather}
Then we have obvious inclusion $\fk K_i(\wtd I)\subset\fk K_i(\wtd J)$ when $\wtd I\subset\wtd J$. A general element  $\fk k\in\fk K_i(\wtd I)$ is of the form
\begin{align*}
\fk k=(G,\fk a_n,\dots,\fk a_1),
\end{align*}
where $G\in\Hom_{\mc A}(\mc H_{i_n}\boxdot\cdots\boxdot\mc H_{i_1},\mc H_i)$ and $(\fk a_n,\dots,\fk a_1)\in\fk H_{i_n,\dots,i_1}(\wtd I)$. We define for each $\mc H_k\in\Obj(\scr C)$  smooth operators $\mc L(\fk k,\wtd I):\mc H_k^\infty\rightarrow(\mc H_i\boxdot\mc H_k)^\infty$ and $\mc R(\fk k,\wtd I):\mc H_k^\infty\rightarrow(\mc H_k\boxdot\mc H_i)^\infty$ by setting for each $\chi\in\mc H_k^\infty$ that
\begin{gather}
\mc L(\fk k,\wtd I)\chi=(G\boxdot\id_k)\mc L(\fk a_n,\wtd I)\cdots\mc L(\fk a_1,\wtd I)\chi,\\
\mc R(\fk k,\wtd I)\chi=(\id_k\boxdot G)\mc R(\fk a_1,\wtd I)\cdots\mc R(\fk a_n,\wtd I)\chi.
\end{gather}
By the naturality of braiding, one has $\ss_{i,k}(G\boxdot\id_k)=(\id_k\boxdot G)\ss_{i_n\boxdot\cdots\boxdot i_1,k}$. This equation, together with proposition \ref{lb32}, shows that $\mc R(\fk k,\wtd I)\chi=\ss_{i,k}\mc L(\fk k,\wtd I)\chi$. 

It is now a routine check that $\mc L$ and $\mc R$ satisfy all the axioms of a weak categorical extension, and that $\scr E^w$ is M\"obius covariant if $\scr E^w_\loc$ is so. Note that the Neutrality follows from $\ss_{i,0}=\id_i$. To prove the Reeh-Schlieder property and the density of fusion products, first prove them when $\mc H_i$ is irreducible and let $G$ be a coisometry.  To check the weak locality, we choose $\wtd J$ clockwise to $\wtd I$, and $\fk k=(F,\fk a_m,\dots,\fk a_1),\fk l=(G,\fk b_n,\dots,\fk b_1)$ where $(\fk a_m,\dots,\fk a_1)\in \fk H_{i_m,\dots,i_1}(\wtd I),(\fk b_n,\dots,\fk b_1)\in \fk H_{j_n,\dots,j_1}(\wtd J)$, and $F\in\Hom_{\mc A}(\mc H_{\wht i},\mc H_i),G\in\Hom_{\mc A}(\mc H_{\wht j},\mc H_j)$ where $\mc H_{\wht i}=\mc H_{i_m}\boxdot\cdots\boxdot\mc H_{i_1}$ and $\mc H_{\wht j}=\mc H_{j_n}\boxdot\cdots\boxdot\mc H_{j_1}$. Let $A=\mc L(\fk a_m,\wtd I)\cdots,\mc L(\fk a_1,\wtd I)$ and $B=\mc R(\fk b_1,\wtd J)\cdots\mc R(\fk b_n,\wtd J)$. Then $A$ and $B$ clearly commute adjointly. So the adjoint commutativity of $\mc L(\fk k,\wtd I)=(F\boxdot\id)A$ and $\mc R(\fk l,\wtd J)=(\id\boxdot G)B$ follows from that of the following matrix of diagrams.
\begin{align*}
\begin{CD}
\mc H_k^\infty @>\quad B \quad >> (\mc H_k\boxdot\mc H_{\wht j})^\infty @> ~~~~~~  \id_k\boxdot G ~~~~~~  >> (\mc H_k\boxdot\mc H_j)^\infty\\
@V A  VV @V A VV @V A VV\\
(\mc H_{\wht i}\boxdot\mc H_k)^\infty @> \quad  B \quad  >> (\mc H_{\wht i}\boxdot\mc H_k\boxdot\mc H_{\wht j})^\infty   @> \quad \id_{\wht i}\boxdot\id_k\boxdot G\quad  >>(\mc H_{\wht i}\boxdot\mc H_k\boxdot\mc H_j)^\infty\\
@V F\boxdot\id_k  VV @V  F\boxdot\id_k\boxdot\id_{\wht j} VV @V  F\boxdot\id_k\boxdot\id_j  VV\\
(\mc H_i\boxdot\mc H_k)^\infty @> \quad B \quad  >>(\mc H_i\boxdot\mc H_k\boxdot\mc H_{\wht j})^\infty  @>~~~~  \id_i\boxdot\id_k\boxdot G~~~~  >>(\mc H_i\boxdot\mc H_k\boxdot\mc H_j)^\infty
\end{CD}
\end{align*} 
\end{proof}

\begin{df}
A quadruple $(A,\fk x,\wtd I,\mc H_i)$ (resp. $(B,\fk y,\wtd J,\mc H_j)$) is called a \textbf{weak left (resp. right) operator} of $\scr E^w_\loc=(\mc A,\mc F,\boxtimes,\fk H)$ if all the conditions in definition \ref{lb39} are satisfied, except that we assume $\mc H_l\in\mc F$.\footnote{Note that we do not assume $\mc H_i$ or $\mc H_j$ to be in $\mc F$.}
\end{df}

From the construction of $\scr E^w$ from $\scr E^w_\loc$, it is clear that we have:

\begin{thm}
Let $\scr E^w$ be the weak categorical extension defined in the proof of theorem \ref{lb40}. Then any weak left (resp. right) operator of $\scr E^w_{\loc}$ is also a weak left (resp. right) operator of $\scr E^w$.
\end{thm}

\begin{co}\label{lb62}
Corollary \ref{lb41} holds verbatim with $\scr E^w$ replaced by $\scr E^w_\loc$.
\end{co}

\section{Application to vertex operator algebras}

We assume that the readers are familiar with the basics of unitary vertex operator algebras (VOAs) and their representations (modules) \cite{CKLW18,DL14}, and the intertwining operators of VOAs \cite{FHL93}. We follow basically the notations in \cite{Gui19a} chapter 1, in which all the relevant definitions and properties are presented.

Let $V$ be a unitary simple (equivalently, CFT-type \cite{CKLW18}) VOA $V$. Then the pre-Hilbert space $V$ (with inner product $\bk{\cdot|\cdot}$\footnote{All inner products are assumed to be non-degenerate.}) has orthogonal grading $V=\bigoplus_{n\in\mbb N}V(n)$ where $V(0)$ is spanned by the vacuum vector $\Omega$, and $\dim V(n)<+\infty$. For any $v\in V$, we write the vertex operator as $Y(v,z)=\sum_{n\in\mbb Z}Y(v)_nz^{-n-1}$ where $z\in\mbb C^\times=\mbb C-\{0\}$, and $Y(v)_n$ is a linear map on $V$, called the $n$-th mode of $Y(v,z)$. Let $\nu\in V(2)$ be the conformal vector. Then $\{L_n=Y(\nu)_{n+1}:n\in\mbb Z\}$ form a Virasoro algebra with central charge $c\geq0$. We have $L_0|_{V(n)}=n\id_{V(n)}$.

A $V$-module $W_i$ has grading $W_i=\bigoplus_{n\in\mbb R}W_i(n)$ where each $W_i(n)$ is finite-dimensional, and $\dim W_i(n)=0$ when $n$ is small enough. For each $v\in V$ we also have a vertex operator $Y_i(v,z)=\sum_{n\in\mbb Z}Y_i(v)_nz^{-n-1}$ where each $Y_i(v)_n\in\End(W_i)$. $\{L_n=Y_i(\nu)_{n+1}:n\in\mbb Z\}$ form a representation of the Virasoro algebra on $W_i$ with the same central charge $c$. We have $L_0|_{W_i(n)}=n\id_{W_i(n)}$ for each $n\in\mbb R$. A vector $w^{(i)}\in W_i$ is called \textbf{homogeneous} if $w^{(i)}\in W_i(s)$ for some $s\in\mbb R$. In this case we say that $s$ is the \textbf{conformal weight} of $w^{(i)}$. If moreover $L_1w^{(i)}=0$, then $w^{(i)}$ is called \textbf{quasi-primary}; if $L_nw^{(i)}=0$ for all $n>0$, then $w^{(i)}$ is called \textbf{primary}. When $W_i$ is a unitary $V$-module, the grading $W_i=\bigoplus_{n\in\mbb R}W_i(n)$ is orthogonal, $\dim W_i(n)=0$ when $n<0$, and $L_n^\dagger=L_{-n}$.\footnote{In this chapter, if $A$ is a linear   map from a pre-Hilbert space $M$ to another one $N$, then $A^\dagger$ (if exists) is the unique linear  map from $N$ to M satisfying $\bk {A\mu|\nu}=\bk{\mu|A^\dagger\nu}$  for any $\mu\in M,\nu\in N$.} $V$ itself is a $V$-module, called the vacuum $V$-module, and is sometimes denoted by $W_0$.

We assume that any irreducible $V$-module $W_i$ is unitarizable, i.e., the vector space $W_i$ admits a (unique up to scalar multiple) inner product under which $W_i$ becomes a unitary $V$-module. Such $V$ is called \textbf{strongly unitary}.

For any $V$-modules $W_i,W_j,W_k$, we let $\mc V{k\choose i~j}$ be the vector space of type ${W_k\choose W_iW_j}={k\choose i~j}$ intertwining operators of $V$.  Then for any $\mc Y\in\mc V{k\choose i~j}$  and any $w^{(i)}\in W_i$, we have $\mc Y(w^{(i)},z)=\sum_{s\in\mbb R}\mc Y(w^{(i)})_sz^{-s-1}$ where each $\mc Y(w^{(i)})_s$ is a linear map from $W_j$ to $W_k$. We say that $W_i,W_j,W_k$ are respectively the \textbf{charge space}, the \textbf{source space}, and the \textbf{target space} of $\mc Y$. We say that $\mc Y$ is a \textbf{unitary} (resp. \textbf{irreducible}) \textbf{intertwining operator} if all $W_i,W_j,W_k$ are unitary (resp. irreducible) $V$-modules. Note that $Y_i\in\mc V{i\choose 0~i}=\mc V{W_i\choose W_0W_i}$ and, in particular, $Y\in\mc V{0\choose 0~0}$.

\subsection{Polynomial energy bounds}\label{lb69}

Suppose that $\mc Y\in\mc V{k\choose i~j}$ is unitary. Let $w^{(i)}\in W_i$ be homogeneous. Given $a\geq 0$, we say that $\mc Y(w^{(i)},z)$\footnote{Here, $z$ is understood as a variable but not as a concrete complex number.} satisfies \textbf{$a$-th order energy bounds}, if there exist $M,b\geq0$, such that for any $s\in\mathbb R,w^{(j)}\in W_j$, 
\begin{align}
\lVert \mc Y(w^{(i)})_sw^{(j)} \lVert\leq M(1+|s|)^b\lVert (1+L_0)^aw^{(j)}\lVert.
\end{align}
$1$-st order energy bounds are called \textbf{linear energy bounds}. We say that $\mc Y(w^{(i)},z)$ is \textbf{energy-bounded} if satisfies $a$-th order energy bounds for some $a$, that $\mc Y$ is energy-bounded if $\mc Y(w^{(i)},z)$ is energy-bounded for any homogeneous $w^{(i)}\in W_i$, that a unitary $V$-module $W_i$ is energy-bounded if $Y_i$ is an energy-bounded intertwining operator, that $V$ is energy-bounded if the vacuum module $W_0$ is energy-bounded. We say that $V$ is \textbf{strongly energy-bounded} if (the vertex operator of) any irreducible unitary $V$-module is energy-bounded. We say that $V$ is \textbf{completely energy-bounded} if any irreducible unitary intertwining operator of $V$ is energy-bounded.

The following is \cite{Gui19a} corollary 3.7-(b).

\begin{pp}\label{lb53}
Let $\mc Y\in\mc V{k\choose i~j}$ be unitary. Assume that $W_i$ is irreducible, and $W_j,W_k$ are energy-bounded. If $\mc Y(w^{(i)},z)$ is energy-bounded for some non-zero homogeneous $w^{(i)}\in W_i$, then $\mc Y$ is an energy-bounded intertwining operator.
\end{pp}

We shall define smeared intertwining operators and recall some basic properties. See \cite{Gui19a} chapter 3 and \cite{Gui21a} section 4.4 for details. For each unitary $V$-module $W_i$ we let $\mc H_i$ be the Hilbert space completion of the pre-Hilbert space $W_i$. We set 
\begin{align*}
\mc H_i^\infty=\bigcap_{n\in\mbb Z_+}\Dom(\ovl{L_0}^n).
\end{align*}
This notation will coincide with the previous one defined for conformal net modules after $\mc H_i$ is equipped with a natural structure of conformal net module. For any $\wtd I=(I,\arg_I)\in\Jtd$ and $f\in C^\infty_c(I)$ , we call $\wtd f=(f,\arg_I)$ a (smooth) \textbf{arg-valued function} on $S^1$ with support inside $\wtd I$, and let $C^\infty_c(\wtd I)$ be the set of all such $\wtd f$.  If $\wtd I\subset\wtd J\in\Jtd$, then $C^\infty_c(\wtd I)$ is naturally a subspace of $C^\infty_c(\wtd J)$ by identifying each $(f,\arg I)\in C^\infty_c(\wtd I)$ with $(f,\arg J)$. Now, assume that $\mc Y(w^{(i)},z)$ is energy-bounded. $\wtd I=(I,\arg_I)\in\Jtd$, and $\wtd f=(f,\arg_I)\in C^\infty_c(\wtd I)$ , we  define the \textbf{smeared intertwining operator} $\mc Y(w^{(i)},\wtd f)$ 
\begin{align}
\mc Y(w^{(i)},\wtd f)=\int_{\arg_I(I)}\mc Y(w^{(i)},e^{i\theta})f(e^{i\theta})\cdot\frac{e^{i\theta}}{2\pi}d\theta
\end{align}
which indeed maps $W_j$ into $\mc H_k^\infty$.
Regarding $\mc Y(w^{(i)},\wtd f)$ as an unbounded operator from $\mc H_j$ to $\mc H_k$ with domain $W_j$, then $\mc Y(w^{(i)},\wtd f)$ is closable, and its closure  contains $\mc H_j^\infty$. Moreover, we have
\begin{gather*}
\ovl{\mc Y(w^{(i)},\wtd f)}\mc H_j^\infty\subset\mc H_k^\infty,\qquad \ovl{\mc Y(w^{(i)},\wtd f)}^*\mc H_k^\infty\subset\mc H_j^\infty,
\end{gather*}
i.e., $\ovl{\mc Y(w^{(i)},\wtd f)}$ is smooth. In the following, \emph{we will always denote by $\mc Y(w^{(i)},\wtd f)$ the restriction of the closed operator $\ovl{\mc Y(w^{(i)},\wtd f)}$ to the core $\mc H_j^\infty$.} Then the formal adjoint $\mc Y(w^{(i)},\wtd f)^\dagger$ exists, which is the restriction of $\mc Y(w^{(i)},\wtd f)^*$ to $\mc H_k^\infty$.

It was proved in \cite{Gui19a} lemma 3.8 that if $\mc Y(w^{(i)},z)$ satisfies $a$-th order energy bounds, then for any $p\geq0$, there exists $M\geq0$ such that for any $\wtd f$ and $\eta\in\mc H_j^\infty$,
\begin{align*}
\big\lVert (1+\ovl{L_0}^p) \mc Y(w^{(i)},\wtd f)\eta\big\lVert \leq M \big\lVert  (1+\ovl{L_0}^{(a+p)}) \eta\big\lVert.
\end{align*}
As a consequence, any product of (finitely many) energy-bounded smeared intertwining operators is bounded by a scalar multiple of $1+\ovl{L_0}^r$ for some $r\geq0$. By \cite{CKLW18} lemma 7.2, $\ovl{L_0}^r$ is localizable (if the action of rotation group is defined by $\varrho(t)=e^{\im t\ovl{L_0}}$). Thus we have:

\begin{pp}\label{lb45}
If $A$ is a product of smooth operators of the form $\mc Y(w^{(i)},\wtd f)$ where $\mc Y$ is a unitary intertwining operator and $\mc Y(w^{(i)},z)$ is energy-bounded, then $A$ is (smooth and) localizable.
\end{pp}

If $v\in V$ is homogeneous,  and $Y_i(v,z)$ is energy-bounded, then the smeared vertex operator $Y_i(v,\wtd f)$ is independent of the argument $\arg_I$. So we will write it as $Y_i(v,f)$ instead.

\subsection{M\"obius and conformal covariance}

We discuss the M\"obius and conformal covariance of smeared intertwining operators. First, by \cite{TL99}, for any unitary $V$-module $W_i$, the action of Virasoro algebra on $W_i$ can be integrated to a unitary representation of $\GAV$ on $\mc H_i$. More precisely, consider the Lie algebra $\Vect(\mbb S^1)$ of $\Diffp(\mbb S^1)$ and its universal cover $\scr G$, which is the Lie algebra of (global and smooth) vector fields on $\mbb S^1$. Let $\Vectc(\mbb S^1)$ be the complexification of $\Vect(\mbb S^1)$, which has a natural $*$-structure (involution) preserving the elements in $\im\Vect(\mbb S^1)$.  Let $e_n\in C^\infty(\mbb S^1)$ be  defined by $e_n(e^{\im\theta})=e^{\im n\theta}$ for any $n\in\mbb Z$. Define for each $n\in\mbb Z$,
\begin{align*}
l_n=-\im e_n\frac{\partial}{\partial\theta}\quad \in\Vectc(\mbb S^1).
\end{align*}
Then $l_n^*=l_{-n}$. If $f=\sum_{n\in\mbb Z}\wht f_ne_n\in C_c^\infty(\mbb S^1)$ (where each $\wht f_n\in\mbb C$), we define
\begin{align*}
T(f)=\sum_{n\in\mbb Z}\wht f_nl_{n-1}\quad\in\Vectc(\mbb S^1).
\end{align*}
Then $\im T(f)\in\Vect(\mbb S^1)$ if and only if $e_{-1}f$ is real. In that case, we can exponentiate $\im T(f)$ to an element $\exp(\im T(f))$ in $\scr G$. For instance, we have
\begin{gather*}
\varrho(t)=\exp(\im tl_0),\qquad \delta(t)=\exp(t\frac{l_1-l_{-1}}2).
\end{gather*}

By \cite{TL99} section 2, if $Y_i(v,z)$ satisfies linear energy bounds and $Y_i(f)^\dagger=Y_i(f)$, then $Y_i(v,f)$ is essentially self-adjoint (i.e. $\ovl{Y_i(f)}$ is self-adjoint). For the conformal vector $\nu$, the vertex operator $Y_i(\nu,z)$ satisfies linear energy bounds by \cite{BS90} section 2. Moreover, when $e_1f$ is real, $Y_i(\nu,f)^\dagger=Y_i(\nu,f)$. So $\ovl{Y_i(\nu,f)}$ is self-adjoint. By \cite{TL99} theorem 5.2.1, there exists a unique strongly continuous projective representation $U_i$ of $\scr G$ on $\mc H_i$ such that for any $f\in C^\infty(\mbb S^1)$ where $e_1f$ is real, $e^{\im\ovl{Y_i(\nu,f)}}$ is a representing element of $U_i(\exp(\im T(f)))$. We write $U_0$ as $U$.\footnote{After constructing a conformal net $\mc A_V$ from $V$, $U$ will be the projective representation making $\mc A_V$ conformal covariant.} Then, as in section \ref{lb29}, one can use  \eqref{eq35} to define  a central extension $\GAV$ of $\scr G$ associated $U:\scr G\curvearrowright\mc H_0$. The natural action of $\GAV$ on $\mc H_0$ is also denoted by $U$.  By (the proof of) \cite{Gui21a} proposition 4.9, we have:

\begin{pp}\label{lb46}
The continuous projective representation $U_i:\scr G\curvearrowright\mc H_i$ can be lifted uniquely to a  unitary representation  of $\GAV$ on $\mc H_i$, denoted also by $U_i$,  such that for any $g\in\GAV$, if there exist $\lambda\in\mbb C$ (where $|\lambda|=1$) and   $f\in C^\infty(\mbb S^1)$ (where $e_{-1}f$ is real) such that the image of $g$ in $\scr G$ is $\exp(\im T(f))$,  and that $U(g)=\lambda e^{\im \ovl{Y(\nu,f)}}$, then $U_i(g)=\lambda e^{\im\ovl{Y_i}(\nu,f)}$.
\end{pp}

Let $g\in\GAV$. If $\wtd I\in\Jtd$ and $\wtd f\in C_c^\infty(\wtd I)$, we define 
\begin{align}
\wtd f\circ g^{-1}=(f\circ g^{-1},\arg_{gI})\quad \in C_c^\infty(g\wtd I)
\end{align}
where $\arg_{gI}$ is defined by $g\wtd I=(gI,\arg_{gI})$.  Recall that $\GAV$ is acting on $\mbb S^1$. So $f\circ g^{-1}(z)=f(g^{-1}z)$ for any $z\in\mbb S^1$. Define $\partial_zg^{-1}$ to be a smooth function on $\mbb S^1$ by setting
\begin{align}
(\partial_zg^{-1})(e^{\im \theta})=\frac{\partial (g^{-1}e^{\im\theta})}{\im e^{\im\theta}\partial\theta}.
\end{align}
This function depends only on the image $\pi(g)$ of $g$ in $\Diffp(\mbb S^1)$. (In the case that $\pi(g^{-1})$ is analytic near $\mbb S^1$, $\partial_zg^{-1}$ is just the ordinary derivative of the analytic function $\pi(g^{-1})$.) We also define $\arg \big((\partial_zg^{-1})(e^{\im \theta})\big)$ as follows: Choose any map  $\gamma:[0,1]\rightarrow\GAV$ satisfying $\gamma(0)=1,\gamma(1)=g$ such that $\gamma$ descends to a (continuous) path in $\Diffp(\mbb S^1)$. Let  $\arg \big((\partial_z\gamma(0)^{-1})(e^{\im \theta})\big)$ be $0$ and let the argument of $(\partial_z\gamma(t)^{-1})(e^{\im \theta})$ change continuously as $t$ varies from $0$ to $1$. Then $\arg \big((\partial_zg^{-1})(e^{\im \theta})\big)$ is defined to be the argument of $(\partial_z\gamma(1)^{-1})(e^{\im \theta})$. Using this arg-value, one can define a smooth function $(\partial_zg^{-1})^\lambda$ for all $\lambda\in\mbb C$, where $(\partial_zg^{-1})^\lambda(e^{\im\theta})=\big((\partial_zg^{-1})(e^{\im \theta})\big)^\lambda$. We understand
\begin{align*}
(\partial_zg^{-1})^\lambda\cdot  (\wtd f\circ g^{-1})=((\partial_zg^{-1})^\lambda\cdot(f\circ g^{-1}),\arg_{gI})
\end{align*}
to be in $C_c^\infty(g\wtd I)$. 

The following can be proved in essentially the same way as \cite{CKLW18} proposition 6.4. Note that  by the fact that $\scr G$ is generated by elements of the form $g=\exp(\im T(f))$ (see for instance \cite{Gui21a} proposition 2.4), and that for any such $g$, $U_i(g)$ is smooth (cf. \cite{TL99} proposition 2.1), we conclude that 
\begin{align}\label{eq51}
\text{$U_i(g)$ is smooth for any $g\in\GAV$}.
\end{align}

\begin{pp}\label{lb44}
Let $\mc Y\in\mc V{k\choose i~j}$ be a unitary intertwining operator, let $w^{(i)}\in W_i$ be quasi-primary with (necessarily non-negative) conformal weight $d$, and assume that $\mc Y(w^{(i)},z)$ is energy-bounded.  Then, for any $g\in\UPSU$, the following holds when acting on $\mc H_j^\infty$.
\begin{align}
U_k(g)\mc Y(w^{(i)},\wtd f)U_j(g)^*=\mc Y(w^{(i)},(\partial_zg^{-1})^{(1-d)}\cdot (\wtd f\circ g^{-1})).\label{eq37}
\end{align}
If $w^{(i)}$ is primary, then \eqref{eq37} holds for any $g\in\GAV$. If $w^{(i)}$ is homogeneous but not assumed to be quasi-primary, then \eqref{eq37} holds for any $g$ inside the rotation subgroup. 
\end{pp}
In this article, the above proposition will be used only when $g$ is an element inside the rotation subgroup $\varrho$ or dilation subgroup $\delta$. (Note that these two subgroups generate $\UPSU$.) When $g=\varrho(t)$, this proposition is also proved in \cite{Gui19a} proposition 3.15. The calculations there can be easily adapted to prove the case $g=\delta(t)$.

\subsection{Strong locality and intertwining property}\label{lb67}

Assume that  $V$ is strongly energy-bounded. The following theorem follows from \cite{Gui19a} propositions 2.13 and 3.13. (See also the proof of \cite{Gui19a} proposition 3.16.)

\begin{thm}[Weak intertwining property]\label{lb63}
Let $\mc Y\in\mc V{k\choose i~j}$ be a unitary intertwining operator. Suppose that $w^{(i)}\in W_i$ is homogeneous and $\mc Y(w^{(i)},z)$ is energy-bounded. Then for any $\wtd I\in\Jtd$, $J\in\mc J$ disjoint from $\wtd I$, $\wtd f\in C_c^\infty(\wtd I)$, $g\in C_c^\infty(J)$, and homogeneous $v\in V$, the following diagram commutes adjointly.
\begin{align}\label{eq38}
\begin{CD}
\mc H_j^\infty @>~Y_j(v,g)~>> \mc H_j^\infty\\
@V \mc Y(w^{(i)},\wtd f)  VV @VV \mc Y(w^{(i)},\wtd f) V\\
\mc H_k^\infty @> Y_k(v,g)>> \mc H_k^\infty
\end{CD}
\end{align}	
\end{thm}

Let $E$ be a set of homogenous vectors in $V$. Given a homogeneous $w^{(i)}\in W_i$ such that $\mc Y(w^{(i)},z)$ is energy-bounded, we say that $\mc Y(w^{(i)},z)$ satisfies the \textbf{$E$-strong intertwining property}, if for any $\wtd I\in\Jtd$, $J\in\mc J$ disjoint from $\wtd I$, $f\in C_c^\infty(\wtd I)$, $g\in C_c^\infty(J)$, and  $v\in E$, diagram \ref{eq38} commutes strongly. It is called the  \textbf{strong intertwining property} if $E$ is the set of all homogeneous vectors of $V$. We say that $\mc Y$ satisfies the \textbf{strong intertwining property} if $\mc Y(w^{(i)},z)$ is so for any homogeneous $w^{(i)}\in W_i$. We say that $V$ is \textbf{$E$-strongly local} \cite{CKLW18} if  $Y(v,z)$ (the vertex operator for the vacuum module $V$) satisfies the strong intertwining property for any homogeneous $v\in E$; if moreover $E$ consists of all homogeneous vectors, we say that $V$ is \textbf{strongly local}.

We now review the construction of conformal nets from VOAs. We say that $E$ \textbf{generates} $V$ if the smallest subspace of $V$ containing $E$ and invariant under all $Y(v)_n$ (where $v\in E,n\in\mbb Z$) is $V$. We say that $E$ is \textbf{quasi-primary} if all the vectors in $E$ are quasi-primary. A generating and quasi-primary set of vectors always exists:

\begin{pp}[\cite{CKLW18} Prop.6.6]\label{lb56}
If $E$ is the set of all quasi-primary vectors in $V$, then $E$ generates $V$.
\end{pp}
\begin{proof}
Use the fact that $Y(L_{-1}v)_n=-nY(v)_{n-1}$ (by the translation property) and that any homogeneous $v\in V$ equals a linear combination of vectors of the form $L_{-1}^nu$ where $n\in\mbb N$ and  $u\in V$ is quasi-primary.
\end{proof}

To construct a conformal net from $V$, we   let the vacuum Hilbert space be $\mc H_0$ (the completion of $V=W_0$), and let $\Omega\in V$ be the vacuum vector. For each $I\in\mc J$, define $\mc A_E(I)$ to be the von Neumann algebra on $\mc H_0$ generated by all $\ovl{Y(v,f)}$ where $v\in E$ and $f\in C_c^\infty(I)$. Let $\mc A_E$ be the family $I\in\mc J\mapsto \mc A_E(I)$. The following two theorems are due to \cite{CKLW18}.

\begin{thm}
Suppose that $E$ is quasi-primary and generates $V$, and $V$ is $E$-strongly local. Then $\mc A_E$ is a conformal net, where the continuous projective representation of $\Diffp(\mbb S^1)$ is defined by $U$.\footnote{More precisely, it is the restriction of $U$ from $\scr G$ to $\Diffp(\mbb S^1)$.}
\end{thm}

\begin{thm}[\cite{CKLW18} Thm. 8.1]\label{lb43}
Suppose that $E$ is quasi-primary and generates $V$, and $V$ is $E$-strongly local. Then $V$ is strongly local, and $\mc A_E=\mc A_V$.
\end{thm}

\begin{proof}
We sketch a proof using the methods in section \ref{lb48}. Choose any $I\in\mc J$. Let $v\in V$ be homogeneous and $f\in C_c^\infty(I)$. We want to show that $A=\ovl{Y(v,f)}$ is affiliated with $\mc A_E(I)$. Set $\xi=A\Omega$, and define $A_\xi$ to be the unbounded operator on $\mc H_0$ with domain $\mc A_E(I')\Omega$ such that $A_\xi y\Omega=y\xi$ for any $y\in\mc A_E(I')$. 

We first show that $A_\xi$ is closable. The proof is similar to that of lemma \ref{lb37}. Let $S_I$ and $S_{I'}$ be the $S$  operators associated to $(\mc A_E(I),\Omega)$ and $(\mc A_E(I'),\Omega)$ respectively. Then, by Tomita-Takesaki theory, $S_{I'}=S_I^*$, and $A_\xi$ is closable if $\xi\in\Dom(S_I)$ (cf. \cite{Gui21b} proposition 7.1). Let $\mc Q$ be the algebra of operators with domain $\mc H_0^\infty$ generated by all $Y(u,g)$ and $Y(u,g)^\dagger$ where $u\in E$ and $g\in C_c^\infty(I')$. Then a standard argument using Schwarz reflection principle shows that $\mc Q\Omega$ is dense in $\mc H_0^\infty$. (See the first paragraph of the proof of \cite{CKLW18} theorem 8.1.) Moreover, since each $B\in\mc Q$ is smooth, $\Omega\in\Dom(B^*)$. Let $\eta=B\Omega$ and set $B_\eta:x\Omega\in\mc A_E(I)\Omega\mapsto x\eta$. Since  $\ovl B$ commutes strongly with all $x\in\mc A_E^\infty(I)$ and hence with all $x\in\mc A_E(I)$, we have $B_\eta\subset \ovl B$. (Indeed one has $B_\eta=\ovl B$; see the second paragraph or proposition \ref{lb27}.) So $B^*\subset B_\eta^*$, which implies $\Omega\in\Dom(B_\eta^*)$. By \cite{Gui21b} proposition 7.1, $\eta\in\Dom(S_{I'})=\Dom(S_I^*)$, and $S_I^*\eta=B_\eta^*\Omega$. Note that $\eta=B\Omega$ and $B^*\Omega=B_\eta^*\Omega$. Using the (weak) commutativity of $A$ and $B$ due to theorem \ref{lb63}, we compute
\begin{align*}
\bk{S_I^*B\Omega|\xi}=\bk{B^*\Omega|A\Omega}=\bk{\Omega|BA\Omega}=\bk{\Omega|AB\Omega}=\bk{A^*\Omega|B\Omega}.
\end{align*}
By the M\"obius covariance of smeared vertex operators, $\mc Q\Omega$ is invariant under the action of $\delta_{I'}$. Thus, by the Bisognano-Wichmann property, it is invariant under $\Delta_{I'}^{\im t}$ where $\Delta_{I'}=S_IS_I^*$. Thus $\mc Q\Omega$ is a core for $S_I^*$ by proposition \ref{lb38}. This shows that $\xi\in\Dom(S_I)$ and $S_I\xi=A^*\Omega$. So $A_\xi$ is closable.

We let $A_\xi$ be the closure of the original operator, and prove $A_\xi=A$ using the method in theorem \ref{lb36}. Note that each $B\in\mc Q$ has domain $\mc H_0^\infty$ and is a product of smeared vertex operators commuting strongly with elements in $\mc A_E(I)$. Therefore, for each $x\in\mc A_E^\infty(I)$, we have $xB\subset Bx,x^*B\subset Bx^*$. Thus $x$ commutes strongly with $\ovl B$. By corollary \ref{lb24}, $\mc A_E^\infty(I)$ is $*$-strongly dense in $\mc A_E(I)$. This shows that elements in $\mc A_E(I)$ commute strongly with $\ovl B$. Therefore,  $A_\xi$ commutes strongly with $\ovl B$. Since $\Omega$ is in the domains of $BA_\xi$ and $B$, by proposition \ref{lb3}-(a), $\Omega\in\Dom(A_\xi B)$ and $A_\xi B\Omega=BA_\xi\Omega=B\xi$. Since $AB\Omega=BA\Omega=B\xi$, we have $A_\xi B\Omega=AB\Omega$ for any $B\in\mc Q$. Therefore,  $A_\xi$ equals $A$ when acting on $\mc Q\Omega$. By rotation covariance, $\mc Q\Omega$ is QRI, and hence is a core for $A$. So $A\subset A_\xi$. Since $\mc A_E^\infty(I')\Omega$ is a core for $A_\xi$ (by the fact that $\mc A_E^\infty(I')$ is strongly dense in $\mc A_E(I')$), and since $\mc A_E^\infty(I')\Omega\subset\mc H_0^\infty$, we must have $A=A_\xi$. Thus $A$ is affiliated with $\mc A_E(I)$.
\end{proof}

We now discuss the relation between unitary $V$-modules and $\mc A_V$-modules. To begin with, suppose that $\mc M$ is a von Neumann algebra on a Hilbert space $\mc K$, and $(\pi,\mc K)$ is a normal representation of $\mc M$ on another Hilbert space $\mc H$. If $A$ is a closed operator on $\mc H$ affiliated with $\mc M$, then we can define $\pi(A)$ to be a closable operator on $\mc K$ affiliated with $\pi(\mc M)$ as follows. Let $A=UH$ be the polar decomposition of $A$ with $U$ unitary and $H$ positive. Let $\pi(H)$ be the generator of the one parameter group $\pi(e^{\im t H})$. Then we set $\pi(A)=\pi(U)\pi(H)$. $\pi(A)$ is described by the fact that $TA\subset\pi(A)T$ for any bounded linear operator $T:\mc H\rightarrow\mc K$ intertwining the actions of $\mc M$. We note that the ranges of all such $T$ span a dense subspace of $\mc K$; equivalently, the projection of $\mc H\oplus\mc K$ onto $\mc H$ has central support $1$ in the commutant of $\mc M\curvearrowright(\mc H\oplus\mc K)$. If $\mc M$ is type III and $\mc H,\mc K$ are separable, there exists a unitary $T$.

Following \cite{CWX}, when $V$ is strongly local, we say that a unitary energy-bounded $V$-module $W_i$ is \textbf{strongly integrable} if there exists a (necessarily unique) representation $\pi_i$ of $\mc A_V$ on $\mc H_i$, such that for any homogeneous $v\in V$ and any $I\in\mc J,f\in C_c^\infty(I)$,
\begin{align}
\pi_{i,I}(\ovl{Y(v,f)})=\ovl{Y_i(v,f)}.\label{eq36}
\end{align}
Of course, the vacuum module $W_0$ is strongly integrable. Clearly we have:

\begin{lm}\label{lb61}
If $V$ is strongly local and $W_i$ is strongly integrable, then for any homogeneous $v\in V$, $Y_i(v,z)$ satisfies the strong intertwining property.
\end{lm}

\begin{proof}
If $I,J\in\mc J$ are disjoint, and $f\in C_c^\infty(I),g\in C_c^\infty(J)$, then for any homogeneous $u,v\in V$, $\ovl{Y_i(v,f)}$ is affiliated with $\pi_{i,I}(\mc A_V(I))$ and $\ovl{Y_i(u,g)}$ is affiliated with $\pi_{i,I}(\mc A_V(J))$. Thus they commute strongly.
\end{proof}

The following was proved in \cite{Gui21a} proposition 4.9.

\begin{pp}
If $V$ is strongly local and $W_i$ is strongly integrable, then the unique strongly continuous unitary representation $U_i:\GAV\curvearrowright\mc H_i$ making $(\mc H_i,\pi_i)$ conformal covariant agrees with the one in proposition \ref{lb46}.
\end{pp}

We are also interested in a weaker notion: Suppose that $E$ is generating and quasi-primary, then $W_i$ is called \textbf{$E$-strongly integrable} if there exists a (necessarily unique) $\pi_i$ such that \eqref{eq36} holds only when $v\in E$. The following proposition can be proved in the same way as \cite{Gui19b} theorem 4.8.\footnote{The first paragraph of step 2 of the proof of \cite{Gui19b} theorem 4.8. checks that the linearly energy-bounded intertwining operator satisfies the strong intertwining property. Here we already assume the strong intertwining property. So that paragraph can be skipped when proving proposition \ref{lb42}.} The main idea is to use the phases of (the closures of) the smeared intertwining operators $\mc Y(w^{(i)},\wtd f)$ to construct many bounded linear operators $T_\alpha$ satisfying the commutativity relations in \cite{Gui19b} proposition 4.7 for any $v\in E$. 

\begin{pp}\label{lb42}
Let $E$ be a generating set of quasi-primary vectors in $V$. Suppose that $V$ is strongly local (equivalently, $E$-strongly local). Let $\mc Y\in\mc V{k\choose i~j}$ be a non-zero unitary intertwining operator, and assume that $W_i,W_j,W_k$ are irreducible. Suppose that $W_j$ is $E$-strongly integrable, and that there exists a non-zero homogeneous vector $w^{(i)}\in W_i$ such that $\mc Y(w^{(i)},z)$ is energy-bounded and satisfies the $E$-strong intertwining property. Then $W_k$ is $E$-strongly integrable.
\end{pp}

The following proposition is a variant of theorem \ref{lb43}.

\begin{pp}\label{lb47}
In proposition \ref{lb42}, if $W_j$ is strongly-integrable, then $\mc Y(w^{(i)},z)$ satisfies the strong intertwining property, and $W_k$ is strongly integrable. 
\end{pp}
\begin{proof}
Let $\pi_j,\pi_k$ be the $\mc A_V$-modules defined by the $E$-strongly integrable unitary $V$-modules $W_j,W_k$ respectively. Choose any $\wtd I,\wtd J\in\Jtd$ with $\wtd J$ clockwise to $\wtd I$, $\wtd f\in C_c^\infty(\wtd I)$, $g\in C_c^\infty(J)$, and homogeneous $v\in V$.  Let $A=\ovl{\mc Y(w^{(i)},f)}$, $B=\ovl{Y(v,g)}$, $B_j=\ovl{Y_j(v,g)}$, $B_k=\ovl{Y_k(v,g)}$.  By the fact that $\mc Y(w^{(i)},z)$ satisfies the $E$-strong intertwining property, and  that all $\ovl{Y(u,g)}$ (where $u\in E,g\in C_c^\infty(J)$) generate $\mc A_V(J)$ (cf. theorem \ref{lb43}), we have $\pi_{k,J}(y)A\subset A\pi_{j,J}(y)$ for any $y\in\mc A_V(J)$. In other words, the diagram
\begin{align}\label{eq39}
\begin{CD}
\mc H_j @>~\pi_{j,J}(y)~>> \mc H_j\\
@V A  VV @VV  A V\\
\mc H_k @> \pi_{k,J}(y)>> \mc H_k
\end{CD}
\end{align}
commutes strongly for any $y\in\mc A_V(J)$, and hence for any $y$ affiliated with $\mc A_V(J)$. In particular, this is true when we take $y$ to be $B$. 

Since $W_j$ is strongly integrable, we have $\pi_{j,J}(B)=B_j$. Therefore, to finish proving the proposition, it remains to check that $\pi_{k,J}(B)=B_k$. Choose any $\eta\in\mc H_j^\infty$. Since both $\pi_{j,J}(B)=B_j$ and $A$ are smooth, $\eta$ is in the domains of $A^*A$ and $B_j^*B_j$. By the strong commutativity of \eqref{eq39} (with $y=B$) and proposition \ref{lb3}-(b), $\eta$ is in the domains of $\pi_{k,J}(B)A$ and $AB_j$, and $\pi_{k,J}(B)A\eta=AB_j\eta$. Since $B_kA=AB_j$ when acting on $\mc H_j^\infty$ (by theorem \ref{lb63}), we conclude that $\pi_{k,J}(B)$ equals $B_k$ when acting on the subspace $\mc W$ of $\mc H_k^\infty$ spanned by vectors of the form $\mc Y(w^{(i)},\wtd f)\eta$ where $\wtd f\in C_c^\infty(\wtd I)$ and $\eta\in\mc H_j^\infty$. By step 1 of the proof of \cite{Gui19b} theorem 4.8 or the proof of \cite{Gui21a} proposition 4.12, $\mc W$ is a dense subspace of $\mc H_k^\infty$. By proposition \ref{lb44}, $\mc W$ is QRI. Therefore, by proposition \ref{lb45}, $\mc W$ is a core for $B_k$. We conclude $B_k\subset\pi_{k,J}(B)$. 

Now set $\eta=B\Omega$. Since $B$ commutes strongly with $\mc A_V(J')$, we have $\eta\in\mc H_0^\pr(J)$ and $\scr R(\eta,\wtd J)|_{\mc H_0}\subset B$. Note that $B$ is localizable by proposition \ref{lb45}. Since $\mc A_V^\infty(J')\Omega$ is a dense and QRI subspace of $\mc H_0^\infty$, it is a core for $B$. Since $\mc A_V^\infty(J')\Omega\subset \mc A_V(J')\Omega=\mc H_0(J')\subset\Dom(\scr R(\eta,\wtd J)|_{\mc H_0})$,  we must have $\scr R(\eta,\wtd J)|_{\mc H_0}=B$. It is easy to see that $\pi_{k,J}(\scr R(\eta,\wtd J)|_{\mc H_0})=\scr R(\eta,\wtd J)|_{\mc H_k}$. (For example,  choose $\mu\in\mc H_k(I)$ such that $L(\mu,\wtd I)|_{\mc H_0}$ is unitary. Then, use the fact that the representation $\pi_{k,J}$ of $\mc A_V(J)$ and the closed operator $\scr R(\eta,\wtd J)|_{\mc H_k}$ are unitarily equivalent  to $\pi_{0,J}$ and $\scr R(\eta,\wtd J)|_{\mc H_0}$ respectively under the unitary operator $L(\mu,\wtd I)|_{\mc H_0}$.) Thus $\pi_{k,J}(B)=\scr R(\eta,\wtd J)|_{\mc H_k}$, which, by proposition \ref{lb21}, has a core $\mc H_k^\infty(J')$. Since  $\mc H_k^\infty(J')$ is a subspace of $\mc H_k^\infty\subset\Dom(B_k)$, we conclude $\pi_{k,J}(B)=B_k$.
\end{proof}

\begin{rem}\label{lb50}
Proposition \ref{lb47} can be used in the following way. Assume $E\subset V$ is generating and quasi-primary. Given a unitary $W_k$, if we can find irreducible unitary $V$-modules $W_{k_0}=W_0,W_{k_1},\dots,W_{k_{n-1}},W_{k_n}=W_k$ such that for each $1\leq m\leq n$, $W_{k_{m-1}}$ and $W_{k_m}$ are connected by a non-zero unitary intertwining operator $\mc Y_m\in\mc V{k_m\choose {i_m}~k_{m-1}}$ where $W_{i_m}$ is irreducible, and if there exists a non-zero quasi-primary $w_m\in W_{i_m}$ such that $\mc Y_m(w_m,z)$ is energy-bounded and satisfies the $E$-strong intertwining property. Then each  $\mc Y_m(w_m,z)$ satisfies the strong intertwining property, and $W_k$ is strongly integrable.
\end{rem}


\begin{rem}\label{lb99}
Let $E$ be a homogeneous subset of $V$. Suppose that $\mc Y$ is a unitary type $W_k\choose W_iW_j$ intertwining operator. Let $W_j=\bigoplus_b W_{j,a},W_k=\bigoplus_c W_{k,b}$ be the (finite and orthogonal) decomposition of $W_j,W_k$ into irreducible unitary submodules. Let $P_a$ be the projection of $\mc H_j$ to $\mc H_{j,a}$, which automatically restricts to $\mc H_j^\infty\rightarrow\mc H_{j,a}^\infty$ and $W_j\rightarrow W_{j,a}$. Let $Q_b$ be the projection of $\mc H_k$ to $\mc H_{k,b}$, which automatically restricts to $\mc H_k^\infty\rightarrow\mc H_{k,b}^\infty$ and $W_k\rightarrow W_{k,b}$. Choose a homogeneous $w^{(i)}\in W_i$. Then $\mc Y(w^{(i)},z)$ is energy-bounded (resp. satisfies the $E$-strong intertwining property) if and only if $Q_b\mc Y(w^{(i)},z) P_a$ (where $Q_b\mc Y(\cdot,z) P_a$ is viewed as a type $W_{k,b}\choose W_iW_{j,a}$ intertwining operator) is so for each $a,b$.
\end{rem}

\begin{proof}
It is obvious that $\mc Y(w^{(i)},z)$ is energy-bounded iff each $Q_b\mc Y(w^{(i)},z) P_a$ is so. Now choose any $\wtd I\in\Jtd,\wtd f\in C_c^\infty(\wtd I),g\in C_c^\infty(I'),v\in E$. Note that $Q_b\mc Y(w^{(i)},\wtd f)P_a:\mc H_j^\infty\rightarrow\mc H_k^\infty$ can be viewed equivalently as a map $\mc H_{j,a}^\infty\rightarrow\mc H_{k,b}^\infty$, and that the strong intertwining property under the two viewpoints are equivalent. Thus, since $P_a$ commutes strongly with $Y_j(v,g)$ and $Q_b$ commutes strongly with $Y_k(v,g)$, by Lem. \ref{lb59}, the diagram \eqref{eq38} commutes strongly if and only if for each $a,b$, the same diagram \eqref{eq38}, but with $\mc Y(w^{(i)},\wtd f)$ replaced by $Q_b\mc Y(w^{(i)},\wtd f)P_a$, commutes strongly.
\end{proof}

\subsection{The rigid braided $C^*$-tensor categories $\RepuV$ and $\scr C$}\label{lb87}

We assume that $V$ is regular, or equivalently, rational and $C_2$-cofinite. For the definitions of these notions and the proof of equivalence, see \cite{ABD04}. In this case, there are only finitely many inequivalent irreducible $V$-modules (\cite{Hua05b} Cor. 6.6). Consequently, any $V$-module is a finite direct sum of irreducible ones. Moreover, any $\mc V{k\choose i~j}$ is finite dimensional (\cite{Hua05a} theorem 3.1 and remark 3.8). 

By \cite{HL95a,HL95b,HL95c,Hua95,Hua05a} or (under a slightly stronger condition) \cite{NT05}, the category $\Rep(V)$ of  $V$-modules form a  ribbon braided tensor category. A sketch of the construction of braided tensor structure can be found in \cite{Gui19a} section 2.4 or \cite{Gui21a} section 4.1.  We let $\boxdot$ denote the tensor (fusion) bifunctor, and let $\ss$ denote the braiding. $\boxdot$ is defined in such a way that for any $W_i,W_j,W_k\in\Obj(\Rep(V))$ there is a natural isomorphism of vector spaces
\begin{align}
\mc V{k\choose i~j}\simeq\Hom_V(W_i\boxdot W_j,W_k)\label{eq47}
\end{align}
Moreover, by \cite{Hua05b,Hua08a,Hua08b}, $\Rep(V)$ is a modular tensor category. 

Let $\RepuV$ be the $C^*$-category of unitary $V$-modules. Since $V$ is assumed to be strongly unitary, $\RepuV$ is equivalent to $\RepV$ as linear categories. If $W_i,W_j$ are unitary, then $W_i\boxdot W_j$ is a unitarizable $V$-module. One needs to choose a correct unitary $V$-module structure on each $W_i\boxdot W_j$  making $\RepuV$ a unitary (i.e. $C^*$-) tensor category. This amounts to finding the correct inner product on each $\mc V{k\choose i~j}$. This was achieved in \cite{Gui19a,Gui19b}. In these two papers, we defined a (necessarily non-degenerate and Hermitian) sesquilinear form $\Lambda$ on each $\mc V{k\choose i~j}$, called the \textbf{invariant sesquilinear form}. We say that a strongly unitary regular VOA $V$ is \textbf{completely unitary} if $\Lambda$ is positive on any vector space of intertwining operator of $V$. We proved in \cite{Gui19b} theorem 7.9 that \emph{if $V$ is completely unitary, then $\RepuV$ is a unitary modular tensor category}. In particular, $\RepuV$ is a rigid braided $C^*$-tensor category. 

As in section \ref{lb49}, we say that a set $\mc F^V$ of irreducible unitary $V$-modules  \textbf{generates} $\RepV$ if any irreducible $V$-module $W_i$ is equivalent to a submodule of a (finite) fusion product of irreducible $V$-modules in $\mc F^V$. By \eqref{eq47}, this is equivalent to saying that for any irreducible unitary $V$-module $W_k$, there is a chain of non-zero unitary irreducible intertwining operators $\mc Y_1,\mc Y_2,\dots,\mc Y_n$ with charge spaces in $\mc F^V$ such that $\mc Y_1$ has source space $W_0=V$, $\mc Y_n$ has target space $W_k$, and the source space of $\mc Y_m$ equals the target space of $\mc Y_{m-1}$ for any $m=2,3,\dots,n$.

Consider the following two conditions:

\begin{cond}\label{cd1}
We assume that $V$ is a regular, strongly unitary, strongly energy-bounded, and simple (equivalently, CFT-type) VOA. Assume that $E$ is a generating set of quasi-primary vectors in $V$ such that $V$ is $E$-strongly local. Assume that $\mc F^V$ is a set of irreducible unitary $V$-modules generating $\RepV$. Assume that for each $W_i\in \mc F^V$ there exists a non-zero homogeneous $w^{(i)}_0\in W_i$ such that for any irreducible unitary intertwining operator $\mc Y$ with charge space $W_i$, $\mc Y(w^{(i)}_0,z)$ is energy-bounded \footnote{It follows from proposition \ref{lb53} that $\mc Y$ is energy-bounded.} and satisfies the $E$-strong intertwining property.
\end{cond}

\begin{cond}\label{cd2}
We assume that condition \ref{cd1} is satisfied, and that for each $W_i\in\mc F^V$, the vector $w^{(i)}_0\in W_i$ in condition \ref{cd1} is quasi-primary.
\end{cond}

In the remaining part of this article,  we shall always assume  condition \ref{cd1}. The following is essentially proved in \cite{Gui19b}. (See \cite{Gui21a} remark 4.21 and theorem 4.22 for a detailed explanation.)

\begin{thm}\label{lb51}
Assume that condition \ref{cd1} is satisfied.  Then $V$ is completely unitary. Consequently, $\RepuV$ is a unitary modular tensor category.
\end{thm} 

Moreover, by proposition \ref{lb47} and  \ref{lb50}, we have:

\begin{thm}\label{lb55}
Assume that condition \ref{cd1} is satisfied. Then any unitary $V$-module $W_k$ is strongly integrable. Moreover, for any $w^{(i)}_0$ and $\mc Y$ in condition \ref{cd1}, $\mc Y(w^{(i)}_0,z)$ satisfies the strong intertwining property.
\end{thm}

Thus,  each $W_i\in\RepuV$ is associated with an $\mc A_V$-module $\mc H_i$. We define a $*$-functor\footnote{A functor $\fk F$ is called a $*$-functor if $\fk F(G^*)=\fk F(G)^*$ for any morphism $G$.} $\fk F:\RepuV\rightarrow \Rep(\mc A_V)$ sending each $W_i$ to $\mc H_i$. If $G\in\Hom_V(W_i,W_j)$, then $\fk F(G):\Hom_{\mc A_V}(\mc H_i,\mc H_j)$ is the closure of the densely defined continuous map $G$.

\begin{thm}[\cite{CWX}, \cite{Gui19b} Thm. 4.3]\label{lb54}
Assume that condition \ref{cd1} is satisfied. Then $\fk F$ is a fully faithful $*$-functor. Namely, for any unitary $V$-modules $W_i,W_j$, the linear map $\fk F:\Hom_V(W_i,W_j)\rightarrow\Hom_{\mc A_V}(\mc H_i,\mc H_j)$ is an isomorphism.
\end{thm}

Let $\scr C=\fk F(\RepuV)$. Then one can define the inverse $*$-functor $\fk F^{-1}$ as follows. Choose any $\mc H_i\in\Obj(\scr C)$. Then there exists a unitary $W_i$ such that $\fk F(W_i)=\mc H_i$. We let $\fk F^{-1}(\mc H_i)=W_i$. This is well defined. Indeed, if $W_{i'}$ is another unitary $V$-module such that $\fk F(W_{i'})=\mc H_i$, then $W_i$ and $W_{i'}$ are both dense subspaces of $\mc H_i$. By theorem \ref{lb54}, the morphism $\id_{\mc H_i}:\fk F(W_i)\rightarrow\fk F(W_{i'})$ arises from a morphism $G\in\Hom_V(W_i,W_{i'})$. Since $\id_{\mc H_i}=\fk F (G)$, and since $\fk F(G)$ is the closure of $G$, $G$ must be $\id_{W_i}$. This proves that $W_i=W_{i'}$ both as pre-Hilbert spaces and as unitary $V$-modules.

\begin{df}\label{lb101}
The above results show that the $*$-functor $\fk F:\RepuV\rightarrow\scr C$ is a $*$-isomorphism of $C^*$-categories with inverse $\fk F^{-1}$. Therefore, we can define a unique$C^*$-tensor category structure $(\scr C,\boxdot,\ss)$ such that $\fk F:\RepuV\rightarrow\scr C$ is a $*$-isomorphism of braided $C^*$-tensor categories. In particular, for any $\mc H_i,\mc H_j\in\Obj(\scr C)$, we define $\mc H_i\boxdot\mc H_j=\fk F(\fk F^{-1}(\mc H_i)\boxdot\fk F^{-1}(\mc H_j))$, and write $\mc H_i\boxdot\mc H_j$ as $\mc H_{ij}$ for short. We let $\mc H_0=\fk F(W_0)$ be the unit object. The structure isomorphisms (i.e. unitors, associators, braiding isomorphisms) of $\scr C$ are transported  from those of $\RepuV$ via $\fk F$.
\end{df}

\subsection{The weak vertex categorical partial extension $\scr E^w_\loc$}\label{lb68}

The most important results of this section are corollaries \ref{lb71} and \ref{lb65}.

$\scr C=\fk F(\RepuV)$ is a semi-simple $C^*$-subcategory of $\Rep(\mc A_V)$. We shall show that $\scr C$ is a braided $C^*$-tensor subcategory of $\Rep(\mc A_V)$ by constructing a weak categorical partial extension $\scr E^w_\loc=(\mc A_V,\mc F,\boxdot,\fk H)$ where $\mc F=\fk F(\mc F^V)$ clearly generates $\scr C$.

For each unitary $W_i,W_j$, we write $W_{ij}=W_i\boxdot W_j$. Thus, by our notations, $\mc H_{ij}$ is the associated $\mc A_V$-module, and $\mc H_{ij}^\infty$ is the dense subspace of smooth vectors. By \cite{Gui21a} section 4.2, we have $\mc L^V$ associates to each unitary $W_i,W_k$  an intertwining operator $\mc L^V$ of type ${ik\choose i~k}={W_i\boxdot W_k\choose W_i~W_k}$ such that for any unitary $W_l$ and any $\mc Y\in\mc V{l\choose i~k}$, there exists a unique $T\in\Hom_V(W_i\boxdot W_k,W_l)$ such that $\mc Y(w^{(i)},z)w^{(k)}=T\mc L^V(w^{(i)},z)w^{(k)}$ for any $w^{(i)}\in W_i,w^{(k)}\in W_k$. Moreover,  we define $\mc R^V$ to be a type ${ki\choose i~k}={W_k\boxdot W_i\choose W_i~W_k}$ intertwining operator by setting
\begin{align}
\mc R^V(w^{(i)},z)w^{(k)}=\ss_{i,k}\mc L^V(w^{(i)},z)w^{(k)},
\end{align}
We understand $\mc L^V(w^{(i)},z)$ and $\mc R^V(w^{(i)},z)$ (and their smeared intertwining operators) as categorical operators acting on any possible unitary $W_k$. We write them as $\mc L^V(w^{(i)},z)|_{W_k}$ and $\mc R^V(w^{(i)},z)|_{W_k}$ if we want to emphasize that they are acting on $W_k$. (Note that in \cite{Gui21a}, these two operators are written as $\mc L_i(w^{(i)},z)$ and $\mc R_i(w^{(i)},z)$.) For any $v\in W_0=V$ we have
\begin{align}
\mc L^V(v,z)|_{W_k}=\mc R^V(v,z)|_{W_k}=Y_k(v,z).
\end{align}
where we have identified $W_{0k}$ and $W_{k0}$ with $W_k$.

We recall some important properties of $\mc L^V$ and $\mc R^V$ proved in \cite{Gui21a} chapter 4.

\begin{pp}[\cite{Gui21a} Eq. (4.24), (4.25)]\label{lb57}
$\mc L^V$ and $\mc R^V$ are natural. More precisely, for any unitary $V$-modules $W_i,W_{i'},W_k,W_{k'}$, any $F\ \in\Hom_V(W_i,W_{i'}),G\in\Hom_V(W_k,W_{k'})$, and any $w^{(i)}\in W_i,w^{(k)}\in W_k$
\begin{gather}
(F\boxdot G)\mc L^V(w^{(i)},z)w^{(k)}=\mc L^V(Fw^{(i)},z)Gw^{(k)},\label{eq70}\\
(G\boxdot F)\mc R^V(w^{(i)},z)w^{(k)}=\mc R^V(Fw^{(i)},z)Gw^{(k)}.\label{eq74}
\end{gather}
\end{pp}

Note that by the semisimpleness of (unitary) $V$-modules, if $W_i$ is irreducible and $w^{(i)}\in W_i$ is homogeneous, then $\mc L^V(w^{(i)},z)$ is energy-bounded (when acting on any unitary $W_k$) if and only if  for any irreducible unitary intertwining operator $\mc Y$ with charge space $W_i$, $\mc Y(w^{(i)},z)$ is energy-bounded.

\begin{pp}[\cite{Gui21a} Prop. 4.12]\label{lb58}
Suppose that $W_i,W_j$ are unitary $V$-modules, $W_i$ is irreducible, $w^{(i)}\in W_i$ is a non-zero homogeneous vector, and $\mc L^V(w^{(i)},z)|_{W_j}$ is energy-bounded and satisfies the strong intertwining property.  Then for each  $\wtd I\in\Jtd$, vectors of the form $\mc L^V(w^{(i)},\wtd f)w^{(j)}$ (where $\wtd f\in C_c^\infty(\wtd I)$ and $w^{(j)}\in W_j$) span a dense subspace of $\mc H_{ij}^\infty$.
\end{pp}

Note that by proposition \ref{lb44}, the dense subspace in the above proposition is also QRI. The following weak locality of smeared $\mc L^V$ and $\mc R^V$  proved in \cite{Gui21a} theorem 4.8 is a generalization of theorem \ref{lb63}.

\begin{thm}[Weak braiding]\label{lb60}
Choose unitary $V$-modules $W_i,W_j,W_k$, homogeneous $w^{(i)}\in W_i,w^{(j)}\in W_j$, and disjoint $\wtd I,\wtd J\in \Jtd$ with $\wtd J$ is clockwise to $\wtd I$. Assume that  $\mc L^V(w^{(i)},z)|_{W_k},\mc L^V(w^{(i)},z)|_{W_{kj}},\mc R^V(w^{(j)},z)|_{W_k},\mc R^V(w^{(j)},z)|_{W_{ik}}$ are energy-bounded. Then for any $\wtd f\in C_c^\infty(\wtd I)$ and $\wtd g\in C_c^\infty(\wtd J)$, the following diagram commutes adjointly.
	\begin{align}\label{eq42}
	\begin{CD}
	\mc H^\infty_k @> \quad\mc R^V(w^{(j)},\wtd g)\quad >> \mc H^\infty_{kj}\\
	@V {\mc L^V(w^{(i)},\wtd f)} VV @VV {\mc L^V(w^{(i)},\wtd f)} V\\
	\mc H^\infty_{ik} @> \quad\mc R^V(w^{(j)},\wtd g)\quad>> \mc H^\infty_{ikj}
	\end{CD}
	\end{align}
\end{thm}

We now construct the weak vertex categorical partial extension. Note that $\mc F:=\fk F(\mc F^V)$ generates the $C^*$-tensor category $\scr C=\fk F(\RepuV)$. Recall that $\mc H_i=\fk F(W_i)$.

\begin{thm}
Suppose that condition \ref{cd1} is satisfied. Then there exists a natural categorical partial extension $\scr E^w_\loc=(\mc A_V,\mc F,\boxdot,\fk H)$ of $\mc A_V$ obtained by smeared vertex operators and smeared intertwining operators. Moreover, if condition \ref{cd2} is satisfied, then $\scr E^w_\loc$ is M\"obius covariant.
\end{thm}

\begin{proof}
For each $W_i\in\mc F^V$, $w^{(i)}_0$ always denotes the non-zero and homogeneou vector mentioned condition \ref{cd1}. Then by theorem \ref{lb55}, for any irreducible unitary intertertwining operator $\mc Y$ with charge space $W_i\in\mc F^V$, $\mc Y(w^{(i)}_0,z)$ is energy-bounded and satisfies the strong intertwining property. 

Choose any $\mc H_i\in\mc F$ (equivalently, choose any $W_i\in\mc F^V$). For each $\wtd I\in\Jtd$, we define
\begin{align}
\fk H_i(\wtd I)_n=&C_c^\infty(\wtd I)\times C_c^\infty(I)^{\times n}\times E^{\times n}\nonumber\\
=&C_c^\infty(\wtd I)\times\underbrace{C_c^\infty(I)\times C_c^\infty(I)\times\cdots\times C_c^\infty(I)}_{n}\times\underbrace{E\times E\times\cdots\times E}_{n} 
\end{align}
for each $n=0,1,2,\dots$, and define 
\begin{align}
\fk H_i(\wtd I)=\coprod_{n\in\mbb N}\fk H_i(\wtd I)_n.
\end{align}
The inclusion $\fk H_i(\wtd I)\subset \fk H_i(\wtd J)$ (when $\wtd I\subset\wtd J$) is defined in an obvious way. For each
\begin{align*}
\fk a=(\wtd f;f_1,f_2,\dots,f_n;u_1,u_2,\dots,u_n)\in C_c^\infty(\wtd I)\times C_c^\infty(I)^{\times n}\times E ^{\times n},
\end{align*}
define smooth operators $\mc L(\fk a,\wtd I)$ and $\mc R(\fk a,\wtd I)$ which, for each $W_k\in\Obj(\RepuV)$, map from $\mc H_k^\infty$ to $\mc H_{ik}^\infty$ and $\mc H_{ik}^\infty$ respectively as
\begin{gather}
\mc L(\fk a,\wtd I)|_{\mc H_k^\infty}=\mc L^V(w^{(i)}_0,\wtd f)Y_k(u_1,f_1)\cdots Y_k(u_n,f_n),\label{eq40}\\
\mc R(\fk a,\wtd I)|_{\mc H_k^\infty}=\mc R^V(w^{(i)}_0,\wtd f)Y_k(u_1, f_1)\cdots Y_k(u_n,f_n).\label{eq41}
\end{gather}
Then, by the definition of $\mc R^V$, one clearly has $\mc R(\fk a,\wtd I)|_{\mc H_k^\infty}=\ss_{i,k}\mc L(\fk a,\wtd I)|_{\mc H_k^\infty}$. 

We now check that the axioms of  weak categorical partial extensions are satisfied. Isotony is obvious. Naturality follows from proposition \ref{lb57}. Braiding is proved. Since $\ss_{i,0}=\id_i$, one has the neutrality. Density of fusion products follows from proposition \ref{lb58}. The Reeh-Schlieder property follows from the fact that $E$ generates $V$, and that vectors of the form $Y(u_1,f_1)\cdots Y(u_n,f_n)\Omega$ (where $u_1,\dots,u_n\in E,f_1,\dots,f_n\in C_c^\infty(I)$) span a dense and clearly QRI subspace of $\mc H_0^\infty$. (The density of this subspace follows from a standard argument using Schwarz reflection principle; see the first paragraph of the proof of \cite{CKLW18} theorem 8.1.) Since each factor on the right hand sides of \eqref{eq40} and \eqref{eq41} satisfies the strong intertwining property (recall lemma \ref{lb61} and remark \ref{lb99}), the axiom of intertwining property for $\scr E^w_\loc$ can be proved using lemma \ref{lb59}. The weak locality follows from theorem \ref{lb60}. The rotation or M\"obius covariance follows from proposition \ref{lb44}.
(Note also \eqref{eq51}.)  
\end{proof}

Combine the above theorem with theorems \ref{lb40} and \ref{lb34} and corollary \ref{lb35}, we obtain the following corollaries.

\begin{co}[Isomorphism of tensor categories]\label{lb71}
If condition \ref{cd1} is satisfied, then $\RepuV$ is isomorphic to a braided $C^*$-tensor subcategory of $\Rep(\mc A_V)$ under the $*$-functor $\fk F$.
\end{co}

\begin{pp}\label{lb78}
If condition \ref{cd1} is satisfied, then for any $W_i,W_j\in\mc F^V$, $W_k\in\Obj(\RepuV)$, $\wtd J$ clockwise to $\wtd I$, and $\wtd f\in C_c^\infty(\wtd I),\wtd g\in C_c^\infty(\wtd J)$, the diagram \eqref{eq42} of closable operators commutes strongly if we set $w^{(i)}=w^{(i)}_0,w^{(j)}=w^{(j)}_0$.
\end{pp}

From the construction of $\scr E^w_\loc$, it is clear that we have:

\begin{pp}\label{lb64}
Suppose that condition \ref{cd1} is satisfied. Choose any $W_i\in\Obj(\RepuV)$ and any $w^{(i)}\in W_i$ such that $\mc Y(w^{(i)},z)$ is energy-bounded for any irreducible unitary intertwining operator $\mc Y$ with charge space $W_i$. Let $\wtd I\in\Jtd,f\in C_c^\infty(\wtd I)$,  set $\fk x=(w^{(i)},\wtd f)$, and set
\begin{align}
A(\fk x,\wtd I)=\mc L^V(w^{(i)},\wtd f),\qquad B(\fk x,\wtd I)=\mc R^V(w^{(i)},\wtd f).
\end{align}
Then $(A,\fk x,\wtd I,\mc H_i)$ and $(B,\fk x,\wtd I,\mc H_i)$ are respectively weak left and right operators of $\scr E^w_\loc$.
\end{pp}

By proposition \ref{lb64} and corollary \ref{lb62}, we have:

\begin{co}[Strong braiding]\label{lb65}
Suppose that condition \ref{cd2} is satisfied. Assume that  $W_i,W_j\in\Obj(\RepuV)$ satisfy the condition that any  unitary intertwining operator with charge space $W_i$ or $W_j$ is energy-bounded.  Then, for any $W_k\in\Obj(\RepuV)$, $\wtd J$ clockwise to $\wtd I$, $\wtd f\in C_c^\infty(\wtd I),\wtd g\in C_c^\infty(\wtd J)$, and homogeneous $w^{(i)}\in W_i,w^{(j)}\in W_j$, the diagram \eqref{eq42} of closable operators commutes strongly.
\end{co}

Note that the assumption on $W_i,W_j$ hold if we choose $W_i,W_j\in\mc F^V$.

\begin{rem}\label{lb72}
In the above corollary, if we take $W_j$ to be $W_0$, then we see that $\mc L^V$ satisfies the strong intertwining property. Since any unitary intertwining operator $\mc Y$ with charge space $W_i$ can be expressed as $\mc Y\mc L^V$ where $\mc Y$ is a morphism which is clearly smooth and rotation invariant, by lemma \ref{lb59}, $\mc Y$ satisfies the strong intertwining property. 
\end{rem}

By corollary \ref{lb65} and the above remark, we immediately have:

\begin{pp}\label{lb88}
Assume that condition \ref{cd2} is satisfied. Then  a unitary intertwining operator is energy-bounded and satisfies the strong intertwining property if its charge space is a (finite) direct sum of unitary $V$-modules in $\mc F^V$.
\end{pp}

Recall that we say that $V$ is completely energy-bounded if any  unitary irreducible intertwining operator of $V$ is energy-bounded. If $V$ is completely unitary and completely energy-bounded, we say that the unitary intertwining operators of $V$ satisfy \textbf{strong braiding}, if for any $W_i,W_j,W_k\in\Obj(\RepuV)$, $w^{(i)}\in W_i,w^{(j)}\in W_j$ being homogeneous, $\wtd J$ clockwise to $\wtd I$, and $\wtd f\in C_c^\infty(\wtd I),\wtd g\in C_c^\infty(\wtd J)$, the diagram \eqref{eq42} of closable operators commutes strongly. The following corollary follows immediately from corollary \ref{lb65}.

\begin{co}\label{lb80}
Assume that $V$ satisfies condition \ref{cd2} and is completely energy-bounded, then the unitary intertwining operators of $V$ satisfy strong braiding and (in particular) the strong intertwining property.
\end{co}

\subsection{Proving conditions \ref{cd1} and \ref{cd2}}

In this section, we give two useful methods of proving conditions \ref{cd1} and \ref{cd2}. The first one concerns sub-VOAs and coset VOAs, the second one is about tensor product VOAs.

\subsubsection{Compression principle}

The goal of this subsection is to show that compressing intertwining operators preserves the polynomial energy bounds and the strong intertwining property. Let $U$ be a unitary simple VOA with vertex operator $Y^U$, and let $V$ a unitary sub-VOA of $U$ with vertex operator $Y$. This means that $V$ is a graded subspace of $U$, that $U$ and $V$ share the same vacuum vector $\Omega$, that $Y^U(v_1,z)v_2=Y(v_1,z)v_2$ for all $v_1,v_2\in V$, and that there is $\nu\in V$ such that $V$ (with the inner product $\bk {\cdot|\cdot}$ inherited from $U$) is a unitary VOA with conformal vector $\nu$. Let $V^c$ be the set of all $u\in U$ satisfying $Y(v)_nu=0$ for any $n\in\mbb N$. Then by \cite{CKLW18} proposition 5.31, $V^c$ is a unitary VOA (the coset VOA) with conformal vector $\nu'$ such that $\nu+\nu'$ is the conformal vector of $U$. Moreover, let $Y'$ be the vertex operator of $V^c$. Then the linear map
\begin{align*}
V\otimes V^c\rightarrow U,\qquad v\otimes v'\mapsto Y(v)_{-1}Y(v')_{-1}\Omega
\end{align*}
is an isometry and a homomorphism of $V\otimes V^c$-modules. Thus one can regard $V\otimes V^c$ as a (conformal) sub-VOA of $U$ with the same conformal vector $\nu+\nu'=\nu\otimes\Omega+\Omega\otimes\nu'$. Note that we have the identification $v=v\otimes\Omega,v'=\Omega\otimes v'$ (where $v\in V,v'\in V^c$). We let $L_n=Y(\nu)_{n+1},L'_n=Y(\nu')_{n+1}$. Note also that $V$ and $V^c$ are clearly CFT-type. Thus they are simple.

Let $(W_{\mbb I},Y_{\mbb I})$ be a unitary $U$-module. Then $W_{\mbb I}$ restricts to a weak $V$-module. We say that $(W_i,Y_i)$  is  a graded irreducible $V$-submodule of $W_{\mbb I}$ if $W_i$ is a $L_0^U$-graded subspace of $W_{\mbb I}$ (equivalently, $W_i$ is $L_0^U$-invariant subspace),  $Y_i(v,z)w^{(i)}=Y_{\mbb I}(v,z)w^{(i)}$ for each $v\in V,w^{(i)}\in W_i$, and $(W_i, Y_i)$ is an irreducible (ordinary) $V$-module.

\begin{pp}\label{lb76}
If $W_i$ is a graded irreducible $V$-submodule of $W_{\mbb I}$, then there exists $\lambda\in\mbb R$ such that $L_0^U|_{W_i}=L_0|_{W_i}+\lambda\id_{W_i}$.
\end{pp}
\begin{proof}
Since the action of $L_0^U$ on $V\otimes V^c$ equals $L_0\otimes \id+\id\otimes L_0'$, it is clear that $L_0|_V=L_0^U|_V$. Therefore, when acting on $W_i$, we have for each $v\in V$ that
\begin{align*}
&[L_0^U, Y_i(v,z)]=[L_0^U, Y_{\mbb I}(v,z)]=Y_{\mbb I}(L_0^Uv,z)+z\partial_zY_{\mbb I}(v,z)\\
=&Y_i(L_0v,z)+z\partial_zY_i(v,z)=[L_0,Y_i(v,z)],
\end{align*}
which shows that $A:=L_0^U|_{W_i}-L_0|_{W_i}$ commutes with the action of $V$ on $W_i$. Since $W_i$ is irreducible, $A$ must be a constant.
\end{proof}

Recall that $\mc H_{\mbb I}$ is the Hilbert space completion of $W_{\mbb I}$, and $\mc H^\infty_{\mbb I}$ is the subspace of smooth vectors (defined by $L_0^U$). By proposition \ref{lb76}, $\mc H_i^\infty$ can be defined unambiguously using either $L_0$ or $L_0^U$. One can thus regard $\mc H_i^\infty$ as a subspace of $\mc H_i$ and also of $\mc H_{\mbb I}^\infty$. If $p_i$ is the projection of $\mc H_{\mbb I}$ onto $\mc H_i$, then, as each $e^{\im t\ovl{L^U_0}}$ leaves $\mc H_i$ invariant, $e^{\im t\ovl{L^U_0}}$ commutes (adjointly) with $p_i$, which shows that $L^U_0$ commutes strongly with $p_i$. Thus $p_i$ is smooth. In the following, we shall always consider $p_i$ as a densely defined continuous operator with domain $\mc H_{\mbb I}^\infty$.

Let $W_{\mbb J},W_{\mbb K}$ be unitary $U$-modules, and let $W_j$ and $W_k$ be respectively graded irreducible $V$-submodules of $W_{\mbb J},W_{\mbb K}$. Let $\mc V{k\choose i~j}$ and $\mc V^U{\mbb K\choose \mbb I~\mbb J}$ be the vector spaces of type $k\choose i~j$ intertwining operators of $V$ and type $\mbb K\choose \mbb I~\mbb J$ intertwining operators of $U$ respectively. Define  projections $p_j$ and $p_k$ in a similar way. If $\mc Y\in\mc V{k\choose i~j}$ and $\mc Y^U\in\mc V^U{\mbb K\choose \mbb I~\mbb J}$, we say that $\mc Y$ is a \textbf{compression} of $\mc Y^U$ if there exists $\lambda\in\mbb R$ such that for any $w^{(i)}\in W_i,w^{(j)}\in W_j,s\in\mbb R$,
\begin{align}
\mc Y(w^{(i)})_sw^{(j)}=p_k\mc Y^U(w^{(i)})_{s+\lambda}w^{(j)}.\label{eq50}
\end{align}
By proposition \ref{lb76}, it is clear that $\mc Y(w^{(i)},z)$ is ($L_0$-) energy-bounded if $\mc Y^U(w^{(i)},z)$ is ($L_0^U$-) energy-bounded. Thus we have:

\begin{pp}\label{lb90}
	Let $\mc Y$ be a compression of $\mc Y^U$. If $\mc Y^U$ is energy-bounded, then so is $\mc Y$.
\end{pp}

It is clear that $Y_i$  is a compression of $Y^U_{\mbb I}$ (the vertex operator of the $U$-module $W_{\mbb I}$). Thus, as a special case of the above proposition, we see that \emph{$W_i$ is energy-bounded if $W_{\mbb I}$ is so}.

We now show that the strong intertwining property is preserved by compressing intertwining operators. Let $\wtd f=(f,\arg_I)\in C_c^\infty(\wtd I)$. By \eqref{eq50}, we have the following natural identification
\begin{align}
\mc Y(w^{(i)},\wtd f)=p_k\mc Y^U(w^{(i)},\wtd f_\lambda)p_j
\end{align}
where $\wtd f_\lambda=(f_\lambda,\arg_I)$ and $f_\lambda(e^{\im t})=e^{\im\lambda t}f(e^{\im t})$.

\begin{pp}\label{lb91}
	Let $\mc Y\in\mc V{k\choose i~j}$ be a compression of  $\mc Y^U\in\mc V^U{\mbb K\choose \mbb I~\mbb J}$, and choose a homogeneous $w^{(i)}\in W_i$. Assume  $W_{\mbb I},W_{\mbb J},W_{\mbb K}$ are  energy-bounded, and that $\mc Y^U(w^{(i)},z)$ is energy-bounded and satisfies the strong intertwining property. Then the same is true for $\mc Y$.
\end{pp}

\begin{proof}
	Choose any $\wtd I\in\Jtd$ and $\wtd f\in C_c^\infty(\wtd I)$. Choose $J$ disjoint from $I$, and choose $g\in C_c^\infty(J)$. Assume that \eqref{eq50} is satisfied. Choose any homogeneous vector $v\in V$. Let $Y_{\mbb J}^U,Y_{\mbb K}^U$ be the vertex operators of $W_{\mbb J},W_{\mbb K}$ respectively. Then the diagram
	\begin{align}
	\begin{CD}
	\mc H_{\mbb J}^\infty @>~Y_{\mbb J}^U(v,g)~>> \mc H_{\mbb J}^\infty\\
	@V \mc Y^U(w^{(i)},\wtd f_\lambda)  VV @VV \mc Y^U(w^{(i)},\wtd f_\lambda) V\\
	\mc H_{\mbb K}^\infty @> Y^U_{\mbb K}(v,g)>> \mc H_{\mbb K}^\infty
	\end{CD}
	\end{align}	
	commutes strongly. Set $\mc H=\mc H_{\mbb J}\oplus \mc H_{\mbb J}\oplus \mc H_{\mbb K}\oplus \mc H_{\mbb K}$ with smooth subspace $\mc H^\infty=\mc H_{\mbb J}^\infty\oplus \mc H_{\mbb J}^\infty\oplus \mc H_{\mbb K}^\infty\oplus \mc H_{\mbb K}^\infty$. Let $R$  (resp. $S$) be  the extension of  $\mc Y^U(w^{(i)},\wtd f_\lambda)$ and $\mc Y^U(w^{(i)},\wtd f_\lambda)$ (resp. $Y_{\mbb J}^U(v,g)$ and $Y_{\mbb K}^U(v,g)$). (See definition \ref{lb77} for details.) Then the smooth closable operators $R$ and $S$ commute strongly. Define a projection $P$ on $\mc H$ (with domain $\mc H^\infty$) to be $\diag(p_j,p_j)$ on the subspace $\mc H_{\mbb J}^\infty\oplus \mc H_{\mbb J}^\infty$, and $\diag(p_k,p_k)$ on the subspace $\mc H_{\mbb K}^\infty\oplus \mc H_{\mbb K}^\infty$. Then $P$ commutes strongly with $L^U_0$ and with $S$. Thus, by lemma \ref{lb59}, $PRP$ commutes strongly with $S$. Obviously, $PRP$  commutes strongly with $P$. Thus, by lemma \ref{lb59} again, $PRP$ commutes strongly with $SP$. This is equivalent to the strong commutativity of diagram \eqref{eq38}.
\end{proof}

\subsubsection{Coset VOAs}

Assume that $V\subset U$ is as above. Assume also that  $U,V,V^c$ are all regular. The goal of this subsection is to show that the coset VOA $V^c$ satisfies condition \ref{cd2} if $U$ and $V$ do. We first need a preliminary result.

\begin{lm}\label{lb89}
Let $W_{\mbb I},W_{\mbb J},W_{\mbb K}$ be irreducible unitary $U$-modules satisfying $\dim\mc V^U{\mbb K\choose\mbb I~\mbb J}>0$. Let $W_k$ be a graded irreducible  $V$-submodule of $W_{\mbb K}$. Then $W_{\mbb I}$ and $W_{\mbb J}$ have graded irreducible $V$-submodules $W_i,W_j$ respectively satisfying $\dim\mc V{k\choose i~j}>0$.
\end{lm}

\begin{proof}
Choose any nonzero $\mc Y^U\in\mc V^U{\mbb K\choose \mbb I~\mbb J}$. Then, by the irreducibility of $W_{\mbb K}$, the restriction of $\mc Y^U$ to $W_{\mbb I},W_{\mbb J},W_k$ is non-zero. Since $V\otimes V^c$ is a regular conformal sub-VOA of $U$,   $W_{\mbb I},W_{\mbb J}$ are finite sums of (graded) irreducible $V\otimes V^c$-modules which are tensor products of irreducible $V$-modules and $V^c$-modules. So $W_{\mbb I},W_{\mbb J}$ are direct sums of graded irreducible $V$-submodules. Thus, there must exist graded irreducible $V$-submodules $W_i,W_j$ such that the restriction $\wtd{\mc Y}$ of $\mc Y^U$ to $W_i,W_j,W_k$ is non-zero. $\wtd{\mc Y}$ satisfies Jacobi identity, and by proposition \ref{lb76}, we can find $\lambda\in\mbb R$ such that $\mc Y=z^\lambda\wtd{\mc Y}$ also satisfies $z\partial_z\mc Y(w^{(i)},z)=[L_0,\mc Y(w^{(i)},z)]+\mc Y(L_0w^{(i)},z)$ for any $w^{(i)}\in W_i$, which is equivalent to the $L_{-1}$-derivative property. So $\mc Y$ is a non-zero element of $\mc V{k\choose i~j}$.
\end{proof}

\begin{thm}\label{lb92}
Assume that $U$ satisfies condition \ref{cd2}, and $V$ is  unitary sub-VOA of $U$. Assume that both $V$ and its commutant $V^c$ are regular, and that $V^{cc}=V$. Then $V$ and $V^c$ satisfy condition \ref{cd2}.
\end{thm}
Note that the condition $V^{cc}=V$ is equivalent to that $V$ is a coset of another unitary sub-VOA, i.e. there is a unitary sub-VOA $\wtd V\subset U$ such that $V=\wtd V^c$.

\begin{proof}
By \cite{KM15} theorem 2, any irreducible $V$-module is a graded irreducible $V$-submodule of an irreducible $U$-module which is unitarizable and energy-bounded. So $V$ is strongly unitary and strongly energy-bounded. Since $U$ is strongly local, so is $V$. Let $\mc F^U$ be the set of irreducible unitary $U$-modules generating $\Rep(U)$ as in condition \ref{cd1}. Then, by proposition \ref{lb88}, any unitary intertwining operator of $U$ whose charge space is in $\mc F^U$ is energy-bounded and satisfies the strong intertwining property. Let $\mc F^V$ be the set of all graded irreducible (unitary) $V$-submodules of the $U$-modules in $\mc F^U$. Then $\mc F^V$ generates $\Rep(V)$ by lemma \ref{lb89}. Let $\mc Y$ be an irreducible unitary intertwining operator of $V$ whose charge space is in $\mc F^V$. Then, by \cite{Gui20a} theorem 4.2, $\mc Y$ is a sum of  compressions of  irreducible unitary intertwining operators of $U$ whose charge spaces are in $\mc F^U$. By propositions \ref{lb90} and  \ref{lb91}, each compression satisfies polynomial energy bounds and  the strong intertwining property. By lemma \ref{lb59}, the same is true for $\mc Y$. This proves that $V$ satisfies condition \ref{cd2}. Since $V^c$ and $V$ satisfy similar assumptions,  $V^c$ also satisfies condition \ref{cd2}.
\end{proof}

The following is the main result of this section.

\begin{thm}[Condition \ref{cd2} for coset VOAs]\label{lb93}
Assume that $U$ satisfies condition \ref{cd2}, and $V$ is  unitary sub-VOA of $U$. Assume that both $V$ and its commutant $V^c$ are regular, and that $V$ is completely unitary (which holds when $V$ satisfies condition \ref{cd2}), then $V^c$ satisfies condition \ref{cd2}.
\end{thm}

\begin{proof}
Note that $V^c=V^{ccc}$. To apply theorem \ref{lb92}, it remains to check that $V^{cc}$ is regular. Note that $V^{cc}$ is a unitary conformal extension of $V$. Since $V$ is completely unitary, $\Repu(V)$ is a unitary modular tensor category. So $\dim_{\Rep(V)}(V^{cc})>0$. Thus, by the proof of \cite{McR19} theorem 4.14, $V^{cc}$ is rational and $C_2$-cofinite, i.e. regular.
\end{proof}

\begin{co}\label{lb79}
Assume that $U$ and $V$ satisfy the assumptions in theorem \ref{lb93}, and that $U$ is completely energy-bounded. Then  $V^c$ is completely energy-bounded, and the unitary intertwining operators of $V^c$  satisfy strong braiding and (in particular) the strong intertwining property.
\end{co}

\begin{proof}
From theorem \ref{lb93} we know that $V^c$ satisfies condition \ref{cd2}.  By \cite{Gui20a}, any irreducible intertwining operator $\mc Y$ of $V$ is a sum of those that are compressions of intertwining operators of $U$. Therefore, as $U$ is completely energy-bounded, so is $V^c$. By corollary \ref{lb80},   strong braiding holds for the unitary intertwining operators $V^c$.
\end{proof}

\subsubsection{Tensor product VOAs}

\begin{thm}\label{lb83}
Suppose that both $V,V'$ satisfy condition \ref{cd1} or \ref{cd2}, then so is $V\otimes V'$. Moreover, if $V$ and $V'$ are completely energy-bounded, then so is $V\otimes V'$, and any unitary intertwining operator of $V\otimes V'$ satisfy strong braiding and (in particular) the strong intertwining property.
\end{thm}

\begin{proof}
Assume that $V,V'$ satisfy condition \ref{cd1}. It is clear that $V\otimes V'$ is CFT-type, unitary, and regular \cite{DLM97}. By \cite{FHL93}, any irreducible $V\otimes V'$-module is of the form $W_i\otimes W_{i'}$ where $W_i\in\Obj(\RepV),W_{i'}\in\Obj(\Rep(V'))$ are irreducible. Thus $V\otimes V'$ is strongly unitary and strongly energy-bounded. Let $E,\mc F^V$ be as in condition \ref{cd1} for $V$, and let $E',\mc F^{V'}$ be those generating sets of $V'$. Set $\wtd E=E\otimes \Omega\cup\Omega\otimes E'$. Then $E$ is a generating set of quasi-primary vectors. Let $\mc F^{V\otimes V'}$ be the set of all $W_i\otimes V'$ and $V\otimes W_{i'}$ (where $W_i\in\mc F^V,W_{i'}\in \mc F^{V'}$). Since $\Rep(V\otimes V')\simeq \Rep(V)\boxtimes \Rep(V')$ (cf. \cite{ADL05} theorem 2.10), $\mc F^{V\otimes V'}$ is a generating set of irreducible unitary $V\otimes V'$-modules. Any irreducible unitary intertwining operator $\wtd{\mc Y}$ of $V\otimes V'$ with charge space $W_i\otimes V'$ is of the form $\mc Y\otimes Y'$ where $\mc Y$ is an irreducible unitary intertwining operator of $V$, and $Y'$ is the vertex operator of $V'$ on suitable modules. As the homogeneous vector $w^{(i)}_0\in W_i$ in condition \ref{cd1} is chosen such that $\mc Y(w^{(i)}_0,z)$ satisfies the $E$-strong intertwining property, it is clear that $\wtd {\mc Y}(w^{(i)}_0\otimes\Omega,z)$, which equals $\mc Y(w^{(i)}_0,z)\otimes\id$, satisfies the $\wtd E$-strong intertwining property. Similarly, if $\wtd{\mc Y}$ has charge space $V\otimes W_{i'}$, then $\wtd{\mc Y}(\Omega\otimes w^{(i')}_0,z)$ also satisfies the $\wtd E$-strong intertwining property. Thus $V\otimes V'$ satisfies condition \ref{cd1}, and if both $V$ and $V'$ satisfy condition \ref{cd2}, so is $V\otimes V'$.

Now assume that $V$ and $V'$ are completely energy-bounded. By the arguments in \cite{CKLW18} section 6 or \cite{Gui19a} proposition 3.5, tensor products of energy-bounded unitary intertwining operators are energy-bounded. Thus $V\otimes V'$ is completely energy-bounded. Therefore, by corollary \ref{lb80}, the unitary intertwining operators of $V\otimes V'$ satisfy strong braiding since they are energy-bounded.
\end{proof}

\subsection{Examples}\label{lb74}

In \cite{Gui21a}, we  proved the equivalence of braided $C^*$-tensor categories and the strong braiding for the following examples:  $c<1$ unitary Virasoro VOAs, Heisenberg VOAs\footnote{If $V$ is a  Heisenberg VOA then $V$ is not regular. Nevertheless, if we take $\RepuV$ to be the category of semisimple unitary $V$-modules, then the results in the last section still hold for $V$.}, even lattice VOAs (except showing the essential surjectivity), and unitary affine VOAs of type $A,C,G_2$. The goal of this subsection is  to  prove theorem \ref{lb94} in the introduction, which greatly expands above list. We will see that all the examples in theorem \ref{lb94} (including their tensor products and regular cosets) satisfy condition \ref{cd2}, and some of them are  completely energy-bounded.  We first give criteria on the complete rationality of $\mc A_V$ and the essential surjectivity of $\fk F$.

\subsubsection{Complete rationality and numbers of irreducibles}

Recall that if $U$ is a unitary VOA, a conformal unitary sub-VOA $V$ has the same conformal vector as that of $U$. We let $\mc A_U,\mc A_V$ be the conformal nets of $U,V$ respectively.

\begin{thm}[Compare \cite{Ten24} Cor.4.24]\label{lb84}
	Let $U$ be a unitary VOA and $V$ a unitary conformal sub-VOA of $U$. Suppose that $U$ (and hence $V$) is strongly local, and   $V$ satisfies condition \ref{cd1}. Then $\mc A_U$ is completely rational if and only if $\mc A_V$ is.
\end{thm}

\begin{proof}
	By corollary \ref{lb71}, $\RepuV$ is a braided $C^*$-tensor subcategory of $\Rep(\mc A_V)$. Let $\mc H_U$ be the vacuum $\mc A_U$-module, which can also be regarded as an $\mc A_V$-module. By the rigidity of $\RepuV$, $U$ is a dualizable $V$-module. Thus $\mc H_U$ is a dualizable $\mc A_V$-module, which is equivalent to that $[\mc A_U:\mc A_V]<+\infty$. The claim of our theorem now follows from \cite{Lon03} theorem 24.
\end{proof}

Recall that $\Repf(\mc A)$ for a conformal net $\mc A$ is the braided $C^*$-tensor category of dualizable representations of $\mc A$.

\begin{thm}\label{lb85}
	Let $U$ be a unitary VOA and $V$ a unitary conformal sub-VOA of $U$. Suppose that $U$ (and hence $V$) is strongly local,  both $U$ and $V$ satisfy condition \ref{cd1}, and one of $\mc A_U$ and $\mc A_V$ is completely rational. Then $\RepuV$ and $\Repf(\mc A_V)$ have the same number of irreducibles if and only if $\Repu(U)$ and $\Repf(\mc A_U)$ do.
\end{thm}

\begin{proof}
	Note that both $\Repu(V)$ and $\Repu(U)$ are unitary modular tensor categories by \cite{Hua08b} and theorem \ref{lb51}. So are $\Repf(\mc A_V),\Repf(\mc A_U)$ by \cite{KLM01} and the complete rationality of $\mc A_V$ and $\mc A_U$. Write $\mc C_V=\Repu(V),\mc D_V=\Repf(\mc A_V),\mc C_U=\Repu(U),\mc D_U=\Repf(\mc A_U)$. Then by corollary \ref{lb71}, $\mc C_V$ and $\mc C_U$ are full braided $C^*$-tensor subcategories of $\mc D_V,\mc D_U$ respectively. Recall that for a unitary modular tensor category $\mc C$ one can define its global dimension $D(\mc C)>0$ whose square is the square sum of the (positive) quantum dimensions of all irreducibles of $\mc C$ (cf. \cite[(3.1.22)]{BK01}). It then follows that $D(\mc C_V)\leq D(\mc D_V)$, and that $D(\mc C_V)=D(\mc D_V)$ if and only if $\mc C_V$ and $\mc D_V$ have the same number of irreducibles. The same can be said about $\mc C_U$ and $\mc D_U$. We shall prove that $D(\mc C_V)=D(\mc D_V)$ if and only if $D(\mc C_U)=D(\mc D_U)$.
	
	By \cite{KLM01}, $D(\mc D_V)^2$ and $D(\mc D_U)^2$ are respectively the $\mu$-indexes of $\mc A_V$ and $\mc A_U$. By \cite{Lon03} Lemma 22, we have
	\begin{align*}
	D(\mc D_V)/D(\mc D_U)=[\mc A_U:\mc A_V]
	\end{align*} 
	where the right hand side is the index of the subnet $\mc A_V\subset\mc A_U$. Let $\mc H_U$ be the vacuum $\mc A_U$-module, which can also be regarded as an $\mc A_V$-module. Then $\dim_{\mc D_V}(\mc H_U)$, the quantum dimension  of $\mc H_U$  in $\mc D_V$, equals $[\mc A_U:\mc A_V]$. By Cor. \ref{lb71}, we have $\dim_{\mc D_V}(\mc H_U)=\dim_{\mc C_V}(U)$. Thus
	\begin{align*}
	D(\mc D_V)/D(\mc D_U)=\dim_{\mc C_V}(U).
	\end{align*}
	
By \cite{HKL15}, $V\subset U$ is described by a commutative algebra $A$ in $\mc C_V$. If we let $\Rep^0(A)$ be the ribbon category  of local $A$ modules (cf. \cite{KO02}), then $\Rep^0(A)\simeq \mc C_U$ by \cite{CKR17}. Thus, $D(\Rep^0(A))=D(\mc C_U)$. Moreover, if we also regard $U$ as an $A$-module, then $\dim_{\mc C_V}(U)=\dim_{\Rep^0(A)}(U)$. By \cite{KO02} theorem 4.5,
\begin{align*}
D(\mc C_V)/D(\Rep^0(A))=\dim_{\Rep^0(A)}(U),
\end{align*}
which implies
\begin{align*}
D(\mc C_V)/D(\mc C_U)=\dim_{\mc C_V}(U).
\end{align*}
Thus $D(\mc D_V)/D(\mc D_U)=D(\mc C_V)/D(\mc C_U)$, which finishes the proof.
\end{proof}

Next, we prove theorem \ref{lb94} for the four classes of VOAs mentioned in that theorem. For some examples, the strong braiding of all intertwining operators is also discussed.

\subsubsection{WZW models and related coset VOAs}

Let $\gk$ be a finite dimensional (unitary) complex simple Lie algebra with dual Coxeter number $h^\vee$. Let  $l\in\mbb Z_+$, and set $l'=\frac{l+h^\vee}{l+h^\vee+1}-h^\vee$. Denote by $L_l(\gk)$ the level $l$ unitary affine VOA and $\mc W_{l'}(\gk)$ the discrete series $W$-algebra (see \cite{ACL19} for the definition), both of which are simple.   $L_l(\gk)$ is naturally a unitary VOA, and any $L_l(\gk)$-module is unitarizable, i.e., $L_l(\gk)$ is strongly unitary. (Cf. \cite{Kac90,DL14,CKLW18,Gui19c}.) Moreover, $L_l(\gk)$ and $\mc W_{l'}(\gk)$ are regular \cite{DLM97,Ara15a,Ara15b}.

\begin{thm}\label{lb81}
Let $V$ be either $L_l(\gk)$  of any type, or $\mc W_{l'}(\gk)$ of type $ADE$. Then $V$ satisfies condition \ref{cd2} and (a) (b) (c) of theorem \ref{lb94}.
\end{thm}

\begin{thm}[Strong braiding]\label{lb82}
If $\gk$ is of type $ADE$, then $L_l(\gk)$ and $\mc W_{l'}(\gk)$ are completely energy-bounded, and the unitary intertwining operators of $L_l(\gk)$ and $\mc W_{l'}(\gk)$ satisfy  strong braiding and (in particular) the strong intertwining property. 
\end{thm}

We prove these two theorems simultaneously.

\begin{proof}
We first discuss affine VOAs. Let $V=L_l(\gk)$ and $E=V(1)$. Then $E$ generates $V$, and any vector in $E$ is quasi-primary. Moreover, the argument in \cite{BS90} shows that for any $u\in E$ and $W_i\in\Obj(\RepuV)$, $Y_i(u,z)$ satisfies linear energy bounds. Thus, by the proof of \cite{CKLW18} proposition 6.1, $V$ is strongly energy-bounded.

Suppose that $\gk$ is not of type $E$. By the results in \cite{Was98,TL04,Gui19c,Gui20b},  one can find $\mc F^V$ such that $V$  satisfies  condition \ref{cd2}, except possibly that $\mc Y(w^{(i)}_0,z)$ satisfies the $E$-strong intertwining property.  (For type $ABCDG$, these results are summarized in \cite{Gui19c} theorems 3.1, 3.3, 4.2, 4.4, 5.7.) However, since $Y_k(u,z)$ is linearly energy-bounded for each unitary $W_k$ and $u\in E$, by \cite{Gui19a} proposition 3.16 which uses a result in \cite{TL99}, $\mc Y(w^{(i)}_0,z)$ satisfies the $E$-strong intertwining property. Thus condition \ref{cd2} is satisfied. By theorem \ref{lb51} and corollaries \ref{lb71} and \ref{lb65}, we have (a) and (b)  except the essential surjectivity of $\fk F$. That $\fk F$ is essentially surjective follows from \cite{Hen19} section 3.2. Moreover, \cite{Hen19} shows that any irreducible representation of $\mc A_V$ arises from a unitary $V$-module which is dualizable. Thus $\Rep(\mc A_V)$ has only finitely many irreducibles, and all of them are dualizable. Therefore, by \cite{LX04} theorem 4.9 and the split property of conformal nets (see \cite{MTW18}), $\mc A_V$ is completely rational. This proves the first theorem for $L_l(\gk)$ not of type $E$.

Now suppose that $\gk$ is of type $ADE$. We prove the two theorems for $L_l(\gk)$ and $\mc W_{l'}(\gk)$ by induction on $l$. As argued in the former paragraph, the essential surjectivity and the complete rationality of $L_l(\gk)$ will follow from (a), (b), and \cite{LX04,Hen19,MTW18}. The proof of these two results for $\mc W_{l'}(\gk)$ is left to the next paragraph. Here, we prove (a), (b) (except the essential surjectivity of $\fk F$), and the second theorem. When $l=1$, by  Frenkel-Kac construction \cite{FK80}, $V$ is an even lattice VOA. So everything follows from the results in \cite{Gui21a} section 5.3.  Suppose level $l$ has been proved. We now consider the case $l+1$. Let $V=L_{l+1}(\gk),U=L_l(\gk)\otimes L_1(\gk)$, and let $V\subset U$ be the diagonal embedding. Then by \cite{ACL19}, we have $V^c=\mc W_{l'}(\gk)$ (for some number $l'$) and $V^{cc}=V$. By \cite{Ara15a,Ara15b}, $\mc W_{l'}(\gk)$ is regular. By induction and theorem \ref{lb83}, any  unitary intertwining operator of $U$ is energy-bounded and satisfies the strong intertwining property. Thus  $V^c$ and  $V=V^{cc}$ satisfy the assumptions and therefore the consequences of theorem \ref{lb92}. This proves the case of level $l+1$ by corollaries \ref{lb71},  \ref{lb79}.

By theorem \ref{lb83}, $V\otimes V^c$ satisfies condition \ref{cd2}. Since $\mc A_U$ has been proved completely rational, so is $\mc A_{V\otimes V^c}\simeq\mc A_V\otimes \mc A_{V^c}$ by theorem \ref{lb84}. Therefore $\mc A_{V^c}$ is completely rational (\cite{Lon03} lemma 25). This proves (c) for $\mc W_{l'}(\gk)$. By theorem \ref{lb85}, $\Repu(V\otimes V^c)\simeq \Repu(V)\boxtimes\Repu(V^c)$ and  $\Repf(\mc A_{V\otimes V^c})\simeq\Repf(\mc A_V)\boxtimes \Repf(\mc A_{V^c})$ have the same number of irreducibles. So the same is true for $\Repu(V^c)$ and $\Repf(\mc A_{V^c})$. Therefore $\fk F:\Repu(V^c)\rightarrow\Repf(\mc A_{V^c})$ must be essentially surjective.
\end{proof}

\begin{rem}
We give a historical remark on theorem \ref{lb81}. In the celebrated paper  \cite{Was98}, Wassermann proved for type $A$ affine VOAs the strong integrability of modules and hence the existence of the fully faithful $*$-functor $\fk F$ (preserving the linear structures). More importantly, he proved the breakthrough result that $\fk F$ induces an injective homomorphism of Grothendieck rings $\fk F:\Gr(\Repu(V))\rightarrow\Gr(\Repf(\mc A_V))$. In other words, for type $A$ WZW models, the conformal net fusion rules  agree with the VOA fusion rules.  Some key ideas about complete unitarity also appear in that paper. 

Wassermann's results were later generalized by Toledano-Laredo to type $D$ affine VOAs \cite{TL04}, and by Loke to unitary Virasoro VOAs (equivalent to $\mc W_{l'}(\fk{sl}_2)$) \cite{Loke94}. On the other hand, Xu showed in \cite{Xu00a} that when $V$ is a type $A$ affine VOA, $\mc A_V$ is completely rational. Xu's calculation of the principle graphs of  multi-interval subfactors also shows that $\fk F$ is essentially surjective. For a general type $A$ discrete series $W$-algebra $V=\mc W_{l'}(\fk{sl}_n)$, Xu calculated in \cite{Xu00b} the Grothendieck ring $\Gr(\Repf(\mc A_V))$ which is equivalent to $\Gr(\RepuV)$. The complete rationality of $\mc A_V$ was shown in \cite{Xu01,Lon03}. The strong integrability of all $V$-modules (which implies that the equivalence $\Gr(\RepuV)\simeq\Gr(\Repf(\mc A_V))$ is realized by $\fk F$) was proved in \cite{CWX}.

A systematic study of complete unitarity was initiated by the author in \cite{Gui19a,Gui19b}. In these two papers, we gave general criteria which, together with the analysis of energy bounds in \cite{Was98,TL04,Gui19c}, enable us to prove the complete unitarity of all (unitary) affine VOAs except type $EF$. The remaining two types, along with discrete series $W$-algebras of type $AE$, were proved completely unitary by Tener in \cite{Ten24} under the framework of bounded localized vertex/intertwining operators (which is very different from our approach). The complete unitarity of type $D$ $W$-algebras follows from our theorem \ref{lb81}. In an upcoming paper \cite{CCP}, Carpi-Ciamprone-Pinzari will give a categorical proof of the complete unitarity of affine VOAs.

In \cite{Gui21a}, we proved for type $ACG$ affine VOAs that $\fk F$ preserves the braided $C^*$-tensor structures. Henriques' result on the irreducibles of loop group conformal nets \cite{Hen19} implies that $\fk F$ is essentially surjective. Thus, for affine VOAs of those types, the equivalence of unitary modular tensor categories $\RepuV\simeq\Repf(\mc A_V)$ was proved. The complete rationality of those types is a consequence of the equivalence of tensor categories and  \cite{LX04,MTW18} (as in the proof of theorem \ref{lb81}). In the case of affine type $A$, the equivalence also follows from \cite{CCP}. A complete and uniform proof of the complete rationality of all WZW models first appeared in \cite{Ten24}. The methods in that paper also imply the complete rationality of the conformal nets associated to parafermion VOAs (mentioned below) and type $ADE$ discrete series $W$-algebras. Theorem \ref{lb81} gives the first complete proof of the equivalence $\RepuV\simeq\Repf(\mc A_V)$ for  affine VOAs of all types,  discrete series $W$-algebras of type $ADE$, and more.
\end{rem}

Let $L_l(\gk)$ be any unitary affine VOA. Let $\fk h$ be the Cartan subalgebra of $\gk$. Then the (unitary) Heisenberg VOA $L(\hk)$ is a unitary sub-VOA of $L_l(\gk)$. The commutant of $L(\hk)$ in $L_l(\gk)$, which we denote by $K_l(\gk)$, is called a parafermion VOA.  $K_l(\gk)$ is regular \cite{DW11a,ALY14,DR17}, and is also the commutant in $L_l(\gk)$ of $V_{\sqrt l\Lambda_\gk}$ where $\Lambda_\gk$ is the even lattice generated by the long roots of $\gk$ \cite{DW11b}. We have seen that $L_l(\gk)$ satisfies condition \ref{cd2}. By \cite{Gui21a}, $V_{\sqrt l\Lambda_\gk}$ also satisfies condition \ref{cd2}. So does $K_l(\gk)$ by theorem \ref{lb93}. So (a) (b) (c) of theorem \ref{lb94} hold for $K_l(\gk)$. In particular, the argument in the last paragraph  of the proof of theorem \ref{lb81} (and theorem \ref{lb82}) proves the complete rationality and the essential surjectivity for $K_l(\gk)$. In the cases that $\gk$ is of type $ADE$, the complete energy bounds and the strong braiding hold by corollary \ref{lb79}. This proves the following theorem.

\begin{thm}
$K_l(\gk)$ satisfies condition \ref{cd2} and (a) (b) (c) of theorem \ref{lb94}. Moreover, if $\gk$ is of type $ADE$, then $K_l(\gk)$ is completely energy-bounded, and the unitary intertwining operators of $V_\Lambda$ satisfy strong braiding and (in particular) the strong intertwining property. 
\end{thm}

\subsubsection{Lattice VOAs}

Let $\Lambda\subset \mbb R^n$ be a rank $n$ even lattice. This means that $\Lambda\simeq\mbb Z^n$ as abelian groups, and that $\lVert\alpha\lVert^2:=(\alpha|\alpha)$ is an even number for any $\alpha\in\Lambda$, where  $(\cdot|\cdot)$ is the inner product inherited from that of $\mbb R^n$. Then the lattice VOA $V_\Lambda$ is simple, unitary \cite{DL14}, regular \cite{DLM97}, and satisfy condition \ref{cd2} \cite{Gui21a}. Moreover, it was shown in \cite{Gui21a} that $V_\Lambda$ is completely energy-bounded, and the unitary intertwining operators satisfy strong braiding. Thus, it remains to prove the complete rationality of $\mc A_{V_\Lambda}$ and the essential surjectivity of $\fk F$.

\begin{lm}\label{lb97}
Let $n=\rank(\Lambda)$. Then for any even sublattice $L\subset\Lambda$ of rank $k<n$, there is a rank $n-k$ sublattice $L'\subset\Lambda$ orthogonal to $L$.
\end{lm}

\begin{proof}
Let $e_1,\dots,e_n$ be a basis of $\Lambda$, and $f_1,\dots,f_k$ a basis of $L$. Then each  $(e_i|f_j)$ is an integer. Consider all $(k_1,\dots,k_n)\in\mbb Q^n$ satisfying $\sum_{i=1}^nk_i(e_i|f_j)=0$ for any $j$. By linear algebra, the ($\mbb Q$-linear) solution space of this system of linear equations for $(k_1,\dots,k_n)$ is spanned by at least $n-k$ linearly independent vectors. Thus it contains $n-k$ linearly independent vectors with integral components. Let $L'$ be the sublattice generated by the vectors of the form $\sum_{i=1}^nk_ie_i$ where $(k_1,\dots,k_n)\in\mbb Z^n$ is one of those $n-k$  vectors. Then $L'$ has rank $n-k$ and is orthogonal to $L$. 
\end{proof}

We say that a rank $n$ even lattice is \textbf{orthogonal} if it is generated by $n$ mutually orthogonal vectors.

\begin{pp}\label{lb98}
Let $n=\rank(\Lambda)$. Then $\Lambda$ has a rank $n$ orthogonal sublattice.
\end{pp}

\begin{proof}
By the above lemma, each non-zero vector in $\Lambda$ is orthogonal to a rank $n-1$ sublattice. Thus the proposition follows easily by induction on $n$.
\end{proof}

\begin{thm}\label{lb96}
Let $\Lambda$ be any even lattice. Then $V_\Lambda$ satisfies condition \ref{cd2} and (a) (b) (c) of theorem \ref{lb94}. Moreover, $V_\Lambda$ is completely energy-bounded, and the unitary intertwining operators of $V_\Lambda$ satisfy strong braiding and (in particular) the strong intertwining property.
\end{thm}

\begin{proof}
Let $V=V_\Lambda$. We shall only show that $\mc A_V$ is completely rational and $\fk F:\Repu(V)\rightarrow\Repf(\mc A_V)$ is essentially surjective, as the other results are already known. Let $n=\rank(\Lambda)$. If $n=1$, as explained in \cite{CKLW18} example 8.8,  the classification results in \cite{BMT88} show that $\mc A_V$ is isomorphic to the corresponding  lattice conformal net $\mc B_\Lambda$ constructed in \cite{DX06}. Moreover, \cite{DX06} classified all the irreducibles  and proved the complete rationality of any such net.  Thus, it is easy to see that $\fk F$ is essentially surjective by comparing the numbers of irreducibles of both categories. This proves the rank $1$ case. A different proof is given in lemma \ref{lb95}.

Now we treat the general case. By the above proposition, $\Lambda$ has an orthogonal sublattice $\wtd \Lambda$. Thus $V_{\wtd\Lambda}$ is a conformal unitary subalgebra of $V=V_\Lambda$. (That this subalgebra is conformal follows from the fact that the central charge of a lattice VOA equals the rank of the lattice.) Since $\wtd\Lambda$ is orthogonal, it is equivalent to $L_1\times L_2\times\cdots\times L_n$ where $L_1,L_2,\dots,L_n$ are rank $1$ even lattices. Since the complete rationality and the essential surjectivity hold for $V_{L_1},V_{L_2},\dots,V_{L_n}$, they hold for their tensor product $V_{L_1}\otimes V_{L_2}\otimes\cdots\otimes V_{L_n}$ which is equivalent to $V_{\wtd \Lambda}$. Thus, by theorems \ref{lb84} and \ref{lb85}, they also hold for $V$.
\end{proof}

\begin{lm}\label{lb95}
Theorem \ref{lb96} holds when $\rank(\Lambda)=1$.
\end{lm}

\begin{proof}
We have $\Lambda\simeq \sqrt {2k}\cdot \mbb Z$ for some $k\in\mbb Z_+$. We prove the lemma by induction on $k$. If $k=1$, then, by Frenkel-Kac construction \cite{FK80}, $V_\Lambda\simeq V^1_{\fk{sl}_2}$. Thus the lemma follows from theorem \ref{lb81}. Now assume that the lemma has been proved for $\sqrt {2k}\cdot \mbb Z$. We shall prove this for $\Lambda=\sqrt {2k+2}\cdot\mbb Z$. Let $\Gamma=\sqrt {2k}\cdot\mbb Z\times \sqrt 2\cdot \mbb Z$. Then one can regard $\Lambda$ as a sublattice of $\Gamma$ generated by the vector $(\sqrt{2k},\sqrt 2)$. By lemma \ref{lb97}, $\Gamma$ has a rank $1$ sublattice $L$ orthogonal to $\Lambda$. Thus  $V_\Lambda\otimes V_L$, which is equivalent to $V_{\Lambda\times L}$, is a conformal unitary sub-VOA of $V_\Gamma$. Since $V_\Gamma\simeq V_{\sqrt{2k}\cdot\mbb Z}\otimes V_{\sqrt 2\cdot\mbb Z}$, the complete rationality and the essential surjectivity hold for $V_\Gamma$. Thus they also hold for $V_\Lambda\otimes V_L$ and hence for $V_\Lambda$ by theorems \ref{lb84} and \ref{lb85}.
\end{proof}

The method of proving theorem \ref{lb96} allows us to prove the following result conjectured in \cite{CKLW18} section 8.

\begin{thm}
For any even lattice $\Lambda$, the conformal net $\mc A_{V_\Lambda}$ constructed via smeared vertex operators of $V_\Lambda$ is isomorphic to the Dong-Xu lattice conformal  net $\mc B_\Lambda$ constructed in \cite{DX06}.
\end{thm}

\begin{proof}
Write $\mc A_{V_\Lambda}$ as $\mc A_\Lambda$ for simplicity. By \cite{BMT88}, when $\rank(\Lambda)=1$ we have an isomorphism $\mc A_\Lambda\simeq \mc B_\Lambda$. Thus this is also true when $\Lambda$ is a product of rank $1$ even lattices, i.e., $\Lambda$ is orthogonal. In general, by proposition \ref{lb98}, $\Lambda$ has an orthogonal sublattice $L$ with the same rank. Let us identify $\mc A_L$ with $\mc B_L$. Then both $\mc A_\Lambda$ and $\mc B_\Lambda$ are finite-index irreducible extensions of $\mc A_L$, which are in turn determined by  commutative irreducible Q-systems $Q_1,Q_2$ respectively. $Q_1$ and $Q_2$ are equivalent as objects in $\Repf(\mc A_L)$. Indeed, the irreducibles in $\Repf(\mc A_L)$ are simple currents, and their fusion rules are equivalent to the abelian group $L^\circ/L$ where $L^\circ=\{\beta\in\mbb RL:(\beta|L)\subset\mbb Z \}$ is the dual lattice of $L$ (cf. \cite{DX06}). Then $Q_1,Q_2$ as objects are both equivalent to the subgroup $\Lambda/L$ of $L^\circ/L$. Thus,  by \cite{KL06} remark 4.4, $Q_1,Q_2$ are unitarily equivalent as Q-systems. This proves $\mc A_\Lambda\simeq\mc B_\Lambda$.
\end{proof}

We remark that the equivalence $\Repu(V_\Lambda)\simeq \Repf(\mc B_\Lambda)$ was proved in \cite{Bis18} section 3.3. This result, together with the above theorem, gives another proof of the equivalence $\Repu(V_\Lambda)\simeq \Repf(\mc A_{V_\Lambda})$.

By the results proved so far, the claims for $U$ and $V$ in theorem \ref{lb94} hold. Similar to the proof for $W_{l'}(\gk)$ and $K_l(\gk)$, we may use corollary \ref{lb71} and theorems \ref{lb93}, \ref{lb84}, \ref{lb85} to show that the commutant $V^c$ also satisfies condition \ref{cd2} and (hence) (a) (b) (c). The proof of theorem \ref{lb94} is now complete.

\appendix

\section{Strongly commuting closed operators}\label{lb9}

In this section, we prove some useful properties for strongly commuting closed operators. We fix a Hilbert space $\mc H$. Recall that if $A$ is a closed operator (with dense domain) on $\mc H$ with polar decomposition $A=UH$ ($H$ is positive and $U$ is a partial isometry)  then we have left and right bounding projections $\{p_s\},\{q_s\}$ defined by  $q_s=\chi_{[0,s]}(H)$ and $p_s=Uq_sU^*$. Note that $\ovl{p_sA}=Aq_s$ is bounded, and $\ovl{p_sA^*A}=p_s A^*Ap_s$ is bounded and positive. The von Neumann algebra generated by $A$ is the one generated by $U$ and all $\{p_s\}$ (equivalently, all $\{q_s\}$). Two closable operators $A,B$ are said to \textbf{commute strongly} if the von Neumann algebras generated by $\ovl{A}$ and by $\ovl{B}$ commute. If $A$ and $B$ are bounded, then they commute strongly if and only if $[A,B]=[A^*,B]=0$, namely, $A$ commutes adjointly with $B$.

In the case that one of the two strongly commuting closed operators is bounded, we have the following well known properties:

\begin{pp}\label{lb15}
	Let $A$ be closed and $x$ be bounded on $\mc H$. The following are equivalent:
	
	(a) $A$ commutes strongly with $x$.
	
	(b) $xA\subset Ax$ and $x^*A\subset Ax^*$.
	
	(c) There is a core $\Dom$ for $A$ such that $x\cdot A|_\Dom\subset A\cdot x$ and $x^*\cdot A|_\Dom\subset A\cdot x^*$. 
\end{pp}

\begin{proof}
	That (b)$\Longleftrightarrow$(c) is a routine check.  (a) is clearly equivalent to (b) when $x$ is unitary. The general case follows from linearity. See for example \cite{Gui19a} section B.1 for more details.
\end{proof}

\begin{pp}\label{lb16}
	Let $A$ be closed and $x$ be bounded on $\mc H$.
	
	(a) If $\Dom(Ax)$ is a dense subspace of $\mc H$, then $Ax$ is closed.
	
	(b) If $A$ commutes strongly with $x$, then $\Dom(Ax)$ is dense.
\end{pp}

\begin{proof}
	Choose any $\xi\in\mc H$ and a sequence of vectors $\xi_n\in\Dom(Ax)$ such that  $\xi_n\rightarrow\xi$ and $Ax\xi_n\rightarrow\eta$. Then $x\xi_n\in\Dom(A)$ and $x\xi_n\rightarrow x\xi$. Since $A$ is closed, we must have $x\xi\in\Dom(A)$ and $Ax\xi=\eta$. This proves (a). (b) is true since $Ax\supset xA$ and $\Dom(xA)=\Dom(A)$ is dense.
\end{proof}

We prove the following result as an interesting application of the above propositions. The special case $H=K$ is well-known, but we will use this result in the case that  $H=\log\Delta$ and $K=\Delta^{\pm\frac 12}$ (where $\Delta$ is the modular operator). 

\begin{pp}\label{lb38}
	Let $H,K$ be  self-adjoint closed operators on $\mc H$, and assume that $K$ is  affiliated with the (abelian) von Neumann algebra generated by $H$. Suppose that $\scr D_0$ is a dense subspace of $\mc H$, that $\scr D_0\subset\Dom(K)$, and that $e^{\im tH}\scr D_0\subset\scr D_0$ for any $t\in\mbb R$. Then $\scr D_0$ is a core for $K$.
\end{pp}

\begin{proof}
	Let $A_0=K|_{\scr D_0}$ and $A=\ovl{A_0}$. Then $A\subset K$. We shall show $\Dom(K)\subset\Dom(A)$. Since $K$ commutes strongly with $H$, it also commutes strongly with each $e^{\im tH}$. Since  $e^{\im tH}\scr D_0\subset\scr D_0$ and $e^{\im tH}K\subset Ke^{\im tH}$, we actually have $e^{\im tH}A_0\subset A_0e^{\im tH}$. Therefore $A$ commutes strongly with each $e^{\im tH}$, and hence with $K$.
	
	Let $\{p_t\}$ be the left (and also right) bounding projections of $K$. Then they commute strongly with $A$. Thus, by proposition \ref{lb16}, $Ap_t$ has dense domain and is a closed operator. Since $Ap_t\subset Kp_t$ and $Kp_t$ is continuous, $Ap_t$ is  continuous and closed, i.e. $Ap_t$ is an (everywhere defined) bounded operator. Therefore $\Dom(Ap_t)=\mc H$; equivalently, $p_t\mc H\subset\Dom(A)$. Choose any $\xi\in\Dom(K)$. Then $p_t\xi\in\Dom(A)$, and as $t\rightarrow+\infty$, we have $p_t\xi\rightarrow\xi$ and  $Ap_t\xi=Kp_t\xi=p_tK\xi\rightarrow K\xi$. Thus $\xi\in\Dom(A)$ and $A\xi=K\xi$.
\end{proof}

We now study the case where neither of the two strongly commuting operators are assumed to be bounded.

\begin{lm}\label{lb4}
	Let $A$ and $B$ be strongly commuting closed operators on $\mc H$. Let $\{p_s\},\{q_s\}$ (resp. $\{e_t\},\{f_t\}$) be the left and the right bounding projections of $A$ (resp. $B$). Then for any $\xi\in\Dom(A^*A)\cap\Dom(B^*B)$, the following limits exist and are equal:
	\begin{align}
	\lim_{s,t\rightarrow+\infty}Aq_sBf_t\xi=\lim_{s,t\rightarrow+\infty}Bf_tAq_s\xi\label{eq3}
	\end{align}
\end{lm}

\begin{proof}
	Since the bounded operators $Aq_s$ and $Bf_t$ are in the von Neumann algebras generated by $A$ and by $B$ respectively, we have $[Aq_s,Bf_t]=[(Aq_s)^*,Bf_t]=0$. 	For any $s_1,s_2,t_1,t_2>0$, we set $s=\min\{s_1,s_2\},t=\min\{t_1,t_2\}$, and compute
	\begin{align*}
	&\bk{Aq_{s_1}Bf_{t_1}\xi|Aq_{s_2}Bf_{t_2}\xi}=\bk{q_{s_2}A^*Aq_{s_1}Bf_{t_1}\xi|Bf_{t_2}\xi}=\bk{Bf_{t_1}q_{s_2}A^*Aq_{s_1}\xi|Bf_{t_2}\xi}\\
	=&\bk{q_{s_2}A^*Aq_{s_1}\xi|f_{t_1}B^*Bf_{t_2}\xi}=\bk{q_sA^*A\xi|f_tB^*B\xi}
	\end{align*}
	which converges to $\bk{A^*A\xi|B^*B\xi}$ when $\min\{s_1,s_2,t_1,t_2\}\rightarrow+\infty$. Therefore $\lVert Aq_{s_1}Bf_{t_1}\xi- Aq_{s_2}Bf_{t_2}\xi\lVert^2$ converges to $0$ as $\min\{s_1,s_2,t_1,t_2\}\rightarrow+\infty$. This proves the existence of limits.
\end{proof}

As shown below, strong commutativity implies weak commutativity.

\begin{pp}\label{lb3}
	Let $A$ and $B$ be strongly commuting closed operators on $\mc H$. Let $\xi\in\mc H$.
	
	(a) If $\xi\in\Dom(AB)\cap\Dom(A)$, then $\xi\in\Dom(BA)$, and $AB\xi=BA\xi$.
	
	(b) If $\xi\in\Dom(A^*A)\cap\Dom(B^*B)$, then $\xi\in\Dom(AB)\cap\Dom(BA)$, and $AB\xi=BA\xi$ are equal to the limit \eqref{eq3}.
\end{pp}

\begin{proof}
	Let $\{p_s\},\{q_s\},\{e_t\},\{f_t\}$ be as in lemma \ref{lb4}. 
	
	(a) Choose any $\xi\in\Dom(AB)\cap\Dom(A)$. Since $Aq_s$ is bounded and commutes strongly with $B$, by proposition \ref{lb15} we have $Aq_sB\subset BAq_s$. In particular, $Aq_s\Dom(B)\subset\Dom(B)$. Therefore $Aq_s\xi\in\Dom(B)$. As $s\rightarrow+\infty$, we have $Aq_s\xi=p_sA\xi\rightarrow A\xi$, and $BAq_s\xi=Aq_sB\xi=p_sAB\xi\rightarrow AB\xi$. Thus, by the closedness of $B$, we have $A\xi\in\Dom(B)$ and $BA\xi=AB\xi$.
	
	(b) Assume $\xi\in\Dom(A^*A)\cap\Dom(B^*B)$. Let $\eta$ be the limit \eqref{eq3}. Choose any $\varepsilon>0$. Then there exists $N>0$ such that whenever $t,s>N$, we have  $\lVert \eta-Aq_sBf_t\xi\lVert<\varepsilon$. Choose any $s>N$. Since $Aq_s$ is bounded and $\lim_{t\rightarrow+\infty}Bf_t\xi=\lim_{t\rightarrow+\infty}e_tB\xi=B\xi$, there exists $t>N$ such that $\lVert Aq_sBf_t\xi-Aq_sB\xi\lVert<\varepsilon$. Therefore $\lVert \eta-Aq_sB\xi\lVert<2\varepsilon$. Since this is true for arbitrary $s>N$, we have proved that $\lim_{s\rightarrow+\infty}Aq_sB\xi=\eta$. Clearly $q_sB\xi\in\Dom(A)$. Since $\lim_{s\rightarrow+\infty}q_sB\xi=B\xi$, we conclude that $B\xi\in\Dom(A)$,  and $\eta=AB\xi$. Similar arguments show $\xi\in\Dom(BA)$ and $\eta=BA\xi$.
\end{proof}

\begin{lm}\label{lb14}
	Let $A_1,A_2,B_1,B_2$ be closed operators on $\mc H$, and let $\xi_1,\xi_2\in\mc H$. Assume that for each $i,j\in\{1,2\}$, $A_i$ commutes strongly with $B_j$, and  $\xi_i$ is inside the domains of $A_i^*A_i,B_i^*B_i$. Then  $\xi_i\in\Dom(A_iB_i)$ and  $A_iB_i\xi_i=B_iA_i\xi_i$ for any $i\in\{1,2\}$, and
	\begin{align}
	\bk{A_1B_1\xi_1|A_2B_2\xi_2}=\left\{ \begin{array}{ll}
	\bk{A_2^*A_1\xi_1|B_1^*B_2\xi_2}\ & \text{if }\xi_1\in\Dom(A_2^*A_1),\xi_2\in\Dom(B_1^*B_2)\\
	~&~\\
	\bk{B_2^*B_1\xi_1|A_1^*A_2\xi_2} & \text{if }\xi_1\in\Dom(B_2^*B_1),\xi_2\in\Dom(A_1^*A_2)
	\end{array}\right..\label{eq1}
	\end{align}
\end{lm}

\begin{proof}
	Assume that for each $i=1,2$, $\xi_i$ is inside the domains of $A_i^*A_i,B_i^*B_i$. Assume also that $\xi_1\in\Dom(A_2^*A_1),\xi_2\in\Dom(B_1^*B_2)$. Then $\xi_i\in\Dom(A_iB_i)$ and $A_iB_i\xi_i=B_iA_i\xi_i$ by proposition \ref{lb3}. Let $\{p^1_t\},\{p^2_t\},\{e^1_t\},\{e^2_t\}$ be respectively the left bounding projections of $A_1,A_2,B_1,B_2$. Similarly, let $\{q^1_t\},\{q^2_t\},\{f^1_t\},\{f^2_t\}$ be respectively the right bounding projections of $A_1,A_2,B_1,B_2$. For any $s_1,s_2,t_1,t_2>0$, we set
	\begin{gather*}
	c_1=\bk{A_1B_1\xi_1|A_2B_2\xi_2},\qquad c_2=\bk{A_2^*A_1\xi_1|B_1^*B_2\xi_2},\\
	c_3=\bk{A_1q^1_{s_1}B_1f^1_{t_1}\xi_1|A_2q^2_{s_2}B_2f^2_{t_2}\xi_2},\qquad c_4=\bk{q^2_{s_2}A_2^*A_1\xi_1|f^1_{t_1}B_1^*B_2\xi_2}.
	\end{gather*}
	Choose any $\varepsilon>0$. Then by proposition \ref{lb3}-(b), there exists $N>0$ such that whenever $s_1,s_2,t_1,t_2>N$, we have $|c_1-c_3|<\varepsilon$. Using the strong commutativity (equivalently, adjoint commutativity for bounded operators) of $A_iq^i_{s_i}$ and $B_jf^j_{t_j}$, one has
	\begin{align*}
	c_3=\bk{\ovl{q^2_{s_2}A_2^*}A_1q^1_{s_1}\xi_1|\ovl{f^1_{t_1}B_1^*}B_2f^2_{t_2}\xi_2}=\bk{\ovl{q^2_{s_2}A_2^*}p^1_{s_1}A_1\xi_1|\ovl{f^1_{t_1}B_1^*}e^2_{t_2}B_2\xi_2}.
	\end{align*}
	Choose $s_2,t_1>N$ such that $|c_4-c_2|<\varepsilon$. Since $\ovl{q^2_{s_2}A^*}$ and $\ovl{f^1_{t_1}B_1^*}$ are bounded, we can find $s_1,t_2>N$ such that $|c_3-c_4|<\varepsilon$. Therefore $|c_1-c_2|<3\varepsilon$. Since $\varepsilon$ is arbitrary, we conclude $c_1=c_2$. This proves the first of $\eqref{eq1}$. The second one is proved similarly.
\end{proof}

The following was proved in \cite{Gui21a} lemma 4.17.

\begin{lm}\label{lb59}
	Let $P(z_1,\cdots,z_m)$ and $Q(\zeta_1,\cdots,\zeta_n)$ be polynomials of $z_1,\dots,z_m$ and $\zeta_1,\dots,\zeta_n$ respectively. Let $D$ be a self-adjoint positive operator on $\mc H$, and set $\mc H^\infty=\bigcap_{n\in\mathbb Z_{\geq0}}\Dom(D^n)$. Choose  closable operators $A_1,\dots,A_m$ and $B_1,\dots,B_n$   on 
	$\mc H$ with common invariant (dense) domain $\mc H^\infty$. Assume that there exists $\varepsilon>0$ such that   $e^{\im tD}A_re^{-\im tD}$ commutes strongly with $B_s$ for any $r=1,\dots,m,s=1,\dots,n$, and  $t\in(-\varepsilon,\varepsilon)$.  Assume also that the unbounded operators $A=P(A_1,\cdots.A_m),B=Q(B_1,\cdots,B_n)$ (with common domain $\mc H^\infty$) are closable. Then $A$ commutes strongly with $B$.
\end{lm}


\noindent {\small \sc Yau Mathematical Sciences Center, Tsinghua University, Beijing, China.}

\noindent {\textit{E-mail}}: binguimath@gmail.com\qquad bingui@tsinghua.edu.cn

\newpage

\begin{thebibliography}{999999}
\footnotesize	

\bibitem[ABD04]{ABD04}
Abe, T., Buhl, G. and Dong, C., 2004. Rationality, regularity, and C2-cofiniteness. Transactions of the American Mathematical Society, 356(8), pp.3391-3402.


\bibitem[ACL19]{ACL19}
Arakawa, T., Creutzig, T. and Linshaw, A.R., 2019. W-algebras as coset vertex algebras. Inventiones mathematicae, 218(1), pp.145-195.

\bibitem[ADL05]{ADL05}
Abe, T., Dong, C.,  Li, H. 2005.  Fusion Rules for the Vertex Operator Algebras $M(1)^+$ and $V_L^+$. Comm. Math. Phys, 253, pp.171-219.


\bibitem[ALY14]{ALY14}
Arakawa, T., Lam, C.H. and Yamada, H., 2014. Zhu's algebra, $C_2$-algebra and $C_2$-cofiniteness of parafermion vertex operator algebras. Advances in Mathematics, 264, pp.261-295.

\bibitem[Ara15a]{Ara15a}
Arakawa, T., 2015. Associated varieties of modules over kac–moody algebras and c 2-cofiniteness of w-algebras. International Mathematics Research Notices, 2015(22), pp.11605-11666.

\bibitem[Ara15b]{Ara15b}
Arakawa, T., 2015. Rationality of W-algebras: principal nilpotent cases. Annals of Mathematics, pp.565-604.

\bibitem[BDH14]{BDH14}
Bartels, A., Douglas, C.L. and Henriques, A., 2014. Dualizability and index of subfactors. Quantum Topology, 5(3), pp.289-345.

\bibitem[BMT88]{BMT88}
Buchholz, D., Mack, G. and Todorov, I., 1988. The current algebra on the circle as a germ of local field theories. Nuclear Physics B-Proceedings Supplements, 5(2), pp.20-56.

\bibitem[BK01]{BK01}
Bakalov, B. and Kirillov, A.A., 2001. Lectures on tensor categories and modular functors (Vol. 21). American Mathematical Soc..

\bibitem[BS90]{BS90}
Buchholz, D. and Schulz-Mirbach, H., 1990. Haag duality in conformal quantum field theory. Reviews in Mathematical Physics, 2(01), pp.105-125.

\bibitem[Bar54]{Bar54}
Bargmann, V., 1954. On unitary ray representations of continuous groups. Annals of Mathematics, pp.1-46.

\bibitem[Bis18]{Bis18}
Bischoff, M., 2018. Conformal net realizability of Tambara-Yamagami categories and generalized metaplectic modular categories. arXiv preprint arXiv:1803.04949.

\bibitem[Bor92]{Bor92}
Borcherds, R. E. (1992). Monstrous moonshine and monstrous Lie superalgebras. Inventiones mathematicae, 109(1), 405-444.

\bibitem[CCP]{CCP}
Carpi, S., Ciamprone, S., Pinzari, C., Weak quasi-Hopf algebras, $C^*$-tensor categories and conformal field theory, to appear.

\bibitem[CKLW18]{CKLW18}
Carpi, S., Kawahigashi, Y., Longo, R. and Weiner, M., 2018. From vertex operator algebras to conformal nets and back (Vol. 254, No. 1213). Memoirs of the American Mathematical Society

\bibitem[CKR17]{CKR17}
Creutzig, T., Kanade, S. and McRae, R., 2017. Tensor categories for vertex operator superalgebra extensions. arXiv preprint arXiv:1705.05017.

\bibitem[CWX]{CWX}
Carpi, S., Weiner, M. and Xu, F., From vertex operator algebra modules to representations of conformal nets. To appear.

\bibitem[DGT19a]{DGT19a}
Damiolini, C., Gibney, A. and Tarasca, N., 2019. On factorization and vector bundles of conformal blocks from vertex algebras. arXiv preprint arXiv:1909.04683.

\bibitem[DGT19b]{DGT19b}
Damiolini, C., Gibney, A. and Tarasca, N., 2019. Vertex algebras of CohFT-type. arXiv preprint arXiv:1910.01658.

\bibitem[DL14]{DL14}
Dong, C. and Lin, X., 2014. Unitary vertex operator algebras. Journal of algebra, 397, pp.252-277.

\bibitem[DLM97]{DLM97}
Dong, C., Li, H. and Mason, G., 1997. Regularity of Rational Vertex Operator Algebras. Advances in Mathematics, 132(1), pp.148-166.



\bibitem[DR17]{DR17}
Dong, C. and Ren, L., 2017. Representations of the parafermion vertex operator algebras. Advances in Mathematics, 315, pp.88-101.

\bibitem[DW11a]{DW11a}
Dong, C. and Wang, Q., 2011. On $C_2$-cofiniteness of parafermion vertex operator algebras. Journal of Algebra, 328(1), pp.420-431.

\bibitem[DW11b]{DW11b}
Dong, C. and Wang, Q., 2011. Parafermion vertex operator algebras. Frontiers of Mathematics in China, 6(4), pp.567-579.

\bibitem[DX06]{DX06}
Dong, C. and Xu, F., 2006. Conformal nets associated with lattices and their orbifolds. Advances in Mathematics, 206(1), pp.279-306.

\bibitem[EGNO15]{EGNO15}
Etingof, P., Gelaki, S., Nikshych, D. and Ostrik, V., 2015. Tensor categories (Vol. 205). American Mathematical Soc..

\bibitem[FB04]{FB04}
Frenkel, E. and Ben-Zvi, D., 2004. Vertex algebras and algebraic curves (No. 88). American Mathematical Soc..

\bibitem[FHL93]{FHL93}
Frenkel, I., Huang, Y.Z. and Lepowsky, J., 1993. On axiomatic approaches to vertex operator algebras and modules (Vol. 494). American Mathematical Soc..

\bibitem[FK80]{FK80}
Frenkel, I.B. and Kac, V.G., 1980. Basic representations of affine Lie algebras and dual resonance models. Inventiones mathematicae, 62(1), pp.23-66.

\bibitem[FLM89]{FLM89}
Frenkel, I., Lepowsky, J. and Meurman, A., 1989. Vertex operator algebras and the Monster (Vol. 134). Academic press.

\bibitem[FRS89]{FRS89}
Fredenhagen, K., Rehren, K.H. and Schroer, B., 1989. Superselection sectors with braid group statistics and exchange algebras. Communications in Mathematical Physics, 125(2), pp.201-226.

\bibitem[FRS92]{FRS92}
Fredenhagen, K., Rehren, K.H. and Schroer, B., 1992. Superselection sectors with braid group statistics and exchange algebras II: Geometric aspects and conformal covariance. Reviews in Mathematical Physics, 4(spec01), pp.113-157.


\bibitem[Fin96]{Fin96}
Finkelberg, M., 1996. An equivalence of fusion categories. Geometric \& Functional Analysis GAFA, 6(2), pp.249-267.

\bibitem[GL96]{GL96}
Guido, D. and Longo, R., 1996. The conformal spin and statistics theorem. Communications in Mathematical Physics, 181(1), pp.11-35.



\bibitem[Gui19a]{Gui19a}
Gui, B., 2019. Unitarity of the modular tensor categories associated to unitary vertex operator algebras, I,  Comm. Math. Phys., 366(1), pp.333-396. 

\bibitem[Gui19b]{Gui19b}
Gui, B., 2019. Unitarity of the modular tensor categories associated to unitary vertex operator algebras, II. Communications in Mathematical Physics, 372(3), pp.893-950.

\bibitem[Gui19c]{Gui19c}
Gui, B., 2019. Energy bounds condition for intertwining operators of type $ B $, $ C $, and $ G_2 $ unitary affine vertex operator algebras. Trans. Amer. Math. Soc. 372 (2019), 7371-7424



\bibitem[Gui20a]{Gui20a}
Gui, B., 2020. Regular vertex operator subalgebras and compressions of intertwining operators. J. Algebra, (2020) 564. 32-48

\bibitem[Gui20b]{Gui20b}
Gui, B., 2020. Polynomial energy bounds for type $F_4$ WZW-models. Vol. 31, No. 12 (2020).

\bibitem[Gui21a]{Gui21a}
Gui, B., 2021. Categorical Extensions of Conformal Nets, Comm. Math. Phys., 383, 763-839 (2021).

\bibitem[Gui21b]{Gui21b}
Gui, B., 2021. Bisognano-Wichmann Property for Rigid Categorical Extensions and Non-local Extensions of Conformal Nets, Ann. Henri Poincaré, 22, 4017-4062 (2021).


\bibitem[Gui21c]{Gui21c}
Gui, B. (2021). On a Connes fusion approach to finite index extensions of conformal nets. arXiv preprint arXiv:2112.15396.

\bibitem[HKL15]{HKL15}
Huang, Y.Z., Kirillov, A. and Lepowsky, J., 2015. Braided tensor categories and extensions of vertex operator algebras. Communications in Mathematical Physics, 337(3), pp.1143-1159.

\bibitem[HK64]{HK64}
Haag, R. and Kastler, D., 1964. An algebraic approach to quantum field theory. Journal of Mathematical Physics, 5(7), pp.848-861.





\bibitem[HL95a]{HL95a}
Huang, Y.Z. and Lepowsky, J., 1995. A theory of tensor products for module categories for a vertex operator algebra, I. Selecta Mathematica, 1(4), p.699.

\bibitem[HL95b]{HL95b}
Huang, Y.Z. and Lepowsky, J., 1995. A theory of tensor products for module categories for a vertex operator algebra, II. Selecta Mathematica, 1(4), p.757.

\bibitem[HL95c]{HL95c}
Huang, Y.Z. and Lepowsky, J., 1995. A theory of tensor products for module categories for a vertex operator algebra, III. Journal of Pure and Applied Algebra, 100(1-3), pp.141-171.


\bibitem[Haag96]{Haag96}
Haag, R., 1996. Local quantum physics: Fields, particles, algebras. Springer Science \& Business Media.

\bibitem[Hen19]{Hen19}
Henriques, A., 2019. H. Loop groups and diffeomorphism groups of the circle as colimits. Communications in Mathematical Physics,  Volume 366, Issue 2, pp 537-565

\bibitem[Hua95]{Hua95}
Huang, Y.Z., 1995. A theory of tensor products for module categories for a vertex operator algebra, IV. Journal of Pure and Applied Algebra, 100(1-3), pp.173-216.

\bibitem[Hua05a]{Hua05a}
Huang, Y.Z., 2005. Differential equations and intertwining operators. Communications in Contemporary Mathematics, 7(03), pp.375-400.

\bibitem[Hua05b]{Hua05b}
Huang, Y.Z., 2005. Differential equations, duality and modular invariance. Communications in Contemporary Mathematics, 7(05), pp.649-706.

\bibitem[Hua08a]{Hua08a}
Huang, Y.Z., 2008. Vertex operator algebras and the Verlinde conjecture. Communications in Contemporary Mathematics, 10(01), pp.103-154.

\bibitem[Hua08b]{Hua08b}
Huang, Y.Z., 2008. Rigidity and modularity of vertex tensor categories. Communications in contemporary mathematics, 10(supp01), pp.871-911.

\bibitem[Jon83]{Jon83}
Jones, V.F., 1983. Index for subfactors. Inventiones mathematicae, 72(1), pp.1-25.

\bibitem[KL06]{KL06}
Kawahigashi, Y. and Longo, R., 2006. Local conformal nets arising from framed vertex operator algebras. Advances in Mathematics, 206(2), pp.729-751.

\bibitem[KLM01]{KLM01}
Kawahigashi, Y., Longo, R. and M\"uger, M., 2001. Multi-Interval Subfactors and Modularity of Representations in Conformal Field Theory. Communications in Mathematical Physics, 219(3), pp.631-669.

\bibitem[KM15]{KM15}
Krauel, M. and Miyamoto, M., 2015. A modular invariance property of multivariable trace functions for regular vertex operator algebras. Journal of Algebra, 444, pp.124-142.

\bibitem[KO02]{KO02}
Kirillov, A. and Ostrik, V., 2002. On a $q$-Analogue of the McKay Correspondence and the ADE Classification of $\wht{\fk{sl}}_2$ Conformal Field Theories. Advances in Mathematics, 2(171), pp.183-227.

\bibitem[KR08]{KR08}
Kong, L., \& Runkel, I. Algebraic structures in Euclidean and Minkowskian two-dimensional conformal field theory, conference proceedings `Non-commutative Structures in Mathematics and Physics' (Brussels, July 2008)


\bibitem[Kac90]{Kac90}
Kac, V.G., 1990. Infinite-dimensional Lie algebras. Cambridge university press.

\bibitem[Kaw15]{Kaw15}
Kawahigashi, Y., 2015. Conformal field theory, tensor categories and operator algebras. Journal of Physics A: Mathematical and Theoretical, 48(30), p.303001.

\bibitem[Kaw18]{Kaw18}
Kawahigashi, Y., 2018. Conformal field theory, vertex operator algebras and operator algebras. Proceedings of ICM-2018. arXiv preprint arXiv:1711.11349.




\bibitem[LR97]{LR97}
Longo, R. and Roberts, J.E., 1997. A theory of dimension. K-theory, 11(2), pp.103-159.

\bibitem[LX04]{LX04}
Longo, R. and Xu, F., 2004. Topological sectors and a dichotomy in conformal field theory. Communications in mathematical physics, 251(2), pp.321-364.

\bibitem[Loke04]{Loke94}
Loke, T.M., 1994. Operator algebras and conformal field theory of the discrete series representations of $\mathrm{Diff}(S^1)$ (Doctoral dissertation, University of Cambridge).

\bibitem[Lon89]{Lon89}
Longo, R., 1989. Index of subfactors and statistics of quantum fields. I. Communications in Mathematical Physics, 126(2), pp.217-247.

\bibitem[Lon90]{Lon90}
Longo, R., 1990. Index of subfactors and statistics of quantum fields. Communications in Mathematical Physics, 130(2), pp.285-309.

\bibitem[Lon03]{Lon03}
Longo, R., 2003. Conformal subnets and intermediate subfactors. Communications in mathematical physics, 237(1-2), pp.7-30.

\bibitem[Lon08]{Lon08}
Longo, R., 2008. Real Hilbert subspaces, modular theory, $\mathrm{SL}(2,\mbb R)$ and CFT. In: Von Neumann algebras in Sibiu, 33–91, Theta Ser. Adv. Math., 10, Theta, Bucharest.

\bibitem[MTW18]{MTW18}
Morinelli, V., Tanimoto, Y. and Weiner, M., 2018. Conformal covariance and the split property. Communications in Mathematical Physics, 357(1), pp.379-406.

\bibitem[McR19]{McR19}
McRae, R., 2019. On the tensor structure of modules for compact orbifold vertex operator algebras. Mathematische Zeitschrift, pp.1-44.



\bibitem[NT05]{NT05}
Nagatomo, K. and Tsuchiya, A., 2005. Conformal field theories associated to regular chiral vertex operator algebras, I: Theories over the projective line. Duke Mathematical Journal, 128(3), pp.393-471.


\bibitem[Nel59]{Nel59}
Nelson, E., 1959. Analytic vectors. Annals of Mathematics, pp.572-615.

\bibitem[Seg04]{Seg04}
G. Segal. The definition of conformal field theory. In Topology, geometry and quantum field theory, volume 308 of London Math. Soc. Lecture Note Ser., pages 421-577. Cambridge Univ. Press, Cambridge, 2004.

\bibitem[TUY89]{TUY89}
Tsuchiya, A., Ueno, K. and Yamada, Y., 1989. Conformal field theory on universal family of stable curves with gauge symmetries. In Integrable Sys Quantum Field Theory (pp. 459-566). Academic Press.

\bibitem[TL99]{TL99}
Toledano-Laredo, V., 1999. Integrating unitary representations of infinite-dimensional Lie groups. Journal of functional analysis, 161(2), pp.478-508.

\bibitem[TL04]{TL04}
Toledano-Laredo, V., 2004. Fusion of positive energy representations of $LSpin(2n)$. arXiv preprint math/0409044.

\bibitem[Ten17]{Ten17}
Tener, J.E., 2017. Construction of the unitary free fermion Segal CFT. Communications in Mathematical Physics, 355(2), pp.463-518.



\bibitem[Ten19a]{Ten19a}
Tener, J.E., 2019. Geometric realization of algebraic conformal field theories. Advances in Mathematics, 349, pp.488-563.


\bibitem[Ten19b]{Ten19b}
Tener, J.E., 2019. Representation theory in chiral conformal field theory: from fields to
observables. Selecta Math. (N.S.) (2019), 25:76

\bibitem[Ten24]{Ten24}
Tener, J. E. (2024). Fusion and positivity in chiral conformal field theory. Geometric and Functional Analysis, 34(4), 1226-1296.



\bibitem[Tur94]{Tur94}
Turaev, V.G., 1994. Quantum invariants of knots and 3-manifolds (Vol. 18). Walter de Gruyter GmbH \& Co KG.

\bibitem[Was98]{Was98}
Wassermann, A., 1998. Operator algebras and conformal field theory III. Fusion of positive energy representations of LSU (N) using bounded operators. Inventiones mathematicae, 133(3), pp.467-538.

\bibitem[Wei06]{Wei06}
Weiner, M. (2006). Conformal covariance and positivity of energy in charged sectors. Communications in mathematical physics, 265, 493-506.

\bibitem[Wit89]{Wit89}
Witten, E., 1989. Quantum field theory and the Jones polynomial. Communications in Mathematical Physics, 121(3), pp.351-399.

\bibitem[Xu00a]{Xu00a}
Xu, F., 2000. Jones-Wassermann subfactors for disconnected intervals. Communications in Contemporary Mathematics, 2(03), pp.307-347.

\bibitem[Xu00b]{Xu00b}
Xu, F., 2000. Algebraic coset conformal field theories. Communications in Mathematical Physics, 211(1), pp.1-43.

\bibitem[Xu01]{Xu01}
Xu, F., 2001. On a conjecture of Kac-Wakimoto. Publications of the Research Institute for Mathematical Sciences, 37(2), pp.165-190.

\bibitem[Yam04]{Yam04}
Yamagami, S., 2004. Frobenius duality in C*-tensor categories. Journal of Operator Theory, pp.3-20.

\bibitem[Zhu96]{Zhu96}
Zhu, Y., 1996. Modular invariance of characters of vertex operator algebras. Journal of the American Mathematical Society, 9(1), pp.237-302.


	
\end{thebibliography}
\end{document}